\documentclass[11pt]{article}
\usepackage{amsmath, latexsym, amsfonts, amssymb, amsthm, amscd}
\usepackage{mathtools} 
\usepackage[left=2cm, right=2cm, top=2.5cm, bottom=2.5cm]{geometry}
\usepackage{upquote} 
\usepackage{color}
\usepackage{setspace} 
\usepackage[labelfont=bf]{caption} 
\usepackage{stmaryrd} 
\usepackage{multicol} 
\usepackage[dvipsnames]{xcolor}

\usepackage{kpfonts} 
\usepackage{dsfont} 

\usepackage[english]{babel} 

\usepackage[utf8]{inputenc}
\usepackage[T1]{fontenc}

\usepackage{bm} 

\usepackage{enumitem} 

\usepackage{xcolor}
\definecolor{Maroon}{HTML}{ad2231}
\definecolor{webgreen}{HTML}{008000}
\usepackage{dsfont}
\usepackage{hyperref}
\hypersetup{colorlinks, breaklinks, urlcolor=Maroon, linkcolor=Maroon, citecolor=webgreen} 

\allowdisplaybreaks



\newtheorem{corollary}{Corollary}
\newtheorem{proposition}{Proposition}
\newtheorem{lemma}{Lemma}
\newtheorem{remark}{Remark}
\newtheorem{theorem}{Theorem}
\newtheorem{definition}{Definition}

\theoremstyle{definition}

\usepackage{todonotes}

\usepackage{chngcntr}
\usepackage{apptools}
\AtAppendix{\counterwithin{lemma}{section}}

\newcounter{Cond}
\newcounter{Cond2}

\newcommand{\R}{\mathbb{R}}

\newcommand{\bS}{\mathbb{S}}
\newcommand{\cN}{\mathcal{N}}
\newcommand{\cT}{\mathcal{T}}

\newcommand{\cEG}{\mathcal{EG}}

\newcommand{\bL}{\mathbb{L}}

\renewcommand{\D}{\mathbb{D}}

\begin{document}
\title{Convergence of trees with a given degree sequence and of their associated laminations.}
\author{Gabriel Berzunza Ojeda\footnote{ {\sc Department of Mathematical Sciences, University of Liverpool, United Kingdom.} E-mail: gabriel.berzunza-ojeda@liverpool.ac.uk }, \,\, Cecilia Holmgren\footnote{ {\sc Department of Mathematics, Uppsala University, Sweden.} E-mail: cecilia.holmgren@math.uu.se} \, \, and \, \, Paul Thévenin\footnote{ {Univ Angers, CNRS, LAREMA, SFR MATHSTIC, F-49000 Angers, France.} E-mail: paul.thevenin@univ-angers.fr} 
}
\date{ }
\maketitle

\vspace{0.1in}

\begin{abstract} 
In this paper, we study uniform rooted plane trees with given degree sequence. We show that, under some natural hypotheses on the degree sequences, these trees converge toward the so-called Inhomogeneous Continuum Random Tree after renormalisation. Our proof relies on the convergence of a modification of the well-known $\L$ukasiewicz path. We also give a unified treatment of the limit, as the number of vertices tends to infinity, of the fragmentation process derived by cutting down the edges of a tree with a given degree sequence, including its geometric representation by a lamination-valued process. The latter is a collection of nested laminations, which are compact subsets of the unit disk made of non-crossing chords. In particular, we prove an equivalence between planar Gromov-weak convergence of discrete trees and the convergence of their associated lamination-valued processes.
\end{abstract}




\section{Introduction}

In his seminal papers \cite{AldousII, AldousI, AldousIII}, Aldous introduced the so-called Brownian Continuum Random Tree (Brownian CRT) as the limit - after renormalisation of distances - of a uniform tree with $n$ vertices, and more generally, of critical size-conditioned Galton–Watson trees with finite offspring variance. The Brownian CRT has appeared since then as the limit of various random tree-like structures such as multi-type Galton-Watson trees \cite{Mie08} or unordered binary trees \cite{MarMie11}. Over the last decade, the study of scaling limits of large discrete random trees toward a random continuum tree has seen numerous developments and found extensive applications in the study of other random structures. These include, among others, random planar maps \cite{LeGall2013}, random dissections of regular polygons \cite{AldousT1994, AldousTT1994}, fragmentation and coalescent processes \cite{AldousPitman1998} or Erd\H{o}s-R\'enyi random graphs in the critical window \cite{AddarioT2012}. In this work, we investigate the scaling limit of trees with a given degree sequence, along with their associated laminations (sets of chords in the unit disk).

For a finite rooted plane tree $\mathbf{t}$, let $V(\mathbf{t})$ denote its number of vertices. If $V(\mathbf{t}) \geq 1$, then for $i \geq 0$, let $N_{i}(\mathbf{t})$ be the number of vertices in $\mathbf{t}$ having $i$ children (or out-degree $i$). The sequence $(N_{i}(\mathbf{t}), i \geq 0)$ is called the degree sequence of $\mathbf{t}$, and satisfies $V(\mathbf{t}) = \sum_{i \geq 0} N_{i}(\mathbf{t})= 1 + \sum_{i \geq 0} i N_{i}(\mathbf{t})$. Moreover, for $n \in \mathbb{N}$, a sequence $\mathbf{s}_{n} = (N_{i}^{n}, i \geq 0)$ of non-negative integers is the degree sequence of some finite rooted plane tree if and only if $\sum_{i \geq 0} N_{i}^{n}= 1 + \sum_{i \geq 0} i N_{i}^{n} < \infty$. Then, a random tree with given degree sequence (TGDS) $\mathbf{s}_{n}$ is a random variable whose law is uniform on the set $\mathbb{T}_{\mathbf{s}_{n}}$ of rooted plane trees with $V_{n} \coloneqq \sum_{i \geq 0} N_{i}^{n}$ vertices amongst which $N_{i}^{n}$ have $i$ offspring for every $i \geq 0$, and $E_{n} \coloneqq \sum_{i \geq 0} i N_{i}^{n}$ edges.

\subsection{Scaling limits of trees}

Scaling limits for trees with given degree sequences were first studied by Broutin \& Marckert \cite{Broutin2014}. Let $\mathbf{s}_{n}$ be a degree sequence and $\mathbf{t}_{n}$ be a random tree sampled uniformly at random in $\mathbb{T}_{\mathbf{s}_{n}}$. We see $\mathbf{t}_{n}$ as a rooted metric measure space $(\mathbf{t}_{n}, r_{n}^{\text{gr}}, \rho_{n}, \mu_{n})$, i.e.\ $\mathbf{t}_{n}$ is identified with its set of $V_{n}$ vertices, $r_{n}^{\text{gr}}$ is the graph-distance on $\mathbf{t}_{n}$ (that is, all edges have length $1$), $\rho_{n}$ is its root, and $\mu_{n}$  is the uniform measure on the set of vertices of $\mathbf{t}_{n}$. Consider the global variance term $\sigma_{n}^{2} = \sum_{i \geq 1} i (i-1) N_{i}^{n}$ for the degree sequence $\mathbf{s}_{n}$ (with $\sigma_n \geq 0$), and the maximum degree $\Delta_{n} = \max \{i \geq 0: N_{i}^{n} >0\}$ of any tree with degree sequence $\mathbf{s}_{n}$. Under technical assumptions on $\mathbf{s}_{n}$, in particular $\sigma_{n}^{2} \sim \sigma^2 V_{n}$ as $n \rightarrow \infty$ for some $\sigma \in (0, \infty)$ and $\lim_{n \rightarrow \infty} \sigma_{n}^{-1} \Delta_{n} = 0$, Broutin \& Marckert \cite{Broutin2014} showed the convergence in distribution,
\begin{eqnarray*}
\left(\mathbf{t}_{n},  \frac{\sigma_{n}}{V_{n}} r_{n}^{\text{gr}}, \rho_{n}, \mu_{n} \right) \xrightarrow[ ]{d} (\mathcal{T}_{\rm Br}, r_{\rm Br}, \rho_{\rm Br}, \mu_{\rm Br}), \hspace*{3mm} \text{as} \hspace*{2mm}  n \rightarrow \infty,
\end{eqnarray*}

\noindent for the so-called Gromov-Hausdorff-Prohorov topology, where $(\mathcal{T}_{\rm Br}, r_{\rm Br}, \rho_{\rm Br}, \mu_{\rm Br})$ is the Brownian CRT. In particular, $\mu_{\rm Br}$ is a probability measure supported on the leaves of $\mathcal{T}_{\rm Br}$.

Marzouk \cite{Cyril2019} proved a weaker convergence (in the sense of subtrees spanned by finitely many random vertices) by requiring only the condition $\lim_{n \rightarrow \infty} \sigma_{n}^{-1} \Delta_{n} = 0$. To be precise, fix $q \geq 1$ and let $u_{1}, \dots u_{q}$ be $q$ i.i.d.\ uniform random vertices of $\mathbf{t}_{n}$. The reduced tree $\mathbf{t}_{n}^{(q)}$ is obtained by keeping only the root of $\mathbf{t}_{n}$, these $q$ vertices, the branching points (if any), and then connecting by a single edge two of these vertices if one is the ancestor of the other in $\mathbf{t}_{n}$ and there is no other vertex of $\mathbf{t}_{n}^{(q)}$ inbetween. We define the length of an edge $e$ in $\mathbf{t}_{n}^{(q)}$ as the number of edges in $\mathbf{t}_{n}$ between the endpoints of $e$. In particular, the combinatorial structure of $\mathbf{t}_{n}^{(q)}$ is that of a rooted plane tree with at most $q$ leaves, so there are only finitely many possibilities, and thus there are a bounded number of edge lengths to record. Following Aldous \cite{AldousIII}, if $\mathbf{t}_{n}^{(q)}$ has $k$ vertices (and thus $k-1$ edges), one can formally regard $\mathbf{t}_{n}^{(q)}$ as a rooted plane tree with edge lengths, that is, a vector $(\hat{\mathbf{t}}, \ell_{1}, \dots, \ell_{k-1}) \in \mathbb{T}_{k} \times \mathbb{R}^{k-1}_{+}$, where $\mathbb{T}_{k}$ is the set of rooted plane trees with $k$ vertices, $\hat{\mathbf{t}}$ is the tree $\mathbf{t}_{n}^{(q)}$ without edge lengths and the $\ell_{i}$'s are the edge lengths. The space of trees with edge lengths is thus endowed with the natural product topology (i.e., $\mathbb{T}_{k}$ is equipped with the discrete topology and $\mathbb{R}^{k-1}_{+}$ with the usual one).
For $x_{1}, \dots, x_{q}$ i.i.d.\ random points of the Brownian CRT $\mathcal{T}_{\rm Br}$ sampled from its mass measure $\mu_{\rm Br}$, one can construct similarly a discrete tree with edge lengths $\mathcal{T}^{(q)}_{\rm Br}$; see \cite{AldousIII}. If $\lim_{n \rightarrow \infty} \sigma_{n}^{-1} \Delta_{n} = 0$, Marzouk \cite{Cyril2019} proved that, for every $q \geq 1$, one has the convergence in distribution,
\begin{eqnarray*}
\frac{\sigma_{n}}{V_{n}} \mathbf{t}_{n}^{(q)}\xrightarrow[ ]{d} \mathcal{T}^{(q)}_{\rm Br}, \hspace*{3mm} \text{as} \hspace*{2mm}  n \rightarrow \infty.
\end{eqnarray*}

In this work, we go one step further and, under the existence of at most countably many large degree vertices (see \ref{B2} and \ref{B}), we prove weak convergence of $\mathbf{t}_{n}$ toward the associated Inhomogeneous continuum random tree (Inhomogeneous CRT, which may be different from the Brownian CRT). The  Inhomogeneous CRT, introduced in \cite{AldousPitman1999} and \cite{Camarri2000}, arises as the scaling limit of another model of random trees called $\mathbf{p}$-trees (or birthday trees). The simplest description of the Inhomogeneous CRT is via a line-breaking construction based on a Poisson point process in the plane which can be found in \cite{AldousPitmanI2000,Camarri2000}. The spanning subtree description is set out in \cite{AldousPitman1999}, and its description via an exploration process is given in \cite{AldousMiermont2004}. An Inhomogeneous CRT $(\mathcal{T}_{\theta}, r_{\theta}, \rho_{\theta}, \mu_{\theta})$ is uniquely defined by a parameter set $\theta \coloneqq (\theta_{0}, \theta_{1}, \dots)$ such that
\begin{eqnarray*}
\theta_{1} \geq \theta_{2} \geq \cdots \geq 0, \hspace*{3mm} \theta_{0} \geq 0, \hspace*{3mm}  \sum_{i \geq 0} \theta_{i}^{2} = 1 \hspace*{3mm} \text{and either} \hspace*{3mm} \theta_{0} >0 \hspace*{3mm} \text{or} \hspace*{3mm} \sum_{i \geq 1} \theta_{i} = \infty; 
\end{eqnarray*}

\noindent see Figure \ref{fig:treelamin}, left for a simulation of an Inhomogeneous CRT. In the special case $\theta = (1,0,0, \dots)$, $\mathcal{T}_{\theta}$ is precisely the Brownian CRT. For $q \geq 1$, a discrete tree with edge lengths, $\mathcal{T}_\theta^{(q)}$, can be constructed from $q$ i.i.d.\ random points of $\mathcal{T}_{\theta}$ sampled from the mass measure $\mu_{\theta}$. Its distribution is described in \cite{AldousPitman1999, AldousPitmanI2000, Camarri2000}. 

\begin{theorem} \label{Theo2}
For $n \in \mathbb{N}$, let $\mathbf{s}_{n} = (N_{i}^{n}, i \geq 0)$ be a degree sequence. Let $(d^{n}(i), 1 \leq i \leq V_{n})$ denote the associated child sequence, obtained by writing $N_{0}^{n}$ zeros, $N_{1}^{n}$ ones, etc., and ordering the resulting sequence decreasingly. Assume that there exists a sequence $(b_{n}, n \geq 1)$ with $b_{n} \rightarrow \infty$ such that, as $n \rightarrow \infty$:
\begin{enumerate}[label=(\textbf{A.\arabic*})]
\item {\bf Size.} $V_{n} \rightarrow \infty$; \label{B1}
\item {\bf Hubs.} For all $i \geq 1$, the sequence $(d^{n}(i)/b_{n}, n \geq 1)$ converges to a limit $\theta_{i} \geq 0$; \label{B2}
\item {\bf Degree variance.}  $\frac{1}{b_{n}^{2}}\sum_{i \geq 0} (i-1)^{2}N_{i}^{n} \rightarrow 1$;
\label{B3}
\item $\theta_{0} = \sqrt{1- \sum_{i \geq 1} \theta_{i}^{2}}>0 \hspace*{3mm} \text{and} \hspace*{3mm} \sum_{i \geq 1}\theta_{i} < \infty$. \label{B}
\setcounter{Cond2}{\value{enumi}}
\end{enumerate}

\noindent Then, for all $q \geq 1$, we have that
\begin{eqnarray*}
\frac{b_{n}}{V_{n}} \mathbf{t}_{n}^{(q)} \xrightarrow[ ]{d} \mathcal{T}_{\theta}^{(q)}, \hspace*{3mm} \text{as} \hspace*{2mm}  n \rightarrow \infty,
\end{eqnarray*}
\noindent where $\mathcal{T}_{\theta}$ is an Inhomogeneous CRT with parameter set $\theta = (\theta_{0}, \theta_{1}, \dots)$ and the convergence is in distribution within the space of trees with edge lengths, equipped with the product topology.
\end{theorem}

By Fatou's lemma and \ref{B2}-\ref{B3}, we have that $\sum_{i \geq 1} \theta_{i}^{2} \leq 1$ and thus, $\theta_{0}$ in \ref{B} is well-defined. Observe also that the maximum degree $\Delta_{n}$ is equal to $d^{n}(1)$ and thus, \ref{B2} implies that $\Delta_{n}/b_{n} \rightarrow \theta_{1}$ as $n \rightarrow \infty$. In particular, if $\theta_{1} = 0$ in \ref{B2}, the hypotheses made in Theorem \ref{Theo2} correspond to the setting studied by Marzouk \cite{Cyril2019}. Indeed, by \ref{B3}, the global variance $\sigma_{n}^{2} = \sum_{i \geq 1} i (i-1) N_{i}^{n}$ of the degree sequence $\mathbf{s}_{n}$ satisfies that $\sigma^{2}_{n}/b_{n}^{2} \rightarrow 1$, as $n \rightarrow \infty$. On the other hand, Theorem \ref{Theo2} together with \cite[Theorem 5]{Greven2009} implies the convergence in distribution
\begin{eqnarray} \label{eq13}
\left(\mathbf{t}_{n},  \frac{b_{n}}{V_{n}} r_{n}^{\text{gr}}, \rho_{n}, \mu_{n} \right) \xrightarrow[ ]{d} (\mathcal{T}_{\theta}, r_{\theta}, \rho_{\theta}, \mu_{\theta}), \hspace*{3mm} \text{as} \hspace*{2mm}  n \rightarrow \infty,
\end{eqnarray}

\noindent for the so-called Gromov-weak topology (often cited as Gromov-Prohorov topology); see e.g.\ Section \ref{proofLaminationT} for background.

\begin{figure}[!htb]
\center
\includegraphics[scale=.3]{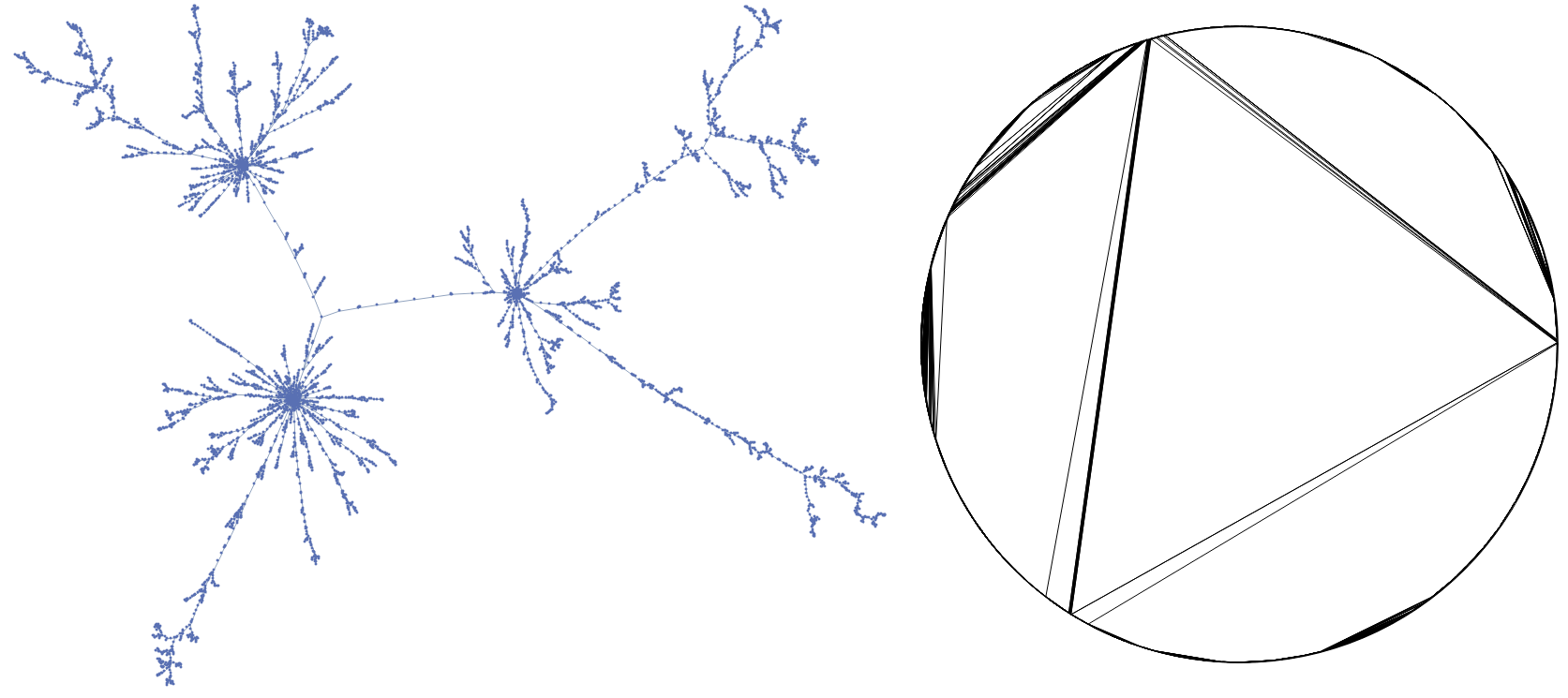}
\caption{Simulation of an Inhomogeneous CRT $\cT_\theta$ and its associated lamination $\bL(\cT_\theta)$, for $\theta = (1/\sqrt{7},2/\sqrt{7}, 1/\sqrt{7}, 1/\sqrt{7}, 0, \ldots)$.}
\label{fig:treelamin}
\end{figure}

In the regime of possibly countably many large degree vertices \ref{B2} and \ref{B}, Theorem \ref{Theo2} characterises the possible scaling limits of this model of random trees. It is then natural to wonder whether \eqref{eq13} can be reinforced to hold for the stronger Gromov-Hausdorff-Prohorov topology under the assumptions \ref{B1}-\ref{B}. To achieve this, one would need to prove the tightness of the sequence of discrete trees (see e.g. \cite[Equation 25]{AldousIII}), which requires precise estimates of the height of $\mathbf{t}_{n}$. However, as pointed out in \cite[Section 1.2]{Cyril2019}, there are cases where the maximal height of the tree can be much larger than $V_{n}/b_{n}$ and thus no general tightness result as in \cite{Broutin2014} holds.

Recently, Blanc-Renaudie \cite[Theorem 5 (a)-(b)]{Arthur2021} independently proved a general result (including the case $\sum_{i \geq 1} \theta_{i} = \infty$). Indeed, \cite[Theorem 6 (b)]{Arthur2021} shows that, under \cite[Assumption 2 (b)]{Arthur2021} (this assumption is similar to our conditions \ref{B1}-\ref{B3}), $\mathbf{t}_{n}$ converges, after normalization by $V_{n}/\sigma_{n}$, toward the Inhomogeneous CRT, for the Gromov-weak topology; see also the remarks after \cite[Theorem 6 (a)-(b)]{Arthur2021}. Nevertheless, our methods are completely different. In contrast to Blanc-Renaudie's recursive construction for TGDSs based on a modified Aldous-Broder algorithm, this work adopts a classical approach centred around the $\L$ukasiewicz path and its associated height process. The height process is a significant object of study in its own right and its analysis is central to the present work.

On the other hand, Blanc-Renaudie \cite[Theorem 7]{Arthur2021} implies, under additional tightness conditions (\cite[Assumption 7]{Arthur2021}), the convergence of $\mathbf{t}_{n}$ for the Gromov–Hausdorff–Prohorov topology. To be precise, Blanc-Renaudie provides under those conditions an upper bound for the height of $\mathbf{t}_{n}$, and uses similar estimates to control the Gromov-Hausdorff distance between $\mathbf{t}_{n}^{(q_{n})}$ and $\mathbf{t}_{n}$, for a well-chosen sequence $(q_{n}, n \geq 1)$. The height estimate in \cite{Arthur2021} is the ingredient we are missing to prove the tightness of the sequence of discrete trees (i.e., \cite[equation (25)]{AldousIII}), and thus, under some additional technical conditions, the convergence for the Gromov–Hausdorff–Prohorov topology.

A natural question would be to deduce an analogous result to Theorem \ref{Theo2} in the case $\sum_{i \geq 1} \theta_{i} = \infty$ by studying the asymptotic behaviour the height process of the discrete trees.
This broader setting, however, introduces additional technical complications. For instance, defining the height process of the Inhomogeneous CRT in this context remains an open problem.

Finally, we expect that similar results also hold for forests with given degree sequence. This has only been investigated under the assumption of no large degree vertices by Lei \cite{Lei2019} and Marzouk \cite{Cyril2019}. In particular, they view the forest as a single tree by attaching all the roots to an extra root vertex. In this framework, the limit is a different continuum tree that is encoded by a certain Brownian first-passage bridge. 

\subsection{Fragmentations and laminations} \label{SecFragLaMain}

Aldous, Evans and Pitman \cite{AldousPitman1998, EvansPitman1998, Pitman1999} initiated the study of fragmentation processes derived by deleting the edges of (random) trees one by one, uniformly at random. As time passes, the deletion of edges creates more and more connected components whose ordered sequence of sizes is called the fragmentation process of the tree. Aldous, Evans and Pitman studied the case of a uniform random tree with $n$ labelled vertices and showed that the associated fragmentation process, suitably rescaled, converges to the fragmentation process of the Brownian CRT as $n \rightarrow \infty$; see also \cite{Brou2016, MarckertWang2019}. This latter is connected to the standard additive coalescent via a deterministic time-change and it is constructed by cutting down the skeleton of the Brownian CRT in a Poisson manner. Aldous and Pitman \cite{AldousPitmanI2000} (see also \cite{EvansPitman1998}) established a similar result in the broader context of $\mathbf{p}$-trees. They showed that this fragmentation process converges after rescaling to the fragmentation process of the Inhomogeneous CRT. Recently, the first authors \cite{Berzunza2020} studied the case of critical Galton-Watson trees conditioned on having $n$ vertices, whose offspring distribution $\mu$ belongs to the domain of attraction of a stable law of index $\alpha \in (1,2]$ ($\alpha$-stable Galton-Watson trees). 
In this case, the limit is the fragmentation process of the so-called $\alpha$-stable L\'evy tree constructed by cutting down its skeleton in a Poisson manner.

It turns out that the fragmentation process of a tree can be coded by a non-decreasing process of subsets of the unit disk called laminations. A lamination is a closed subset of the closed unit disk $\bar{\mathbb{D}}$ made of the union of the unit circle $\mathbb{S}^{1}$ and a set of chords that do not intersect in the open unit disk $\mathbb{D}$. A face of a lamination $L$ is a connected component of the complement of $L$ in $\bar{\mathbb{D}}$. Laminations appear for instance in topology and hyperbolic geometry, see \cite{Bonahon2001} and references therein. We denote by $\bL(\bar{\mathbb{D}})$ the set of laminations of $\bar{\mathbb{D}}$ and equip it with the usual Hausdorff topology on the compact subsets of $\bar{\mathbb{D}}$. 

The idea of coding (random) trees by (random) laminations of $\bar{\mathbb{D}}$ goes back to Aldous \cite{AldousT1994, AldousTT1994} in his study of a uniform triangulation of a large polygon. Since then, laminations have appeared in different contexts, as limits of discrete structures \cite{Kor14, CK14, The19} or in the theory of random maps \cite{LGP08}. Roughly speaking, each chord of the lamination corresponds to an edge of the tree. Then, by adding chords one by one in the order in which the corresponding edges are removed, we code the fragmentation of the tree by a random process taking its values in $\bL(\bar{\mathbb{D}})$ (see Section \ref{sec:laminations} for a rigorous definition of this process). Furthermore, at any given time in the process, there is a one-to-one correspondence between faces of the lamination and connected components of the fragmented tree. In the case of $\alpha$-stable Galton-Watson trees, the third author proves in \cite{The19} the convergence of this lamination-valued process, toward a limiting process that can be constructed directly from the corresponding $\alpha$-stable L\'evy tree and encodes its fragmentation process.

A natural way of extending the previous investigations is to study the asymptotic behaviour of the fragmentation process and the lamination-valued process derived by cutting-down a rooted plane tree, and in particular a tree with given degree sequence. This is the second goal of this paper. To present the main result of this section, we need some notation and background that some readers may not be familiar with. We refer to Section \ref{sec:laminations} for proper definitions. For $\tau$ a finite rooted plane tree, let $(\bL_{t}(\tau), t \geq 0)$ be the lamination-valued process associated to the fragmentation of $\tau$ (that is, for all $t \geq 0$, $\bL_{t}(\tau)$ is obtained by removing the first $\lfloor t \rfloor \wedge (\zeta(\tau)-1)$ edges from $\tau$, where $\zeta(\tau)$ denotes the number of vertices of $\tau$, and drawing the corresponding chords in the disk); see Definition \ref{Def2}. For $I \subseteq \R_+$ a closed interval, let $\mathbf{D}(I, \mathbb{M})$ be the space of c\`adl\`ag functions (that is, right-continuous with left limits) from $I$ to a metric space $\mathbb{M}$. We equip $\mathbf{D}(I, \mathbb{M})$ with the $J_{1}$ Skorohod topology; see e.g.\  \cite[Section 5 in Chapter 3]{Ethier1986} or \cite[Chapter 3]{Billi1999} for details on this space. We denote by $\mathcal{T} = (\mathcal{T}, r, \rho, \mu)$ a plane continuum tree, that is, $(\mathcal{T}, r)$ is a metric space, $\rho$ is a distinguished element of $\mathcal{T}$ called the root, $\mu$ is a probability measure on the set of leaves of $\mathcal{T}$, and $\cT$ is endowed with a compatible total order; see Definition \ref{Def6}.
 As in the case of finite trees, it is possible to define from $\cT$ a lamination-valued process $(\bL_t(\cT), t \geq 0)$ obtained by cutting $\cT$ in a Poissonian way and associating to each cutpoint a chord in the disk; see Section \ref{sec:lamnoncompact}.

For $n \geq 1$, let $\tau_{n}$ be a (possibly random) rooted plane tree. We view it as a rooted metric measure space $(\tau_{n}, r_{n}^{\text{gr}}, \emptyset_{n}, \mu_{n})$, i.e., $\tau_{n}$ is identified as its set of vertices, $r_{n}^{\text{gr}}$ is the graph distance on $\tau_{n}$, $\emptyset_{n}$ is the root of $\tau_{n}$ and $\mu_{n}$ is the uniform measure on the set of vertices of $\tau_{n}$. In this work, whenever we consider a random rooted plane tree $\tau$, we always assume that the number of its vertices, $\zeta(\tau)$, is deterministic. Similarly, for a sequence of such trees  $(\tau_{n}, n\geq 1)$, we also assume that $\zeta(\tau_{n})$ is deterministic for all $n \geq 1$.

The following result, which in particular can be applied to trees with given degree sequence, states the equivalence of (a planar version of) the Gromov-weak convergence of a sequence of plane trees (see Definition \ref{def:planarGW}) and the convergence of its associated lamination-valued processes. Planar Gromov-weak convergence is weaker than the convergence used in Theorem \ref{Theo2} (i.e., of trees with edge lengths) because, in the former, branching points may merge at the limit.
\begin{theorem}
\label{thm:cvlamproc}
Let $(\tau_{n}, n \geq 1)$ be a sequence of random rooted plane trees, $\cT$ be a random plane continuum tree and $(a_{n}, n \geq 1)$ be a sequence of non-negative real numbers satisfying $a_{n} \rightarrow \infty$ and $\zeta(\tau_{n})/a_{n} \rightarrow \infty$, as $n \rightarrow \infty$. Then, the following assertions are equivalent:
\begin{enumerate}[label=(\textbf{C.\arabic*})]
\item 
\label{D1} 
$\displaystyle (\bL_{ta_{n}}(\tau_{n}), t \geq 0) \xrightarrow[ ]{d} (\bL_{t}(\mathcal{T}), t \geq 0), \hspace*{3mm} \text{as} \hspace*{2mm}  n \rightarrow \infty, \hspace*{2mm}  \text{in} \hspace*{2mm} \mathbf{D}(\mathbb{R}_{+}, \bL(\bar{\mathbb{D}}))$. 

\item $\displaystyle  \left(\tau_{n}, \frac{a_{n}}{\zeta(\tau_{n})} r_{n}^{\rm gr}, \emptyset_{n}, \mu_{n} \right) \rightarrow (\mathcal{T}, r, \rho, \mu)$, as $n \rightarrow \infty$, in the planar Gromov-weak sense. \label{D2}
\end{enumerate}
\end{theorem}

To prove Theorem \ref{thm:cvlamproc}, we develop a general approach based on the notion of reduced laminations that may be of independent interest. These reduced laminations are constructed considering reduced trees obtained by sampling only a finite number of vertices in the tree. In particular, the more vertices are sampled, the closer one is to the lamination-valued process associated with the entire tree.

Theorem \ref{thm:cvlamproc} and Theorem \ref{Theo7} (in Section \ref{ModyF}) together imply the following result regarding the lamination-valued process associated with the fragmentation of a tree with a given degree sequence $\mathbf{t}_{n}$. Specifically, Theorem \ref{Theo7} establishes the convergence of $(\mathbf{t}_{n},  (b_{n}/V_{n}) r_{n}^{\text{gr}}, \rho_{n}, \mu_{n})$ towards the Inhomogeneous CRT $\mathcal{T}_{\theta}$ in the planar Gromov-weak sense (Definition \ref{def:planarGW}). In contrast to Theorem \ref{Theo2}, Theorem \ref{Theo7} takes the order into account.

\begin{corollary} \label{thm:mainlamresultintro}
Suppose that $\mathbf{s}_{n}$ satisfies \ref{B1}-\ref{B}. Let $\mathcal{T}_{\theta}$ be an Inhomogeneous CRT with parameter set $\theta = (\theta_{0}, \theta_{1}, \dots)$. Then, jointly with the convergence of Theorem \ref{Theo2}, we have that
\begin{align*}
(\bL_{tb_{n}}(\mathbf{t}_{n}), t \geq 0) \xrightarrow[ ]{d} (\bL_{t}(\mathcal{T}_{\theta}), t \geq 0), \hspace*{3mm} \text{as} \hspace*{2mm}  n \rightarrow \infty, \hspace*{2mm} \text{in} \hspace*{2mm} \mathbf{D}(\mathbb{R}_{+}, \bL(\bar{\mathbb{D}})).
\end{align*}
\end{corollary}

In Figure \ref{fig:treelamin}, one can see a simulation of $\bL_\infty(\cT_\theta) \coloneqq \overline{\lim_{t \rightarrow \infty} \bL_t(\cT_\theta)}$, for a given parameter set $\theta$. Theorem \ref{thm:cvlamproc} (or Corollary \ref{thm:mainlamresultintro}) does not directly imply the convergence, after proper rescaling, of the fragmentation process associated with a tree $\mathbf{t}_{n}$ having a given degree sequence. However, this convergence can be established independently, without relying on the scaling limit of $\mathbf{t}_{n}$. The precise statement (Theorem \ref{Theo1}) is deferred to Section \ref{sec:fragmentation}. We highlight here that the limiting process, say $(\mathbf{F}(t), t \geq 0)$, of this rescaled fragmentation process corresponds exactly to Bertoin's construction \cite{Bertoin2000, Bertoin2001} using partitions of the unit interval induced by specific bridges with exchangeable increments (see Section \ref{sec:fragmentation}). Furthermore, $(\mathbf{F}(t), t \geq 0)$ coincides with the fragmentation process $(\mathbf{F}_{\cT_\theta}(t), t \geq 0)$ of the Inhomogeneous CRT $\cT_\theta$, with parameter set $\theta = (\theta_{0}, \theta_{1}, \dots)$, as shown in \cite{AldousPitmanI2000}. 

On the other hand, one can readily observe that the rescaled size of a component in the fragmentation process of $\mathbf{t}_{n}$ corresponds to the mass of the associated face in the lamination process, defined as $(2\pi)^{-1}$ times the fraction of its perimeter lying on the unit circle. Then, as a consequence of Theorem \ref{Theo1} below, we can relate the sequence of masses of the faces of $(\bL_{t}(\mathcal{T}_{\theta}), t \geq 0)$ with the fragmentation process $(\mathbf{F}(t), t \geq 0)$. For any lamination $L \in \mathbb{L}(\bar{\mathbb{D}})$, let $\mathbb{M}{\rm ass}[L]$ denote the sequence of the masses of its faces, sorted in non-increasing order. Consider also the infinite ordered set
\begin{eqnarray*}
\boldsymbol{\Delta} \coloneqq \Big \{ \mathbf{x} = (x_{1}, x_{2}, \dots): x_{1} \geq x_{2} \geq \cdots \geq 0 \hspace*{2mm} \text{and} \hspace*{2mm} \sum_{i=1}^{\infty} x_{i} < \infty \Big \}.
\end{eqnarray*}
\noindent We equip $\boldsymbol{\Delta}$ with the $\ell^{1}$-norm, $\Vert \mathbf{x} \Vert_{1} = \sum_{i=1}^{\infty} |x_{i}|$ for $\mathbf{x} \in \boldsymbol{\Delta}$. 

\begin{corollary}
\label{cor:masses}
Suppose that $\mathbf{s}_{n}$ satisfies \ref{B1}-\ref{B3} and 
\begin{enumerate}[label=(\textbf{A.\arabic*})]
\setcounter{enumi}{\value{Cond2}}
\item  either $\theta_{0} := \sqrt{1-\sum_{i \geq 1}\theta_{i}^{2}} > 0$ or $\sum_{i \geq 1}\theta_{i} = \infty$. \label{A4}
\end{enumerate}
\noindent Then,
\begin{align*}
(\mathbb{M}{\rm ass}[\bL_{tb_{n}}(\mathbf{t}_{n})], t \geq 0) \xrightarrow[ ]{d} \left(\mathbf{F}(t), t \geq 0 \right), \hspace*{3mm} \text{as} \hspace*{2mm}  n \rightarrow \infty, \hspace*{2mm} \text{in} \hspace*{2mm} \mathbf{D}(\mathbb{R}_{+}, \boldsymbol{\Delta}).
\end{align*}
\noindent Furthermore, the limiting process satisfies that, almost surely, $\left(\mathbf{F}_{\cT_\theta}(t), t \geq 0 \right) = (\mathbb{M}{\rm ass}[\bL_{t}(\mathcal{T}_{\theta})], t \geq 0)$.
\end{corollary}

Note that there is a priori no connection between the convergence of the fragmentation process in Corollary \ref{cor:masses} and the convergence of the lamination-valued process in Corollary \ref{thm:mainlamresultintro}. Roughly speaking, Corollary \ref{cor:masses} controls how small edges disappear at the limit, along the convergence of the lamination-valued processes in Corollary \ref{thm:mainlamresultintro}.

Let us finish with a remark on the assumptions that we make on the degree sequences. The hypotheses size \ref{B1}, hubs \ref{B2}, degree variance \ref{B3} and unbounded variation \ref{A4} are exactly those made in \cite{Angtuncio2020} to study the profile of a TGDS. They are necessary to apply the characterization and convergence results for exchangeable increments processes of \cite{Kallenberg1973} that are crucial to understand the shape of a TGDS via its $\L$ukasiewicz path. 

\subsection{Organization of the paper}

In Section \ref{sec:treesencoding}, we first recall the definition of rooted plane trees and their encoding by paths. In Section \ref{sec:convergence}, we prove Theorem \ref{Theo2} by studying the behaviour of a modified version of the $\L$ukasiewicz paths associated with trees with given degree sequence. Section \ref{sec:laminations} is devoted to the study of lamination-valued processes of rooted plane trees and plane continuum trees; we prove in particular Theorem \ref{thm:cvlamproc}. 
Finally, in Section \ref{sec:fragmentation}, we prove Theorem \ref{Theo1}, showing the convergence of the associated fragmentation processes, as well as Corollary \ref{cor:masses}.

\paragraph{Notation.}  
For $I=[a,b]$ for some $a<b$, or $I=[a,+\infty)$ for some $a \in \R_+$, for $f \in \mathbf{D}(I, \mathbb{R})$, we denote by $f(t)$ the value of $f$ at $t \in I$, by $f(t-)$ its left-hand limit at time $t$ (with the convention $f(a-) = f(a)$) and by $\Delta f(t) = f(t) - f(t-)$ the size of the jump (if any) at $t$. 

We write $\xrightarrow[ ]{d}$, $\xrightarrow[ ]{\mathbb{P}}$ and $\xrightarrow[ ]{a.s.}$ to denote convergence in distribution, probability and almost surely, respectively.

\section{Plane trees and their encoding paths}
\label{sec:treesencoding}

We provide here some background on finite rooted trees and recall how they can be coded by different integer-valued paths.

Following Neveu's formalism \cite{Neveu1986}, let $\mathbb{N} = \{1, 2, \dots \}$ be the set of positive integers and consider the set of labels $\mathbb{U} = \bigcup_{n \geq 0} \mathbb{N}^{n}$ (with the convention that $\mathbb{N}^{0} = \{ \emptyset \}$). An element $u \in \mathbb{U}$ is a sequence $u = (u_{1}, \dots, u_{n})$ of positive integers. If $v = (v_{1}, \dots, v_{m}) \in \mathbb{U}$, we let $uv = (u_{1}, \dots, u_{n}, v_{1}, \dots, v_{m}) \in \mathbb{U}$ be the concatenation of $u$ and $v$. By a slight abuse of notation, if $z \in \mathbb{N}$, we let $u z= (u_{1}, \dots, u_{n}, z)$. A rooted plane tree is a non-empty, finite subset $\tau \subset \mathbb{U}$, whose elements are called vertices, such that: (i) $\emptyset \in \tau$; (ii) if $v \in \tau$ and $v = u z$ for some $z \in \mathbb{N}$, then $u \in \tau$; (iii) if $u \in \tau$, then there exists an integer $k_{u} \geq 0$ such that $u i \in \tau$ if and only if $1 \leq i \leq k_{u}$. We view each vertex of $\tau$ as an individual of a population whose genealogical tree is $\tau$. The vertex $\emptyset$ is called the root of the tree. For every $u  = (u_{1}, \dots, u_{n}) \in \tau$, the vertex $pr(u) = (u_{1}, \dots, u_{n-1})$ is its parent, $k_{u}$ represents the number of children of $u$ (if $k_{u}=0$, then $u$ is called a leaf, otherwise, $u$ is called an internal vertex), and $|u| = n$ represents the length (or generation, or height) of $u$. We let $\chi_{u} \in \{1, \dots, k_{pr(u)} \}$ be the only index such that $u=pr(u) \chi_{u}$, which is the relative position of $u$ amongst its siblings. For two vertices $u, v \in \tau$, denote by $u \wedge v$ the first (highest) common ancestor of $u$ and $v$. The total progeny (or size) of $\tau$ will be denoted by $\zeta(\tau) = \text{Card}(\tau)$ (i.e., the number of vertices of $\tau$). In the following, by tree, we will always mean a finite rooted plane tree and we denote the set of all trees by $\mathbb{T}$. By a slight abuse, we consider a tree $\tau$ as a metric space, by drawing an edge of length $1$ between each non-root vertex $u$ and its parent $pr(u)$.

For $u, v \in \tau$, we denote by $\llbracket u, v \rrbracket$ the unique geodesic path between $u$ and $v$ in $\tau$, and $\llbracket u, v \llbracket = \llbracket u, v \rrbracket \backslash \{v\}$. In particular, we write $\llbracket \emptyset, u \llbracket$ for the ancestral line of $u$. For $u \in \tau$, let us denote by ${\rm L}(u)$ and ${\rm R}(u)$ respectively the number of vertices whose parent is a strict ancestor of $u$ and which lie strictly to the left, respectively to the right, of the ancestral line $\llbracket \emptyset, u \llbracket$. To be precise, for some $n \in \mathbb{N}$, let $u = (u_{1}, \dots, u_{n}) \in \tau$ where $u_{i} \in \mathbb{N}$, for $1 \leq i \leq n$. For $1 \leq i \leq n$, let $v_{(n)} = u$ and $v_{(n-j)} = pr(v_{(n+1-j)}) = (u_{1}, \dots, u_{n-j})$, for $1 \leq j  \leq n$ (with the convention $v_{(0)} = \emptyset$, the root of $\tau$.). So,  $\llbracket \emptyset, u \llbracket = \{v_{(i)}: 0 \leq i < n\}$,
\begin{align} \label{rightEq1}
{\rm L}(u) = \sum_{i=0}^{n-1} (u_{i+1}-1) \quad \text{and} \quad {\rm R}(u) = \sum_{i=0}^{n-1} (k_{v_{(i)}} - u_{i+1}).
\end{align}
\noindent Then, we set ${\rm LR}(u) = {\rm L}(u) + {\rm R}(u)$, the total number of individuals branching-off the ancestral line of $u$.

We will use three different orderings of the vertices of a tree $\tau \in \mathbb{T}$:
\begin{itemize}
\item[(i)]\textbf{Lexicographical ordering.} Given $v,w \in \tau$, we write $v \prec_{ \text{lex}} w$ if there exists $u \in \tau$ such that $v = u(v_{1}, \dots, v_{n})$, $w = u(w_{1}, \dots, w_{m})$ and $v_{1} < w_{1}$. 

\item[(ii)] \textbf{Reverse-lexicographical ordering.} Given $v,w \in \tau$, we write $v \prec_{\text{rev}} w$ if $w \prec_{\text{lex}} v$.

\item[(iii)] \textbf{Prim ordering.} Let $\textbf{edge}(\tau)$ be the set of edges of $\tau$ and consider a sequence of distinct and positive weights $\mathbf{w} = (w_{e}: e \in \textbf{edge}(\tau))$ (i.e., each edge $e$ of $\tau$ is marked with a different and positive weight $w_{e}$). Given two distinct vertices $u,v \in \tau$, we write $\{u, v\}$ for the edge connecting $u$ and $v$ in $\tau$ - if it exists. Let us describe the Prim order $\prec_{\text{prim}}$ of the vertices in $\tau$, that is, $\emptyset = u(0) \prec_{\text{prim}}  u(1) \prec_{\text{prim}} \dots  \prec_{\text{prim}} u(\zeta(\tau)-1)$. First set $u(0) = \emptyset$ and $D_{1} = \{u(0)\}$. Suppose that for some $1 \leq i \leq \zeta(\tau)-1$, the vertices $u(0), \dots, u(i-1)$ have been defined. We will use the notation $D_{i}$ for the set $\{u(0), \dots, u(i-1)\}$, for $0 \leq i \leq \zeta(\tau)-1$. Consider the minimum of the set of weights $\{w_{\{u,v\}}: u \in D_{i}, v \not \in D_{i}\}$ of the edges between a vertex of $D_{i}$ and another outside of $D_{i}$. Since all the weights are distinct, this minimum is reached at a unique edge $\{\tilde{u}, \tilde{v} \}$ where $\tilde{u} \in D_{i}$ and $\tilde{v} \not \in D_{i}$. Then set $u(i) = \tilde{v}$. This iterative procedure completely determines the Prim order $\prec_{\text{prim}}$.
\end{itemize}

\paragraph{$\L$ukasiewicz path, reverse-$\L$ukasiewicz path and Prim path.} Fix $\tau \in \mathbb{T}$, and for $\ast \in \{{\rm lex}, {\rm rev}, {\rm prim} \}$, associate to the ordering $\emptyset = u(0) \prec_{\ast} u(1) \prec_{\ast} \dots \prec_{\ast} u(\zeta(\tau)-1)$ of its vertices a path $W_{\tau}^{\ast}= (W_{\tau}^{\ast}(i), 0 \leq i \leq \zeta(\tau))$, by letting $W_{\tau}^{\ast}(0) = 0$ and for $0 \leq i \leq \zeta(\tau)-1$, $W_{\tau}^{\ast}(i+1) = W_{\tau}^{\ast}(i)  + k_{u(i)}-1$. Observe that  $W_{\tau}^{\ast}(i+1) - W_{\tau}^{\ast}(i) = k_{u(i)} -1 \geq -1$ for every $0 \leq i \leq \zeta(\tau)-1$, with equality if and only if $u(i)$ is a leaf of $\tau$. Note also that $W_{\tau}^{\ast}(i) \geq 0$ for every $0 \leq i \leq \zeta(\tau)-1$, and $W_{\tau}^{\ast}(\zeta(\tau)) = -1$. We shall think of such a path as the step function on $[0,\zeta(\tau)]$ given by $s \mapsto W_{\tau}^{\ast}(\lfloor s \rfloor)$. The path $W^{\text{lex}}_{\tau}$ is usually called the $\L$ukasiewicz path of $\tau$ and we will refer to $W^{\text{rev}}_{\tau}$ and $W^{\text{prim}}_{\tau}$ as the reverse-$\L$ukasiewicz path and the Prim path of $\tau$, respectively. 

In particular, if $\emptyset = u(0) \prec_{ \text{lex}}  u(1) \prec_{ \text{lex}}  \dots \prec_{ \text{lex}}  u(\zeta(\tau) -1)$ is the sequence of vertices of $\tau$ in lexicographical order, then 
\begin{eqnarray} \label{eq10}
{\rm R}(u(i)) = W^{\text{lex}}_{\tau}(i), \hspace*{4mm} \text{for} \hspace*{2mm} 0 \leq i \leq \zeta(\tau) -1. 
\end{eqnarray}

\noindent Similarly, if $\emptyset = u(0) \prec_{ \text{rev}}  u(1) \prec_{ \text{rev}}  \dots \prec_{ \text{rev}}  u(\zeta(\tau) -1)$ are listed in reverse-lexicographical order, then 
\begin{eqnarray} \label{eq19}
{\rm L}(u(i)) = W^{\text{rev}}_{\tau}(i), \hspace*{4mm} \text{for} \hspace*{2mm} 0 \leq i \leq \zeta(\tau) -1. 
\end{eqnarray}

\paragraph{Height process.} Let $\emptyset = u(0) \prec_{ \text{lex}}  u(1) \prec_{ \text{lex}}  \dots \prec_{ \text{lex}}  u(\zeta(\tau) -1)$ be the sequence of vertices of $\tau \in \mathbb{T}$ in lexicographical order. The height process $H_{\tau}= (H_{\tau}(i): 0 \leq i \leq \zeta(\tau))$ of $\tau$ is defined by letting $H_{\tau}(i) = |u(i)|$, for every $i \in \{0, \dots, \zeta(\tau) -1 \}$, and $H_{\tau}(\zeta(\tau)) = 0$. We sometimes think of $H_{\tau}$ as a continuous function on $[0,\zeta(\tau)]$, obtained by linear interpolation.

\paragraph{Contour function.} The contour function $C_{\tau}= (C_{\tau}(s), s \in [0, 2\zeta(\tau)])$ of $\tau \in \mathbb{T}$ is defined as follows. We see $\tau$ embedded in the oriented half-plane, with each edge having length one. We then imagine a particle exploring continuously all edges of $\tau$ from left to right at unit speed, going backwards when it reaches a leaf. For all $s \in [0,2\zeta(\tau)-2]$, let $C_{\tau}(s)$ be the distance from the particle to the root $\emptyset$ at time $s$. By convention, we set $C_{\tau}(s)=0$ for $s \in [2\zeta(\tau)-2, 2\zeta(\tau)]$. In particular, $C_{\tau}$ is continuous and $C_{\tau}(0)=C_{\tau}(2\zeta(\tau))=0$.

\section{Convergence of the trees with given degree sequences}
\label{sec:convergence}

In this section, we prove Theorem \ref{Theo2} which states the convergence of trees with a given degree sequence toward the Inhomogeneous CRT. In Section \ref{ExploICRT}, we first recall the definition of the exploration process that encodes the Inhomogeneous CRT. In Section \ref{ModyF}, we introduce the discrete version of the above exploration process, the so-called modified $\L$ukasiewicz path, which encodes a TGDS. We then prove that this modified $\L$ukasiewicz path, suitably rescaled, converges to the exploration process of the Inhomogeneous CRT, which implies Theorem \ref{Theo2}. Finally, in Section \ref{ConvGHP}, we consider a specific case in which we can prove the convergence of the trees for the Gromov-Hausdorff-Prohorov topology.

\subsection{The exploration process of the Inhomogeneous CRT} \label{ExploICRT}
Let us start with some definitions. A bridge with exchangeable increments (abridged EI bridge) is a continuous-time stochastic process $X^{\rm bg} = (X^{\rm bg}(t), t \in [0,1])$ with paths in $\mathbf{D}([0,1], \mathbb{R})$, of the form
\begin{eqnarray*}
X^{\rm bg}(t) = \theta_{0} B^{\rm bg}(t) + \sum_{i=1}^{\infty} \theta_{i} (\mathbf{1}_{\{ U_{i} \leq t\}} -t), \hspace*{4mm} t \in [0,1], 
\end{eqnarray*}

\noindent where $B^{\rm bg}= (B^{\rm bg}(t), t \in [0,1])$ is a Brownian bridge on $[0,1]$, $(U_{i}, i \geq 1)$ are i.i.d.\ random variables with uniform law on $[0,1]$ independent of $B^{\rm bg}$, and $\theta_{0} \in \mathbb{R}_{+}$, $\theta_{1} \geq \theta_{2} \geq \cdots \geq 0$ are constants such that $\sum_{i \geq 0} \theta_{i}^{2} < \infty$. We say that $X^{\rm bg}$ is an EI bridge with parameters $(\theta_{i}, i \geq 0)$; see e.g. \cite[Theorem 16.21]{Kall2005}. 

The so-called Vervaat transform (or Vervaat excursion) was introduced by Tak\'acs \cite{Tak1967} and used by Vervaat \cite{Veervat1979} (see also \cite[Section 3]{Bertoin2001} or \cite[Section 4.2]{DuquesneC2003}) to change a bridge-type process with paths in $\mathbf{D}([0,1], \mathbb{R})$ into an excursion-type process (i.e., a non-negative process that is equal to $0$ at times $0$ and $1$). More precisely, let $X = (X(t), t \in [0,1])$ be a stochastic process with paths in $\mathbf{D}([0,1], \mathbb{R})$ such that $X(0)=0$. 
\noindent Then, the Vervaat transform of $X$ is the stochastic process ${\rm Ver}_X := ({\rm Ver}_{X}(t), t \in [0,1])$ with paths in $\mathbf{D}([0,1], \mathbb{R})$, defined by 
\begin{eqnarray} \label{VerTran}
{\rm Ver}_{X}(t) = \left\{ \begin{array}{lcl}
              X(t+\rho_{X}) -  \inf_{0 \leq s \leq 1} X(s) & \mbox{  if } & t+\rho_{X} \leq 1,\\
            X(t+\rho_{X}-1) + X(1)-  \inf_{0 \leq s \leq 1} X(s)&  \mbox{  if } & t+\rho_{X} \geq 1 \\
              \end{array}
    \right. 
\end{eqnarray}
\noindent where $\rho_{X} = \inf \left\{ u \in [0,1]: \min(X(u-), X(u)) = \inf_{0 \leq s \leq 1} X(s) \right\}$.

Through this manuscript, unless otherwise specified, we always consider an EI bridge $X^{\rm bg}$ with parameters $(\theta_{i}, i \geq 0)$ such that $\sum_{i \geq 0}\theta_{i}^{2}=1$ and  \ref{A4} is satisfied (i.e.\ either  $\theta_{0} >0$ or $\sum_{i \geq 1} \theta_{i} = \infty$). This is a necessary and sufficient condition for $X^{\rm bg}$ to have paths of infinite variation. More importantly, by \cite{Knight1996} or \cite[Proof of Lemma 6]{Bertoin2001}, it is well-known that under this condition $X^{\rm bg}$ almost surely achieves its infimum at a unique time $\rho_{X^{\rm bg}}$ and continuously. Then, we let $X^{\rm exc} = (X^{\rm exc}(t), t \in [0,1])$ be the excursion-type process associated to $X^{\rm bg}$ via its Vervaat transform. Note that $X^{\rm exc}(0) = X^{\rm exc}(1)=0$ and $X^{\rm exc}(t) >0$ for all $t \in (0,1)$.

The excursion process $X^{\rm exc}$  is not necessarily continuous. However, following \cite[Section 2]{AldousMiermont2004}, one can also associate to $X^{\rm bg}$ a continuous excursion process $H^{\rm exc}$. For $i \geq 1$ such that $\theta_{i} >0$, write $t_{i} = \{U_{i} - \rho_{X^{\rm bg}} \}$ (the fractional part of $U_{i} - \rho_{X^{\rm bg}}$) for the location of the jump with size $\theta_{i}$ in $X^{\rm exc}$. For each $i \geq 1$ such that $\theta_{i} >0$, write $T_{i} = \inf \{t \in (t_{i},1]:  X^{\rm exc}(t) = X^{\rm exc}(t_{i}-) \}$, which exists since the process $X^{\rm exc}$ has no negative jumps and goes back to $0$ at time $1$. In particular, all $t_{i}$'s and $T_{i}$'s are distinct almost surely. For $i \geq 1$ such that $\theta_{i} >0$, let $R_{i} = (R_{i}(u), u \in [0,1])$ be the process defined by
\begin{eqnarray} \label{eq14}
R_{i}(u) = \left\{ \begin{array}{lcl}
              \inf_{t_{i} \leq s \leq u} X^{\rm exc}(s) - X^{\rm exc}(t_{i}-)  & \mbox{  if } & u \in [t_{i}, T_{i}],\\
            0  & & \mbox{otherwise}. \\
              \end{array}
    \right. 
\end{eqnarray}

\noindent If $\theta_{i} = 0$ then let $R_{i}$ be the null process on $[0,1]$. Define the process $H^{\rm exc} = (H^{\rm exc}(u), u \in [0,1])$ by
\begin{eqnarray} \label{ExcH}
H^{\rm exc}(u) = X^{\rm exc}(u) - \sum_{i \geq 1} R_{i}(u), \hspace*{4mm} \text{for} \hspace*{2mm} u \in [0,1]. 
\end{eqnarray}

As explained in \cite[Section 2]{AldousMiermont2004}, $H^{\rm exc}$ is a well-defined continuous excursion-type process. Furthermore, Aldous, Miermont and Pitman \cite[Theorem 1]{AldousMiermont2004} also showed that, when $\theta_{0}>0$ and $\sum_{i  \geq 1} \theta_{i} < \infty$, $H^{\rm exc}$ is the height process (up to a scaling factor) of an Inhomogeneous CRT with parameter set $\theta = (\theta_{0}, \theta_{1}, \dots)$. 

\subsection{The modified $\L$ukasiewicz path} \label{ModyF}

For $n \in \mathbb{N}$, let $\mathbf{s}_{n}$ be a degree sequence satisfying \ref{B1}-\ref{B3} and \ref{A4}. We then sample a tree $\mathbf{t}_{n}$ uniformly at random from $\mathbb{T}_{\mathbf{s}_{n}}$ and let $W_{n}^{\rm lex} = (W_{n}^{\rm lex}( V_{n} u), u \in [0,1])$ denote its time-rescaled $\L$ukasiewicz path. Recall that $V_n$ denotes the number of vertices in $\mathbf{t}_{n}$. Let $(b_{n}, n \geq 1)$ be a sequence satisfying \ref{B2}, and $X$ be an EI bridge with parameters $(\theta_{i}, i \geq 0)$ as in \ref{B2}.

\begin{theorem} \label{Theo5}
Suppose that $\mathbf{s}_{n}$ satisfies \ref{B1}-\ref{B3} and \ref{A4}. Then,  
\begin{eqnarray*}
(b_{n}^{-1} W_{n}^{\rm lex}( V_{n} u), u \in [0,1]) \xrightarrow[ ]{d} (X^{\rm exc}(u), u \in [0,1]), \hspace*{3mm} \text{as} \hspace*{2mm}  n \rightarrow \infty, \hspace*{2mm} \text{in} \hspace*{2mm} \mathbf{D}([0,1], \mathbb{R}).
\end{eqnarray*}
\end{theorem}

\begin{proof}
Let $\pi_{n}$ be a uniform permutation of $\{1, \dots, V_{n}\}$. Recall that $(d^{n}(i), 1 \leq i \leq V_{n})$ denotes the associated child sequence of $\mathbf{s}_{n}$. Define the process $Y_{n} = (Y_{n}(u), u \in [0,1])$ by letting
\begin{align*}
Y_{n}(u) = \sum_{i =1}^{ \lfloor V_{n}u \rfloor} (d^{n}(\pi_{n}(i)) -1), \quad \text{for} \quad u \in [0,1].
\end{align*}
\noindent For $1 \leq i \leq V_n$, define $\Delta Y_n(i/V_n) = Y_n(i/V_n)-Y_n((i-1)/V_n) = d^{n}(\pi_{n}(i)) -1$. Note that the increments $(\Delta Y_{n}(1/V_{n}),  \Delta Y_{n}(2/V_{n}) \dots, \Delta Y_{n}(1))$ are exchangeable (i.e., their distribution is invariant under permutation). Note also that $\Delta Y_{n}(i/V_{n}) = d^{n}(\pi_{n}(i)) -1 \geq -1$, for every $1 \leq i \leq V_{n}$, and furthermore, $Y_{n}(1) = -1$. Define
\begin{align*}
\rho_{n} = \min \left\{ i \in \{1, \dots, V_{n}\}: Y_{n}(i/V_{n}) = \min_{1 \leq j \leq V_{n}} Y_{n}(j/V_{n}) \right \}.
\end{align*} 
and let $Y_{n}^{\rm exc} = (Y_{n}^{\rm exc}(u), u \in [0,1])$ be the Vervaat transform of $Y_{n}$; see (\ref{VerTran}). Note that $Y_{n}^{\rm exc}(u) \geq 0$, for $u \in [0,1)$, and $Y_{n}^{\rm exc}(1) = -1$. 

Under our assumptions \ref{B1}-\ref{B3} and \ref{A4}, it follows from \cite[Theorem 2.2]{Kallenberg1973} that 
\begin{eqnarray*}
(b_{n}^{-1} Y_{n}(u), u \in [0,1]) \xrightarrow[ ]{d} (X^{\rm bg}(u), u \in [0,1]), \hspace*{3mm} \text{as} \hspace*{2mm}  n \rightarrow \infty, \hspace*{2mm} \text{in} \hspace*{2mm} \mathbf{D}([0,1], \mathbb{R}).
\end{eqnarray*}
\noindent Since the process $X^{\rm bg}$ achieves its infimum in a unique time and continuously, \cite[Lemma 3]{Bertoin2001} implies that $b_{n}^{-1} Y_{n}^{\rm exc}$ converges to $X^{\rm exc}$, the Vervaat transform of $X^{\rm bg}$. Finally, our claim follows by noticing (see e.g. \cite[Proof of Lemma 7]{Broutin2014}) that $Y_{n}^{\rm exc}$ has the same distribution as the $\L$ukasiewicz path $W_{n}^{\rm lex}$. 
\end{proof}

We now describe a modification of $W^{\text{lex}}_{n}$ from which we construct the discrete analogue of $H^{\rm exc}$. Here we make the convention $\inf \, \varnothing =1$ and $\sup \, \varnothing = 0$. For $\delta >0$, define
\begin{align*}
t_{0}^{n,\delta} \coloneqq 0 \quad \text{and} \quad t_{i}^{n,\delta} = \inf\{  u \in (t^{n,\delta}_{i-1}, 1]: |\Delta W^{\text{lex}}_{n}(V_{n}u)| > \delta b_{n}\}, \quad \text{for} \quad i \geq 1,
\end{align*}
\noindent and 
\begin{align*}
T_{i}^{n, \delta} = \inf\{ u \in (t_{i}^{n, \delta},1]: W^{\text{lex}}_{n}(V_{n}u) - W^{\text{lex}}_{n}(V_{n}t_{i}^{n,\delta}-)=-1 \}, \quad \text{for} \quad i \geq 1.
\end{align*}
\noindent Set $I^{ \delta}_{n} = \sup \{ i \geq 0: 0<t_{i}^{n,\delta}<1\}$. In particular, $I^{ \delta}_{n} \leq V_n$.
Note that $t_{i}^{n,\delta} = T_{i}^{n, \delta} =1$ for $i >I^{ \delta}_{n}$. If $I^{ \delta}_{n} >0$,  then for $i \geq 1$, we let $R_{i}^{n, \delta} = (R_{i}^{n, \delta}(u), u \in [0,1])$ be the process given by
\begin{eqnarray*}
R_{i}^{n, \delta}(u) = \left\{ \begin{array}{lcl}
             \inf_{t_{i}^{n, \delta}\leq s \leq u} W^{\text{lex}}_{n}(V_{n}s)-W^{\text{lex}}_{n}(V_{n}t_{i}^{n, \delta}-)  & \mbox{  if }  & u \in [t_{i}^{n, \delta}, T_{i}^{n, \delta}], \\
                           0 &  \mbox{otherwise}. &  \\
              \end{array}
    \right. 
\end{eqnarray*}

\noindent For $i > I^{ \delta}_{n}$ (even when $I^{ \delta}_{n} =0)$, we set $R_{i}^{n, \delta}(u)=0$, for all $u \in [0,1]$. Then, for $\delta >0$, we define the modified $\L$ukasiewicz path $G_{n}^{\delta} = (G_{n}^{\delta}(V_{n}u), u \in [0,1])$ of $\mathbf{t}_{n}$ by letting
\begin{eqnarray} \label{ModLu}
G_{n}^{\delta}(V_{n}u) = W^{\rm lex}_{n}(V_{n}u) - \sum_{i\geq 1}R_{i}^{n, \delta}(u).
\end{eqnarray}
\noindent Note that the sum on the right hand side is always finite for every $u \in [0,1]$, since the number of nonzero summands is $I^{ \delta}_{n}$. 

Let $(H^{\rm exc}(u), u \in [0,1])$ be the continuous excursion process associated to an EI bridge $(X^{\rm exc}(u), u \in [0,1])$ with parameters $(\theta_{i}, i \geq 0)$ defined in (\ref{ExcH}). Recall also the process $R_{i}=(R_{i}(u), u \in [0,1])$ defined in (\ref{eq14}). For $\delta >0$, we define the process $G^{\delta} = (G^{\delta}(u), u \in [0,1])$ by
\begin{eqnarray} \label{ModHeight}
G^{\delta}(u) = H^{\rm exc} + \sum_{i \geq 1} R_{i}(u)\mathbf{1}_{\{\theta_{i} \leq \delta \}} =  X^{\rm exc}(u) - \sum_{i \geq 1} R_{i}(u)\mathbf{1}_{\{\theta_{i} > \delta \}}.
\end{eqnarray}

\begin{theorem} \label{Theo6}
Suppose that $\mathbf{s}_{n}$ satisfies \ref{B1}-\ref{B}. Then, for $\delta \notin \{\theta_{i}: i \geq 1\}$, 
\begin{eqnarray*}
\left( b_{n}^{-1} G_{n}^{\delta}( V_{n} u), b_{n}^{-1} W^{{\rm lex}}_{n}(V_{n}u): u \in [0,1] \right) \xrightarrow[ ]{d} (G^{\delta}(u), X^{\rm exc}(u): u \in [0,1]), \hspace*{3mm} \text{as} \hspace*{2mm}  n \rightarrow \infty, 
\end{eqnarray*}
\noindent in $\mathbf{D}([0,1], \mathbb{R}) \times \mathbf{D}([0,1], \mathbb{R})$. Moreover, as $\delta \downarrow 0$ (still with the assumption that $\delta \notin \{\theta_{i}: i \geq 1\}$), $(G^{\delta}(u), u \in [0,1])$ converges uniformly to $H^{\rm exc}$.   
\end{theorem}

Observe that \ref{B} is included in \ref{A4}, so that in particular Theorem \ref{Theo5} holds under the assumptions of Theorem \ref{Theo6}.

\begin{proof}[Proof of Theorem \ref{Theo6}]
Recall that, if $\theta_{i} >0$, then $t_{i}$ denotes the location of the jump with size $\theta_{i}$ of $X^{{\rm exc}}$, i.e. $t_{i} = \{U_{i} - \rho_{X^{\rm bg}} \}$ (the fractional part of $U_{i} - \rho_{X^{\rm bg}}$). If $\theta_{i} =0$, then we let $t_{i} = 1$. For $\delta >0$, note that $I^{\delta} := \sup \{ i \geq 1: \theta_{i} >\delta \}< \infty$. By convenience, if $\theta_{1} \leq \delta$ (and thus, $\theta_{i} \leq \delta$, for all $i \geq 1$), then we set $I^{\delta} = 0$.  On the other hand, for $\delta >0$, define 
\begin{align*}
t_{0}^{\delta} \coloneqq 0 \quad \text{and} \quad t_{i}^{\delta} = \inf\{  u \in (t^{\delta}_{i-1}, 1]: |\Delta X^{\rm exc}(u)| > \delta\}, \quad \text{for} \quad i \geq 1.
\end{align*}
\noindent Here we also make the convention $\inf \varnothing =1$. If $I^{\delta} > 0$, then the sequence of times $(t_{i}^{\delta}, 1 \leq i \leq I^{\delta})$ is the sequence $(t_{i}, 1 \leq i \leq I^{\delta})$ in increasing order. For each $i \geq 1$, write $T_{i}^{\delta} = \inf \{t \in (t_{i}^{\delta},1]:  X^{\rm exc}(t) = X^{\rm exc}(t_{i}^{\delta}-) \}$ and let $R_{i}^{\delta} = (R_{i}^{\delta}(u), u \in [0,1])$ be the process given by
\begin{eqnarray*} 
R_{i}^{\delta}(u) = \left\{ \begin{array}{lcl}
              \inf_{t_{i}^{\delta} \leq s \leq u} X^{\rm exc}(s) - X^{\rm exc}(t_{i}^{\delta}-)  & \mbox{  if } & u \in [t_{i}^{\delta}, T_{i}^{\delta}],\\
            0  & & \mbox{otherwise}. \\
              \end{array}
    \right. 
\end{eqnarray*}
\noindent Note that, if $t_{i}^{\delta} = 1$, then $R_{i}^{\delta}(u)=0$ for all $u \in [0,1]$ and, if $I^{\delta}=0$, then $R_{i}^{\delta}(u)=0$, for all $u \in [0,1]$ and $i \geq 1$. Then, it should be clear that
\begin{eqnarray} \label{IdenExt}
G^{\delta}(u) =  X^{\rm exc}(u) - \sum_{i \geq 1} R_{i}^{\delta}(u), \quad \text{for} \quad u \in [0,1].
\end{eqnarray}

Note that, under our assumptions \ref{B1}-\ref{B}, the result in Theorem \ref{Theo5} holds, i.e.\ the rescaled $\L$ukasiewicz path $b_{n}^{-1} W_{n}^{\rm lex}$ converges towards $X^{\rm exc}$. By the Skorokhod representation theorem, there exists a probability space on which this convergence holds almost surely. Let us work on this space from now on. We claim that, for $\delta \notin \{\theta_{i}: i \geq 1\}$, 
\begin{itemize}
\item[(i)] for all $i \geq 1$, $\lim_{n \rightarrow \infty} t_{i}^{n, \delta} = t_{i}^{\delta}$ a.s., 
\item[(ii)] if $I^{\delta} > 0$, then for all $1 \leq i \leq I^{\delta}$, $\lim_{n \rightarrow \infty} T_{i}^{n, \delta} = T_{i}^{\delta}$ a.s., and
\item[(iii)] for all $i \geq 1$, $\left( b_{n}^{-1} R^{n, \delta}_{i}(u), u \in [0,1] \right) \rightarrow (R_{i}^{\delta}(u), u \in [0,1])$ a.s., as $n \rightarrow \infty$, in $\mathbf{D}([0,1], \mathbb{R})$.
\end{itemize}

Note that (i) follows from Theorem \ref{Theo5} and \cite[Proposition VI.2.7]{jacod2003}. Indeed, \cite[Proposition VI.2.7]{jacod2003} consider paths in $\mathbf{D}(\mathbb{R}_{+}, \mathbb{R})$. However, a close inspection of its proof shows that the result holds for paths in $\mathbf{D}([0,1], \mathbb{R})$ with minor modifications. 

Next, we prove (ii). It follows from Theorem \ref{Theo5} that $\liminf_{n \rightarrow \infty} T_{i}^{n,\delta} \geq T_{i}^{\delta}$. Suppose that $t^{\prime} = \limsup_{n \rightarrow \infty} T_{i}^{n, \delta} > T_{i}^{\delta}$ and up to extraction suppose that $t^{\prime}$ is actually the limit of $(T_{i}^{n, \delta}, n \geq 1)$. Since (i) and \cite[Proposition VI.2.1]{jacod2003} imply that $\lim_{n \rightarrow \infty} b_{n}^{-1}W^{\text{lex}}_{n}(V_{n}T_{i}^{n, \delta}) = X^{{\rm exc}}(t_{i}^{\delta}-)$, we would find that $t^{\prime} > T_{i}^{\delta}$ with $X^{\rm exc}(t^{\prime}) = X^{{\rm exc}}(t_{i}^{\delta}-)$ and $X^{\rm exc}(s) \geq X^{{\rm exc}}(t_{i}^{\delta}-)$, for $s \in [T_{i}^{\delta}, t^{\prime}]$. This shows that $X^{{\rm exc}}(t_{i}^{\delta}-)$ is a local minimum of $X^{{\rm exc}}$, attained at time $T_{i}^{\delta}$, which is a.s.\ impossible by \ref{B} and \cite[Lemma 1]{AldousMiermont2004}. Therefore, $\lim_{n \rightarrow \infty} T_{i}^{n, \delta} = T_{i}^{\delta}$ which proves (ii). 

We now prove (iii). Suppose first that $I^{\delta}=0$. Then $t_{i}^{\delta}=1$ and $R_{i}^{\delta}(u)=0$, for all $i \geq 1$ and $u \in [0,1]$. Moreover, for $i \geq 1$, a.s., 
\begin{align*}
\sup_{0 \leq u \leq 1} |b_{n}^{-1} R_{i}^{n, \delta}(u)| \leq b_{n}^{-1}\max(|\Delta W^{\text{lex}}_{n}(V_{n}t_{i}^{n, \delta})|,1) \rightarrow |\Delta X^{\text{exc}}(1)|=0, \quad \text{as} \quad  n \rightarrow \infty,
\end{align*}
\noindent by \cite[Proposition VI.2.1]{jacod2003}. Thus, (iii) holds whenever $I^{\delta}=0$. Suppose now that $I^{\delta}>0$. By Theorem \ref{Theo5}, there exists a sequence of
strictly increasing, continuous mappings $(\lambda_{n}, n \geq 1)$ of $[0,1]$ onto itself such that, a.s., 
\begin{eqnarray} \label{Nequ1}
\lim_{n \rightarrow \infty} \sup_{u \in [0,1]}| \lambda_{n}(u) - u| = 0 \hspace*{2mm} \text{and} \hspace*{2mm} \lim_{n \rightarrow \infty} \sup_{u \in [0,1]} \left| b_{n}^{-1} W^{{\rm lex}}_{n}(V_{n}\lambda_{n}(u)) -  X^{\rm exc}(u)  \right| = 0. 
\end{eqnarray}
\noindent For $1 \leq i \leq I^{\delta}$, we shall prove that, a.s., 
\begin{eqnarray} \label{Nequ2}
 \lim_{n \rightarrow \infty} \sup_{u \in [0,1]} \left| b_{n}^{-1} R_{i}^{n, \delta}(\lambda_{n}(u)) -  R^{\delta}_{i}(u)  \right| = 0. 
\end{eqnarray}
\noindent Suppose by contradiction that (\ref{Nequ2}) does not hold. Then there exists $\varepsilon > 0$ and a sequence $(z_{n}, n \geq 1) \in [0,1]^\mathbb{N}$ such that
\begin{align} \label{Nequ3}
\left| b_{n}^{-1} R_{i}^{n, \delta}(\lambda_{n}(z_{n})) -  R^{\delta}_{i}(z_{n})  \right| > \varepsilon.
\end{align}
\noindent for infinitely many indices $n$. Extracting a subsequence if necessary, let us assume that this holds for all $n$ and that either $z_{n} \uparrow z$ or $z_{n} \downarrow z$ for some $z \in [0,1]$, as $n \rightarrow \infty$. If $z \notin [t_{i}^{\delta}, T_{i}^{\delta}]$, then, by (i) and (ii), for $n$ large enough, $\lambda_{n}(z_{n}), z_{n} \notin [t_{i}^{\delta}, T_{i}^{\delta}] \cup  [t_{i}^{n, \delta}, T_{i}^{n, \delta}]$, which leads to a contradiction of (\ref{Nequ3}). 

Now, assume that $z_{n} \uparrow t_{i}^{\delta}$, as $n \rightarrow \infty$ and $z_{n} \neq t_{i}^{\delta}$ for all $n \geq 1$. It follows also from (\ref{Nequ1}) that $\lambda_{n}(z_{n}) \rightarrow t_{i}^{\delta}$ and $b_{n}^{-1} W^{{\rm lex}}_{n}(V_{n}\lambda_{n}(z_{n})) \rightarrow  X^{\rm exc}(t_{i}^{\delta}-)$, a.s., as $n \rightarrow \infty$. On the other hand, if  $z_{n} \uparrow T_{i}^{\delta}$, as $n \rightarrow \infty$, then by (\ref{Nequ1}) we have that $\lambda_{n}(z_{n}) \rightarrow T_{i}^{\delta}$ and $b_{n}^{-1} W^{{\rm lex}}_{n}(V_{n}\lambda_{n}(z_{n})) \rightarrow  X^{\rm exc}(T_{i}^{\delta})$ a.s., as $n \rightarrow \infty$ (note that $X^{\rm exc}(T_{i}^{\delta}) = X^{\rm exc}(t_{i}^{\delta}-)$). Thus, in both cases, $b_{n}^{-1} R_{i}^{n, \delta}(\lambda_{n}(z_{n})) \rightarrow 0$ and $R^{\delta}_{i}(z_{n}) \rightarrow 0$ a.s., as $n \rightarrow \infty$, which contradicts (\ref{Nequ3}). If $z_{n} \uparrow t_{i}^{\delta}$, as $n \rightarrow \infty$ and $z_{n} = t_{i}^{\delta}$, for $n$ large enough, then it follows from the definition of $t_{i}^{n,\delta}$, \cite[Proposition VI.2.1 and Proposition VI.2.7]{jacod2003} that 
$\lambda_{n}(z_{n}) = \lambda_{n}(t_{i}^{\delta}) = t_{i}^{n,\delta}$, for $n$ large enough. In particular, for $n$ large enough, $R_{i}^{n, \delta}(\lambda_{n}(z_{n})) = \Delta W^{{\rm lex}}_{n}(V_{n}t_{i}^{n,\delta})$, and thus, by \cite[Proposition VI.2.7]{jacod2003}, $b_{n}^{-1}R_{i}^{n, \delta}(\lambda_{n}(z_{n})) \rightarrow \Delta X^{\rm exc}(t_{i}^{\delta}) =  R_{i}^{\delta}(t_{i}^{\delta})$ a.s., as $n \rightarrow \infty$, which contradicts (\ref{Nequ3}). 

Suppose finally that $z_{n} \uparrow z \in (t_{i}^{\delta}, T_{i}^{\delta})$, then, by (i) and (ii), for $n$ large enough, $\lambda_{n}(z_{n}) \in [t_{i}^{n, \delta}, T_{i}^{n, \delta}]$ and $z_{n} \in [t_{i}^{\delta}, T_{i}^{\delta}]$. Then,
\begin{align} \label{Nequ4}
\left| b_{n}^{-1} R_{i}^{n, \delta}(\lambda_{n}(z_{n})) -  R^{\delta}_{i}(z_{n})  \right| \leq \left| \inf_{t_{i}^{n, \delta} \leq u \leq \lambda_{n}(z_{n})}  b_{n}^{-1}W^{\text{lex}}_{n}(V_{n}u) - \inf_{t_{i}^{\delta} \leq u \leq z_{n}}  X^{\text{exc}}(u)  \right|+ \left| b_{n}^{-1}W^{\text{lex}}_{n}(V_{n}t_{i}^{n,\delta}-) -  X^{\text{exc}}(t_{i}^{\delta}-)  \right|. 
\end{align}

\noindent On the other hand, note that
\begin{align} \label{Nequ5}
& \left| \inf_{t_{i}^{n, \delta} \leq u \leq \lambda_{n}(z_{n})}  b_{n}^{-1}W^{\text{lex}}_{n}(V_{n}u) - \inf_{t_{i}^{\delta} \leq u \leq z_{n}}  X^{\text{exc}}(u)  \right| \nonumber \\
& \quad \quad  \leq  \left| \inf_{\lambda_{n}^{-1}(t_{i}^{n, \delta}) \leq u \leq z_{n}}  b_{n}^{-1}W^{\text{lex}}_{n}(V_{n}\lambda_{n}(u)) - \inf_{\lambda_{n}^{-1}(t_{i}^{n,\delta}) \leq u \leq z_{n}}  X^{\text{exc}}(u)  \right| + \left| \inf_{\lambda_{n}^{-1}(t_{i}^{n,\delta}) \leq u \leq z_{n}}  X^{\text{exc}}(u)  -  \inf_{t_{i}^{\delta} \leq u \leq z_{n}}  X^{\text{exc}}(u)  \right| \nonumber \\
& \quad \quad  \leq  \sup_{\lambda_{n}^{-1}(t_{i}^{n,\delta}) \leq u \leq z_{n}} \left|  b_{n}^{-1}W^{\text{lex}}_{n}(V_{n}\lambda_{n}(u)) -  X^{\text{exc}}(u)  \right| + \left| \inf_{\lambda_{n}^{-1}(t_{i}^{n,\delta}) \leq u \leq z_{n}}  X^{\text{exc}}(u)  -  \inf_{t_{i}^{\delta} \leq u \leq z_{n}}  X^{\text{exc}}(u)  \right|, 
\end{align}
\noindent where $\lambda_{n}^{-1}$ denotes the inverse of $\lambda_{n}$. It follows from (\ref{Nequ1}), (\ref{Nequ4}), (\ref{Nequ5}), (i) and \cite[Proposition VI.2.1]{jacod2003} that $b_{n}^{-1} R_{i}^{n, \delta}(\lambda_{n}(z_{n})) -  R^{\delta}_{i}(z_{n})$ converges to $0$, a.s., as $n \rightarrow \infty$, and this yields to a contradiction of (\ref{Nequ3}). 

Similarly, if $z_{n} \downarrow z \in [0,1]$, as $n \rightarrow \infty$, then one can check that (\ref{Nequ3}) does not hold by considering the cases $z= t_{i}^{\delta}$, $z= T_{i}^{\delta}$, $z \in (t_{i}^{\delta}, T_{i}^{\delta})$ and $z \not \in [t_{i}^{\delta}, T_{i}^{\delta}]$, separately. We leave the details to the reader.

Therefore, we have showed (\ref{Nequ2}), that is, (iii), for every $1 \leq i \leq I^{\delta}$. For $i > I^{\delta}$, the proof of (iii) follows as in the case $I^{\delta} =0$. 

Finally, we have all the ingredients to prove Theorem \ref{Theo6}. Again, by Theorem \ref{Theo5}, we know that there exists a sequence of strictly increasing, continuous mappings $(\lambda_{n}, n \geq 1)$ of $[0,1]$ onto itself such that (\ref{Nequ1}) holds. Note that we have actually shown during the proof of (iii) that (\ref{Nequ2}) holds for every $i \geq 1$. Hence the triangle inequality shows that for every fixed $q \geq 1$ that, a.s., 
\begin{eqnarray} \label{Nequ6}
 \lim_{n \rightarrow \infty} \sup_{u \in [0,1]} \left| b_{n}^{-1} W^{{\rm lex}}_{n}(V_{n}\lambda_{n}(u)) - \sum_{i=1}^{q} b_{n}^{-1} R_{i}^{n, \delta}(\lambda_{n}(u)) -  X^{\rm exc}(u)  + \sum_{i = 1}^{q} R_{i}^{\delta}(u) \right| = 0. 
\end{eqnarray}
\noindent Note that there exists $q$ such that, for $n$ large enough, $t_{i}^{n, \delta} = t_{i}^{\delta} =1$, for $i >q$, which implies that $R_{i}^{n, \delta}(\lambda_{n}(u)) = R_{i}^{\delta}(u) = 0$, for all $u \in [0,1]$ and $i >q$. Therefore, the first part of Theorem \ref{Theo6} follows from (\ref{IdenExt}) and (\ref{Nequ6}). The second part of Theorem \ref{Theo6} follows by Dini's theorem (see for e.g., \cite[Theorem 7.13]{Rudin1976}).
\end{proof}

Let $(H_{n}(V_{n}u), u \in [0,1])$ be the time-rescaled height process associated to $\mathbf{t}_{n}$. 

\begin{theorem} \label{Theo7}
Suppose that $\mathbf{s}_{n}$ satisfies assumptions \ref{B1}-\ref{B}. Fix $q \geq 1$, let $U_{1}, \dots, U_{q}$ be i.i.d.\ uniform random variables in $[0, 1]$ independent of the rest and denote by $0 = U_{(0)} < U_{(1)} < \cdots < U_{(q)}$ their order statistics. Then,
\begin{eqnarray*}
\frac{\theta_{0}^{2}}{2} \frac{b_{n}}{V_{n}}\left( H_{n}(V_{n}U_{(i)}), \inf_{U_{(i-1)} \leq u \leq U_{(i)}} H_{n}(V_{n} u) \right)_{1 \leq i \leq q} \xrightarrow[ ]{d} \left( H^{\rm exc}(U_{(i)}), \inf_{U_{(i-1)} \leq u \leq U_{(i)}} H^{\rm exc}(u) \right)_{1 \leq i \leq q}, \hspace*{3mm} \text{as} \hspace*{2mm}  n \rightarrow \infty,
\end{eqnarray*}
\noindent holds jointly with the convergence in distribution of the $\L$ukasiewicz path from Theorem \ref{Theo6}.
\end{theorem}

Assuming Theorem \ref{Theo7}, we have now all the ingredients to establish Theorem \ref{Theo2}. 

\begin{proof}[Proof of Theorem \ref{Theo2}]
It follows from Theorem \ref{Theo7} and \cite[Theorem 1]{AldousMiermont2004}.  Indeed, \cite[Theorem 1]{AldousMiermont2004} shows that, under our assumptions, the exploration process of the Inhomogeneous CRT $\mathcal{T}_{\theta}$ with parameter set $\theta = (\theta_{0}, \theta_{1}, \dots)$ is distributed as $\frac{2}{\theta_{0}^{2}}H^{\rm exc}$. Further details can be found in \cite[Section 4]{AldousMiermont2004}.
\end{proof}

Note that, in Theorem \ref{Theo7}, we assumed that the degree sequence satisfies \ref{B}. But clearly, \ref{B} is included in \ref{A4}.

As a preparation for the proof of Theorem \ref{Theo7}, we deduce the following property for degree sequences that satisfy \ref{B1}-\ref{B3} and \ref{A4}. 

\begin{lemma} \label{lemma1}
Suppose that $\mathbf{s}_{n}$ satisfies assumptions \ref{B1}-\ref{B3} and \ref{A4}. Then, $\lim_{n \rightarrow \infty} V_{n}/b_{n} = \infty$. 
\end{lemma}

\begin{proof}
First, suppose that $\sum_{i \geq 1}\theta_{i} = \infty$ in \ref{A4}. Then, our claim follows from \ref{B1}-\ref{B2} since $V_{n} \geq 1 + \sum_{i=1}^{k} d^{n}(i)$, for every fixed $1 \leq k \leq V_{n}$. Suppose now that $\theta_{0} >0$ in \ref{A4} and that our claim does not hold, i.e., there exists a constant $C>0$ such that, along a subsequence, $V_{n}< C b_{n}$. Then, for any $\varepsilon>0$,
\begin{align} \label{eq:less}
b_{n}^{-2} \sum_{i \geq 0} (i-1)^{2} N_{i}^{n} 
&\leq b_{n}^{-2} N_0^n + b_{n}^{-2} \sum_{i \geq 0} i^{2} N_{i}^{n}
\leq C b_n^{-1} + b_{n}^{-2} \sum_{i \leq \lfloor \varepsilon b_{n} \rfloor} i^{2} N_{i}^{n} + b_{n}^{-2} \sum_{i > \lfloor \varepsilon b_{n} \rfloor} i^{2} N_{i}^{n} \nonumber \\
& \leq C b_n^{-1} + \varepsilon C + b_{n}^{-2}\sum_{i > \lfloor \varepsilon b_{n} \rfloor} i^{2} N_{i}^{n},
\end{align}
\noindent where we have used that $V_{n} = \sum_{i \geq 0} N_{i} = 1 + \sum_{i \geq 0}i N_{i}$. Since we have assumed  $V_n < C b_n$, we necessarily have $\sum_{i > \lfloor \varepsilon b_{n} \rfloor} N_i^n \leq \lceil C/\varepsilon \rceil$ (by considering the number of children of these vertices). Hence,
\begin{align*}
b_{n}^{-2} \sum_{i > \lfloor \varepsilon b_{n} \rfloor} i^{2} N_{i}^{n} \leq b_{n}^{-2} \sum_{i \leq \lceil C/\epsilon \rceil} (d^n(i))^2.
\end{align*}
By \ref{B2}, we get that
\begin{align*}
\limsup_{n \rightarrow \infty} b_{n}^{-2} \sum_{i > \lfloor \varepsilon b_{n} \rfloor} i^{2} N_{i}^{n} \leq \sum_{i \leq \lceil C/\epsilon \rceil} \theta_{i}^{2} \leq \sum_{i \geq 1} \theta_i^2.
\end{align*}
\noindent Taking $\varepsilon< \theta_{0}^{2}/C$ provides a contradiction between \eqref{eq:less} and \ref{B3}.
\end{proof}

\begin{proof}[Proof of Theorem \ref{Theo7}]
We will follow an argument similar to that of the proof  of \cite[Theorem 2.4]{Cyril2019}. Let $U$ be uniformly distributed on $[0, 1]$ independently of $\mathbf{t}_{n}$ and let $u_{n}$ be the  $\lfloor V_{n} U \rfloor+1$-st vertex of $\mathbf{t}_{n}$ in lexicographical order, so that it has the uniform distribution in $\mathbf{t}_{n}$. Recall that $H_{n}( \lfloor V_{n} U \rfloor ) = |u_{n}|$ and that ${\rm R}(u_{n})$ denotes the number of individuals branching-off strictly to the right of the ancestral line $\llbracket \emptyset, u_{n}\llbracket$ in $\mathbf{t}_{n}$; see (\ref{rightEq1}). Recall that, by (\ref{eq10}), ${\rm R}(u_{n}) = W^{\rm lex}_{n}(V_{n} U)$. Fix $\delta >0$ and henceforth consider $n$ large enough such that $\delta b_{n} >1$. Let ${\rm R}^{\delta}(u_{n})$ be the number of individuals branching-off strictly to the right of the ancestral line $\llbracket \emptyset, u_{n}\llbracket$ in $\mathbf{t}_{n}$, except for vertices that are children of vertices with degree larger than $\delta b_{n}+1$, i.e., 
\begin{align} \label{eqExtra1}
{\rm R}^{\delta}(u_{n}) = G_{n}^{\delta}(\lfloor V_{n} U \rfloor),
\end{align}
\noindent where $G_{n}^{\delta}$ is the modified $\L$ukasiewicz defined in (\ref{ModLu}). Define also
\begin{align*}
\sigma_{n,\delta}^{2}  = \sum_{k = 0}^{ \lfloor \delta b_{n}  +1\rfloor} k (k-1) N_{k}^{n}.
\end{align*}
\noindent We claim that, for every $\varepsilon >0$, 
\begin{align} \label{eq11}
\lim_{\delta \downarrow 0}\limsup_{n \rightarrow \infty}\mathbb{P} \left( \left|  {\rm R}^{\delta}(u_{n}) - \frac{\sigma_{n,\delta}^{2}}{2 E_{n}} |u_{n}| \right| > \varepsilon  \sigma_{n,\delta} \right)= 0. 
\end{align}
\noindent We will prove \eqref{eq11} later; assume it for now. Note that \ref{B2}-\ref{B} imply that
\begin{align} \label{eqExtra2}
\lim_{n \rightarrow \infty}b_{n}^{-2} \sigma_{n,\delta}^{2}  = 1 - \sum_{i \geq 1} \theta_{i}^{2} \mathbf{1}_{\{ \theta_{i} > \delta \}} \quad \text{and} \quad \lim_{\delta \downarrow 0} \left( 1 - \sum_{i \geq 1} \theta_{i}^{2} \mathbf{1}_{\{ \theta_{i} > \delta \}} \right)  = \theta_{0}^{2}.
\end{align} 
\noindent Recall the definition of $G^\delta$ in \eqref{ModHeight} and that by Theorem \ref{Theo6}, $G^{\delta}$ converges uniformly to $H^{\rm exc}$, as $\delta \downarrow 0$. Therefore, a combination of (\ref{eqExtra1}), (\ref{eq11}), (\ref{eqExtra2}), the union bound and Theorem \ref{Theo6} shows that, jointly with the convergence of the $\L$ukasiewicz path from Theorem \ref{Theo5}, 
\begin{align*}
\frac{\theta^{2}_{0}}{2} \frac{b_{n}}{V_{n}}\left( H_{n}(V_{n}U_{(i)}) \right)_{1 \leq i \leq q} \xrightarrow[ ]{d} \left( H^{\rm exc}(U_{(i)}) \right)_{1 \leq i \leq q}, \hspace*{3mm} \text{as} \hspace*{2mm}  n \rightarrow \infty,
\end{align*}
\noindent that is, the convergence of the first marginal in Theorem \ref{Theo7}. 

Now we proceed to prove (\ref{eq11}). Recall the notation ${\rm LR}(u_{n}) = {\rm L}(u_{n}) + {\rm R}(u_{n})$ for the total number individuals branching-off the ancestral line of $u_{n}$ introduced after (\ref{rightEq1}) in Section \ref{sec:treesencoding}. Fix $\varepsilon, \eta >0$. By \cite[Propositions 4.4 and 4.5]{Cyril2019} and our assumptions, we can and will consider $K >0$ such that
\begin{align*}
\mathbb{P} \left( {\rm LR}(u_{n}) \leq K b_{n} \hspace*{1mm} \text{and}  \hspace*{1mm} |u_{n}| \leq K b_{n}^{-1}V_{n} \right ) \geq 1-\eta.
\end{align*} 
\noindent for every $n$ large enough. Then, 
\begin{align*}
\mathbb{P} \left( \left|  {\rm R}^{\delta}(u_{n}) - \frac{\sigma_{n,\delta}^{2}}{2 E_{n}} |u_{n}| \right| > \varepsilon  \sigma_{n,\delta} \right) \leq \eta + \mathbb{P} \left( \left| {\rm R}^{\delta}(u_{n}) - \frac{\sigma_{n,\delta}^{2}}{2 E_{n}} |u_{n}| \right| > \varepsilon \sigma_{n,\delta}, \, \, {\rm LR}(u_{n}) \leq K b_{n} \, \, \text{and} \, \, |u_{n}| \leq K b_{n}^{-1}V_{n}\right).
\end{align*}

Suppose that an urn contains initially $kN_{k}^{n}$ balls labelled $k$ for every $k \geq 1$, so $E_{n}$ balls in total. Let us pick balls repeatedly one after the other without replacement. For every $1 \leq i \leq E_{n}$, we denote the label of the $i$-th ball by $\xi_{n}(i)$. Conditionally on $(\xi_{n}(i), 1 \leq i \leq E_{n})$, let us sample independent random variables $(\chi_{n}(i), 1 \leq i \leq E_{n})$ such that each $\chi_{n}(i)$ is uniformly distributed in $\{1, \dots, \xi_{n}(i)\}$. The spinal decomposition obtained in \cite[Lemma 4.1]{Cyril2019} (see also \cite[Section 3]{Broutin2014}) with $q = 1$ shows that the probability of  the event that $|u_{n}| = h$ and that for all $0 \leq i < h$, the ancestor of $u_{n}$ at generation $i$ has $k_{i}$ offspring and its $j_{i}$-th is the ancestor of $u_{n}$ at generation $i+1$ is bounded by
\begin{eqnarray*}
\frac{1+\sum_{1 \leq i \leq h}(k_{i}-1)}{V_{n}} \cdot \mathbb{P} \left( \bigcap_{i \leq h} \{ (\xi_{n}(i), \chi_{n}(i)) = (k_{i}, j_{i}) \} \right).
\end{eqnarray*}

\noindent Note that, in the previous event, ${\rm R}^{\delta}(u_{n}) = \sum_{i \leq h} (k_{i}-j_{i})\mathbf{1}_{\{ k_{i} \leq  \delta b_{n}+1 \}}$. To see this, compare ${\rm R}^{\delta}(u_{n})$ with ${\rm R}(u_{n})$ defined in (\ref{rightEq1}). Note also that, in the previous event, ${\rm LR}(u_{n}) = \sum_{1 \leq i \leq h}(k_{i}-1) \leq Kb_{n}$. Thus, by decomposing according to the height of $u_{n}$ and taking the worst case (i.e.\ the union bound), we then obtain that, 
\begin{eqnarray} \label{eqExtra6}
\mathbb{P} \left( \left|  {\rm R}^{\delta}(u_{n}) - \frac{\sigma_{n,\delta}^{2}}{2 E_{n}} |u_{n}| \right| > \varepsilon  \sigma_{n,\delta} \right) \leq \eta + \left(\frac{K}{b_{n}} +K^{2}\right) \sup_{h \leq K V_{n} / b_{n}} \mathbb{P} \left( \left| \sum_{i \leq h} (\xi_{n}(i) - \chi_{n}(i)) \mathbf{1}_{\{ \xi_{n}(i) \leq \delta b_{n} +1  \}} - \frac{\sigma_{n, \delta}^{2}}{2 E_{n}} h \right| > \varepsilon \sigma_{n, \delta}\right).
\end{eqnarray}

\noindent Hence to prove (\ref{eq11}), it is enough to check the following two estimates:
\begin{itemize}
\item[(i)] $\displaystyle \mathbb{E} \left [\sum_{i \leq h} (\xi_{n}(i) - \chi_{n}(i)) \mathbf{1}_{\{ \xi_{n}(i) \leq \delta b_{n} +1  \}} \right] = \frac{\sigma_{n, \delta}^{2}}{2 E_{n}} h$, and

\item[(ii)]  $\displaystyle {\rm Var} \left (\sum_{i \leq h} (\xi_{n}(i) - \chi_{n}(i)) \mathbf{1}_{\{ \xi_{n}(i) \leq \delta b_{n} +1  \}} \right) \leq \frac{\sigma_{n, \delta}^{2}}{E_{n}} \delta b_{n}h$, for $h \geq 1$. 
\end{itemize}

\noindent Indeed, (\ref{eqExtra6}) and Chebychev’s inequality imply that
\begin{align*}
\mathbb{P} \left( \left|  {\rm R}^{\delta}(u_{n}) - \frac{\sigma_{n,\delta}^{2}}{2 E_{n}} |u_{n}| \right| > \varepsilon  \sigma_{n,\delta} \right) \leq \eta +  \left(\frac{K}{b_{n}} +K^{2}\right) \frac{K \delta \sigma_{n, \delta}^{2} V_{n} b_{n}}{\varepsilon^{2} \sigma_{n, \delta}^{2} E_{n} b_{n}} \leq  \eta +  \left(\frac{K}{b_{n}} +K^{2}\right) \frac{K \delta V_{n}}{\varepsilon^{2}  E_{n}}.
\end{align*}
\noindent Therefore, given that $\eta>0$ is arbitrary and \ref{B1}, (\ref{eq11}) follows by letting $n \rightarrow \infty$ and then $\delta \downarrow 0$. 

We start with the proof of (i). Note that, for every $1 \leq i \leq E_{n}$, $\mathbb{P}(\xi_{n}(i) = k) = kN_{k}^{n}/E_{n}$, for $k \geq 1$. Moreover, for every $k \geq 1$, $\mathbb{P}(\chi_{n}(i) = j | \xi_{n}(i) = k) = k^{-1}$, for $1 \leq j \leq k$. Hence, 
\begin{align*}
\displaystyle \mathbb{E} \left [\sum_{i \leq h} (\xi_{n}(i) - \chi_{n}(i)) \mathbf{1}_{\{ \xi_{n}(i) \leq \delta b_{n} +1  \}} \right] = h \sum_{k = 1}^{\lfloor \delta b_{n}  +1\rfloor} \sum_{j=1}^{k}\frac{(k-j)}{k}\frac{kN_{k}^{n}}{E_{n}} = \frac{\sigma_{n, \delta}^{2}}{2E_{n}} h.
\end{align*}
\noindent which proves (i). Now, we show (ii). Note that
\begin{align} \label{eqExtra3}
{\rm Var} \left (\sum_{i \leq h} (\xi_{n}(i) - \chi_{n}(i)) \mathbf{1}_{\{ \xi_{n}(i) \leq \delta b_{n} +1  \}} \right) & = {\rm Var} \left ( \mathbb{E} \left [\sum_{i \leq h} (\xi_{n}(i) - \chi_{n}(i)) \mathbf{1}_{\{ \xi_{n}(i) \leq \delta b_{n} +1  \}} \Big| (\xi_{n}(i), 1 \leq i \leq h) \right] \right) \nonumber \\
& \quad \quad  + \mathbb{E} \left[ {\rm Var} \left (\sum_{i \leq h} (\xi_{n}(i) - \chi_{n}(i)) \mathbf{1}_{\{ \xi_{n}(i) \leq \delta b_{n} +1  \}} \Big| (\xi_{n}(i), 1 \leq i \leq h) \right) \right] \nonumber \\
& = {\rm Var} \left (\sum_{i \leq h} \sum_{j=1}^{\xi_{n}(i)} \frac{(\xi_{n}(i) - j)}{\xi_{n}(i)}\mathbf{1}_{\{ \xi_{n}(i) \leq \delta b_{n} +1  \}}  \right) \nonumber \\
& \quad \quad  + \mathbb{E} \left[ {\rm Var} \left (\sum_{i \leq h} (\xi_{n}(i) - \chi_{n}(i)) \mathbf{1}_{\{ \xi_{n}(i) \leq \delta b_{n} +1  \}} \Big| (\xi_{n}(i), 1 \leq i \leq h) \right) \right]. 
\end{align}
\noindent On the one hand, the variables $\chi_{n}(1), \dots, \chi_{n}(E_{n})$ are conditionally independent. Then,
\begin{align} \label{eqExtra4}
& \mathbb{E} \left[ {\rm Var} \left (\sum_{i \leq h} (\xi_{n}(i) - \chi_{n}(i)) \mathbf{1}_{\{ \xi_{n}(i) \leq \delta b_{n} +1  \}} \Big| (\xi_{n}(i), 1 \leq i \leq h) \right) \right]  \nonumber \\
& \quad = \mathbb{E} \left[ \mathbb{E} \left[ \left (\sum_{i \leq h} (\xi_{n}(i) - \chi_{n}(i)) \mathbf{1}_{\{ \xi_{n}(i) \leq \delta b_{n} +1  \}}- \mathbb{E} \left[ \sum_{i \leq h} (\xi_{n}(i) - \chi_{n}(i)) \mathbf{1}_{\{ \xi_{n}(i) \leq \delta b_{n} +1  \}} \Big| (\xi_{n}(i), 1 \leq i \leq h) \right] \right)^{2}  \Big| (\xi_{n}(i), 1 \leq i \leq h) \right] \right]   \nonumber \\
& \quad = \mathbb{E} \left[ \mathbb{E} \left[ \sum_{i \leq h} \left ( (\xi_{n}(i) - \chi_{n}(i)) \mathbf{1}_{\{ \xi_{n}(i) \leq \delta b_{n} +1  \}}- \mathbb{E} \left[  (\xi_{n}(i) - \chi_{n}(i)) \mathbf{1}_{\{ \xi_{n}(i) \leq \delta b_{n} +1  \}} \Big| (\xi_{n}(i), 1 \leq i \leq h) \right] \right)^{2}  \Big| (\xi_{n}(i), 1 \leq i \leq h) \right] \right]   \nonumber \\
& \quad = \mathbb{E} \left[ \sum_{i \leq h} {\rm Var} \left( (\xi_{n}(i) - \chi_{n}(i)) \mathbf{1}_{\{ \xi_{n}(i) \leq \delta b_{n} +1  \}}   \Big| (\xi_{n}(i), 1 \leq i \leq h) \right) \right]  \nonumber \\
& \quad  \leq \mathbb{E} \left[ \sum_{i \leq h} \mathbb{E}  \left[  \left ((\xi_{n}(i) - \chi_{n}(i)) \mathbf{1}_{\{ \xi_{n}(i) \leq \delta b_{n} +1  \}}\right)^{2}  \Big| (\xi_{n}(i), 1 \leq i \leq h) \right]  \right]  \nonumber \\
& \quad  = \mathbb{E} \left[ \sum_{i \leq h} \sum_{j=1}^{\xi_{n}(i)} \frac{(\xi_{n}(i) - j)^{2}}{\xi_{n}(i)}\mathbf{1}_{\{ \xi_{n}(i) \leq \delta b_{n} +1  \}}  \right] \nonumber \\
& \quad  \leq \delta b_{n} h \mathbb{E} \left[ \sum_{j=1}^{\xi_{n}(1)} \frac{(\xi_{n}(1) - j)}{\xi_{n}(1)}\mathbf{1}_{\{ \xi_{n}(1) \leq \delta b_{n} +1  \}}  \right] \nonumber  \\
& \quad  \leq \frac{\sigma_{n, \delta}^{2}}{2 E_{n}} \delta b_{n} h. 
\end{align}
\noindent On the other hand, recall that the variables $\xi_{n}(1), \dots, \xi_{n}(E_{n})$ are obtained by successive picks without replacement in an urn, and therefore they are negatively correlated. Then, 
\begin{align}\label{eqExtra5}
{\rm Var} \left (\sum_{i \leq h} \sum_{j=1}^{\xi_{n}(i)} \frac{(\xi_{n}(i) - j)}{\xi_{n}(i)}\mathbf{1}_{\{ \xi_{n}(i) \leq \delta b_{n} +1  \}}  \right) & \leq h {\rm Var} \left ( \sum_{j=1}^{\xi_{n}(1)} \frac{(\xi_{n}(1) - j)}{\xi_{n}(1)}\mathbf{1}_{\{ \xi_{n}(1) \leq \delta b_{n} +1  \}}  \right)  \nonumber \\
& \leq h \mathbb{E} \left [ \left( \sum_{j=1}^{\xi_{n}(1)} \frac{(\xi_{n}(1) - j)}{\xi_{n}(1)}\mathbf{1}_{\{ \xi_{n}(1) \leq \delta b_{n} +1  \}} \right)^{2} \right] \nonumber \\
& \leq h \mathbb{E} \left [ \sum_{j=1}^{\xi_{n}(1)} \frac{(\xi_{n}(1) - j)^{2}}{\xi_{n}(1)}\mathbf{1}_{\{ \xi_{n}(1) \leq \delta b_{n} +1  \}}\right] \nonumber \\
& \leq  \frac{\sigma_{n, \delta}^{2}}{2 E_{n}} \delta b_{n} h.
\end{align} 
\noindent Therefore, the combination of \eqref{eqExtra3}, \eqref{eqExtra4} and \eqref{eqExtra5} proves (ii). 

The proof of the full statement in Theorem \ref{Theo7} follows as explained at the end of the proof of \cite[Theorem 2.4]{Cyril2019}, which is based on an extension of the previous argument and the spinal decomposition in \cite[Lemma 4.1]{Cyril2019}. Fix $q \geq 2$ an integer and let $U_{1}, \dots, U_{q}$ be $q$ i.i.d.\ uniform random variables in $[0,1]$ independently of $\mathbf{t}_{n}$. For $i=1, \dots, q$, let $u_{i,n}$ be the $\lfloor V_{n} U_{i} \rfloor+1$-st vertex of $\mathbf{t}_{n}$ in lexicographical order, so that $u_{1,n}, \dots, u_{q,n}$ are $q$ i.i.d.\ uniform random vertices of $\mathbf{t}_{n}$ (possibly with repetition). Let $\mathbf{T}^{(q)}_{n}$ be the reduced tree built from $\mathbf{t}_{n}$ by keeping only the root of $\mathbf{t}_{n}$ and these $q$ i.i.d.\ vertices together with their ancestors. Note that $\mathbf{T}^{(q)}_{n}$ has a random number of branching points $b \in \{0, \dots, q-1\}$ and a random number of leaves $\hat{q} \in \{1, \dots, q\}$.  

Remove from $\mathbf{T}^{(q)}_{n}$ its $\hat{q}$ leaves and $b$  branching points (if any) to obtain a random number $\hat{q}+b$ of single branches. Denote by $B_{n} = \{B^{1}_{n}, \dots, B^{\hat{q}+b}_{n}\}$ this collection (i.e., each $B^{i}_{n}$, for $i =1, \dots, \hat{q}+b$, is a set of vertices). It is worth noting that the sets (``branches'') $B^{i}_{n}$  may be empty. In the case $q=1$, with a single vertex $u_{1,n}$, the set $B_{1}^{1}$ is the ancestral line $\llbracket \emptyset, u_{1,n}\llbracket$. For every $i =1, \dots, \hat{q}+b$, we let $v_{i}(1), \dots, v_{i}(\# B^{i}_{n})$ be the vertices in $B^{i}_{n}$ (if $\# B^{i}_{n}  \geq 1)$ in lexicographical order. So, any vertex $v_{i}(m)$, for $m =1, \dots, \# B^{i}_{n}$, has $k_{m,i}^{n} \geq 1$ offspring in the original tree $\mathbf{t}_{n}$, and only one of them, the $j_{m,i}^{n}$-th say, is an ancestor of (or equal to) at least one the random vertices $u_{1,n}, \dots, u_{q,n}$. In particular, let $\hat{v}_{i}(\# B^{i}_{n})$ be the $j_{m,i}^{n}$-th child of $v_{i}(\# B^{i}_{n})$. 

As before, for a vertex $u$ (say, the $k$-th vertex of $\mathbf{t}_{n}$ in lexicographical order) that belongs to one of the single branches in $B_{n}$, let ${\rm R}^{\delta}(u)$ be the number of individuals branching-off strictly to the right of the ancestral line $\llbracket \emptyset, u \llbracket$ in $\mathbf{t}_{n}$, except for vertices that are children of vertices with degree larger than $\delta b_{n}+1$, i.e.,  
\begin{align} \label{eqExtra12vert2}
{\rm R}^{\delta}(u) = G_{n}^{\delta}(k),
\end{align}
\noindent where $G_{n}^{\delta}$ is the modified $\L$ukasiewicz defined in (\ref{ModLu}). For $i =1, \dots, \hat{q}+b$, recall that $pr(v_{i}(1))$ denotes the parent of the vertex $v_{i}(1)$ and let $\chi_{v_{i}(1)} \in \{1, \dots, k_{pr(v_{i}(1))} \}$ be the only index such that $v_{i}(1) = pr(v_{i}(1)) \chi_{v_{i}(1)}$ (here, $k_{pr(v_{i}(1))}$ denotes the number of children of $pr(v_{i}(1))$). Now, let
\begin{eqnarray} \label{eqExtra12vert3}
\hat{R}_{i,n}^{\delta} = \left\{ \begin{array}{lcl}
     {\rm R}^{\delta}(\hat{v}_{i}(\# B^{i}_{n}))- {\rm R}^{\delta}(pr(v_{i}(1))) - (k_{pr(v_{i}(1))} - \chi_{v_{i}(1)}) \mathbf{1}_{\{ |k_{pr(v_{i}(1))}-1| > \delta b_{n}\}}  & \mbox{  if } \,\, \# B^{i}_{n}  \geq 1,\\
            0  &  \mbox{otherwise}. \\
              \end{array}
    \right. 
\end{eqnarray}
\noindent for $i =1, \dots, \hat{q}+b$. We claim that, for every $\varepsilon >0$,
\begin{align}  \label{eq112vert}
\lim_{\delta \downarrow 0}\limsup_{n \rightarrow \infty}\mathbb{P} \left( \sup_{i=1,\dots, \hat{q}+b}\left| \hat{R}_{i,n}^{\delta} - \frac{\sigma_{n,\delta}^{2}}{2 E_{n}} \# B^{i}_{n} \right| > \varepsilon \sigma_{n,\delta} \right)= 0. 
\end{align}
\noindent Therefore, a combination of (\ref{eqExtra12vert2}),  (\ref{eqExtra12vert3}), (\ref{eq112vert}), (\ref{eqExtra2}), the union bound and Theorem \ref{Theo6} shows Theorem \ref{Theo7} for $q\geq 2$. 

We prove (\ref{eq112vert}) along the lines of the proof of (\ref{eq11}). Fix $\varepsilon, \eta >0$. By \cite[Propositions 4.4 and 4.5]{Cyril2019} and our assumptions, we can and will consider $K >0$ such that 
\begin{align} \label{EventTwo}
\mathbb{P} \left( \sum_{i=1}^{\hat{q}+b} \sum_{m=1}^{\# B^{i}_{n}} (k_{m,i}^{n} -1) \leq K b_{n} \quad \text{and} \quad  \sum_{i=1}^{\hat{q}+b} \# B^{i}_{n} \leq  \frac{KV_{n}}{b_{n}} \right ) \geq 1-\eta.
\end{align} 
\noindent for $n$ large enough. Since we have assumed that $b_{n} \rightarrow \infty$, as $n \rightarrow \infty$, we have that, under the above event in particular, with high probability none of the $q$ random vertices is an ancestor of another (i.e., the $q$ i.i.d.\ vertices are leaves in the reduced tree), so that $q=\hat{q}$. Recall the definition of the random variables $(\xi_{n}(i), 1 \leq i \leq E_{n})$, and $(\chi_{n}(i), 1 \leq i \leq E_{n})$. The spinal decomposition in \cite[Lemma 4.1]{Cyril2019} with $q \geq 2$ shows that under the above event (i.e., \eqref{EventTwo}), for any $h_{0}, h_{1}, \dots, h_{q+b}$ with $h_{0}=0$ and $h_{1}+\dots+h_{q+b} \leq K b_{n}^{-1}V_{n}$, and any integers $k_{m,i} \geq j_{m,i} \geq 1$ for every $m \in\{1, \dots, q+b\}$ and $i \in \{1, \dots, h_{m}\}$, the probability that the tree reduced to the $q$ random vertices and their ancestors has $b$  branching points and that, for every $m$, we have $\# B^{m}_{n} = h_{m}$ and $(k_{m,i}^{n}, j_{m,i}^{n})_{1 \leq i \leq \# B^{m}_{n}} = (k_{m,i}, j_{m,i})_{1 \leq i \leq h_{m}}$ is upper bounded by 
\begin{eqnarray*}
C \left(\frac{b_{n}}{V_{n}} \right)^{q+b} \mathbb{P} \left( \bigcap_{m=1}^{q+b} \bigcap_{i=1}^{h_{m}} \{ (\xi_{n}(i), \chi_{n}(i)) = (k_{m,i}, j_{m,i}) \} \right),
\end{eqnarray*}
\noindent for some constant $C >0$ that depends on $K$ (and $n$ large enough). Thus, by the union bound, 
\begin{align}
& \mathbb{P} \left( \sup_{i=1,\dots,\hat{q}+b}\left| \hat{R}_{i,n}^{\delta} - \frac{\sigma_{n,\delta}^{2}}{2 E_{n}} \# B^{i}_{n} \right| > \varepsilon \sigma_{n,\delta} \right) \nonumber \\
& \quad \quad \leq \eta + C \sum_{b=0}^{q-1}\left(\frac{b_{n}}{V_{n}} \right)^{q+b}   \sum_{h_{1}+\dots + h_{q+b} \leq K b_{n}^{-1}V_{n}} \sum_{m=1}^{q+b}   \mathbb{P} \left( \left| \sum_{1 \leq i \leq h_{m}} (\xi_{n}(i) - \chi_{n}(i)) \mathbf{1}_{\{ \xi_{n}(i) \leq \delta b_{n} +1  \}} - \frac{\sigma_{n, \delta}^{2}}{2 E_{n}} h_{m} \right| > \varepsilon \sigma_{n, \delta}\right).
\end{align}
\noindent Hence, by (i) and (ii) and the Chebychev's inequality, we have that, for some constant $C^{\prime}>0$
\begin{align}
& \mathbb{P} \left( \sup_{i=1,\dots,\hat{q}+b}\left| \hat{R}_{i,n}^{\delta} - \frac{\sigma_{n,\delta}^{2}}{2 E_{n}} \# B^{i}_{n} \right| > \varepsilon \sigma_{n,\delta} \right) \leq \eta + C^{\prime} \sum_{b=0}^{q-1}\frac{K^{q+b+1} \delta V_{n}}{\varepsilon^{2}E_{n}}.
\end{align}

\noindent Therefore, given that $\eta>0$ is arbitrary, using \ref{B1} and Lemma \ref{lemma1}, (\ref{eq112vert}) follows by letting $n \rightarrow \infty$ and then $\delta \downarrow 0$. 
\end{proof}

\subsection{Gromov-Hausdorff-Prohorov convergence of a specific model of TGDS} \label{ConvGHP}

In this section, we consider a specific case in which the convergence of trees with given degree sequence holds for the Gromov-Hausdorff-Prohorov topology. 

\begin{proposition} \label{The7}
Suppose that $\mathbf{s}_{n}$ satisfies \ref{B1}-\ref{B}. If moreover, 
\begin{align*}
\lim_{n \rightarrow \infty} \frac{b_{n}^{2}}{V_{n}} = 1 \quad \text{and} \quad \limsup_{n \rightarrow \infty} \frac{N_{1}^{n}}{V_{n}} <1,
\end{align*}
\noindent then
\begin{align*}
(\mathbf{t}_{n},  b_{n}^{-1} r_{n}^{{\rm gr}}, \rho_{n}, \mu_{n} ) \xrightarrow[ ]{d} (\mathcal{T}_{\theta}, r_{\theta}, \rho_{\theta}, \mu_{\theta}), \hspace*{3mm} \text{as} \hspace*{2mm}  n \rightarrow \infty,
\end{align*}
\noindent for the Gromov-Hausdorff-Prohorov topology, where $\mathcal{T}_{\theta}$ is an Inhomogeneous CRT with parameter set $\theta = (\theta_{0}, \theta_{1}, \dots)$.
\end{proposition}

As mentioned in the introduction, Blanc-Renaudie \cite[Theorem 7 (a)]{Arthur2021} established necessary conditions for Gromov-Hausdorff-Prohorov convergence. Our assumptions \ref{B1}-\ref{B} and the condition $b_{n}^{2}/V_{n} \rightarrow 1$, as $n \rightarrow \infty$, align with \cite[Assumption 2]{Arthur2021} (denoted as $\mathcal{D}_{n} \Rightarrow \Theta$ in \cite[Theorem 7]{Arthur2021}). However, \cite{Arthur2021} imposed additional technical requirements, specifically \cite[Assumption 7]{Arthur2021}. In particular, \cite[Assumption 7 (ii)]{Arthur2021} imposes a condition on vertices with an out-degree greater than one (corresponding to $\sum_{i\geq 2} N_{i}^{n}$ in our notation), which is not a requirement for Proposition \ref{The7}. Furthermore, \cite[Assumption 7 (i)]{Arthur2021} presents another technical condition that does not appear straightforward to verify under the assumptions of Proposition \ref{The7}. Therefore, we have included this specific case because its proof, based on the convergence of height processes (see \eqref{HeightConv} below), differs from the approach in \cite{Arthur2021} and may be of independent interest. The proof of Proposition \ref{The7} builds upon arguments found in the proofs of \cite[Theorems 1 and 3]{Broutin2014}.

\begin{proof}[Proof of Proposition \ref{The7}] 
We prove that the family of processes $((b_{n}^{-1}H_{n}(V_{n}u), u \in [0,1]): n \geq 1)$ is tight. Since $H_{n}(0)=0$, it is enough to check that for any $\varepsilon, \varepsilon^{\prime} >0$, there exists $0 < \eta < 1$ such that 
\begin{align} \label{eq23}
\limsup_{n \rightarrow \infty}  \mathbb{P}(  \omega_{\eta}(H_{n}) \geq \varepsilon b_{n}) \leq \varepsilon^{\prime}, 
\end{align}

\noindent where $\omega_{\eta}(g) = \sup_{|u -u^{\prime}|\leq \eta} |g(u) - g(u^{\prime})|$, for a continuous function $g :[0,1] \rightarrow \mathbb{R}$; see \cite[Theorem 2.7.3]{Billi1999}. 

Once tightness is in place, Theorem \ref{Theo7} implies that any subsequential weak limit of $((b_{n}^{-1}H_{n}(V_{n}u), u \in [0,1]): n \geq 1)$ is distributed as $\frac{2}{\theta_{0}^{2}} H^{\rm exc}$, i.e., 
\begin{eqnarray} \label{HeightConv}
\left(\frac{\theta_{0}^{2}}{2b_{n}} H_{n}( V_{n} u), u \in [0,1] \right) \xrightarrow[ ]{d} (H^{\rm exc}(u), u \in [0,1]), \hspace*{3mm} \text{as} \hspace*{2mm}  n \rightarrow \infty, \hspace*{2mm} \text{in} \hspace*{2mm} \mathbf{D}([0,1], \mathbb{R}).
\end{eqnarray}
\noindent Then, the claim in Proposition \ref{The7} follows, from e.g.  \cite[Lemma 3.19]{Marckert2006E} and \cite[Proposition 3.3]{Abraham2013}. 

For two vertices $u, v \in \mathbf{t}_{n}$, recall that $u \wedge v$ denotes the first (highest) common ancestor of $u$ and $v$. Recall also that $|u|$ denotes the height of $u \in \mathbf{t}_{n}$. Then,
\begin{align} \label{eq1NSe3}
||u|-|v|| \leq ||u|-|u \wedge v|| + ||v|-|u \wedge v||. 
\end{align}
Let $u(0) \prec_{ \text{lex}}  u(1) \prec_{ \text{lex}}  \dots \prec_{ \text{lex}}  u(V_{n} -1)$ be the sequence of vertices of $\mathbf{t}_{n}$ in lexicographical order, where in particular $u(0)$ is the root of $\mathbf{t}_{n}$. By (\ref{eq1NSe3}), we deduce that
\begin{align} \label{eq20}
\underset{|i-j|\leq \eta V_{n}}{\sup}|H_{n}(i)-H_{n}(j)| \leq 2+2\sup_{\substack{|i-j|\leq \eta V_{n} \\ u (j) \in \llbracket u(0), u(i) \rrbracket  }} |H_{n}(i)-H_{n}(j)|, 
\end{align}
\noindent where the supremum is over $i,j \in \{ 0, \dots, V_{n}-1 \}$ such that $ u (j) \in \llbracket u(0), u(i) \rrbracket$. Recall that $\llbracket u,v \rrbracket$ denotes the unique geodesic path between $u$ and $v$ in $\mathbf{t}_{n}$. For $i,j \in \{ 0, \dots, V_{n}-1 \}$ such that $ u (j) \in \llbracket u(0), u(i) \rrbracket$, we see that every $v \in \llbracket u(j), u(i) \rrbracket$ which has degree more than one contributes at least one to the number of vertices branching-off the path $\llbracket u(j), u(i) \rrbracket$. So, for $\delta >0$ and $i,j \in \{ 0, \dots, V_{n}-1 \}$ such that $ u (j) \in \llbracket u(0), u(i) \rrbracket $, we have that
\begin{align} \label{eq21}
H_{n}(i)-H_{n}(j) - J_{n}(i,j) & \leq 1  + \sum_{v \in \llbracket u(j), u(i) \rrbracket} (k_{v}-1) \mathbf{1}_{\{ |k_{v}-1| \leq \delta b_{n} \}} + \sum_{v \in \llbracket u(j), u(i) \rrbracket} \mathbf{1}_{\{ |k_{v}-1| > \delta b_{n} \}},
\end{align}
\noindent where 
\begin{align}
J_{n}(i,j) = \sum_{v \in \llbracket u(j), u(i) \rrbracket} \mathbf{1}_{\{ k_{v} =1 \}}.
\end{align}

For $\delta >0$, let $G_{n}^{\delta}$ be the modified-$\L$ukasiewicz path defined in (\ref{ModLu}). We define also $\hat{G}_{n}^{\delta} = (\hat{G}_{n}^{\delta}(V_{n}u), u \in [0,1])$ as the version of the modified-$\L$ukasiewicz path defined as in (\ref{ModLu}) but with the time-rescaled reverse-$\L$ukasiewicz path $W_{n}^{\rm rev} = (W_{n}^{\rm rev}(V_{n}u), u \in [0,1])$ of $\mathbf{t}_{n}$ instead of $W_{n}^{\rm lex}$. In particular,  the claim in Theorem \ref{Theo6} remains valid if we replace $W_{n}^{\rm lex}$ and $G_{n}^{\delta}$ by $W_{n}^{\rm rev}$ and $\hat{G}_{n}^{\delta}$, respectively. It follows from (\ref{rightEq1}), (\ref{eq10}) and (\ref{eq19}) that, $i,j \in \{ 0, \dots, V_{n}-1 \}$ such that $ u (j) \in \llbracket u(0), u(i) \rrbracket $,
\begin{align} \label{eq22}
\sum_{v \in \llbracket u(j), u(i) \rrbracket} (k_{v}-1) \mathbf{1}_{\{ |k_{v}-1| \leq \delta b_{n} \}} \leq G_{n}^{\delta}(i+1) - G_{n}^{\delta}(j+1) + \hat{G}_{n}^{\delta}(i+1) - \hat{G}_{n}^{\delta}(j+1) +  2(k_{u(j)}-1) \mathbf{1}_{\{ |k_{u(j)}-1| \leq \delta b_{n}\}}
\end{align}
\noindent and
\begin{align} \label{eq2NSe3}
\sum_{v \in \llbracket u(j), u(i) \rrbracket}  \mathbf{1}_{\{ |k_{v}-1| > \delta b_{n} \}} \leq  \sum_{r=0}^{i}  \mathbf{1}_{\{ |\Delta W_{n}^{\rm lex}(r+1)|  > \delta b_{n} \}} \leq I_{n}^{\delta} = \max \{ i \in \{1, \dots, V_{n}\}: |\Delta W^{\text{lex}}_{n}(i)| > \delta b_{n}\}.
\end{align}

\noindent Then,  (\ref{eq20}), (\ref{eq21}), (\ref{eq22}) and (\ref{eq2NSe3}) allow us to deduce that, for $0 < \eta < 1$,
\begin{align} \label{eq3NSe3}
& \underset{|i-j|\leq \eta V_{n}}{\sup}|H_{n}(i)-H_{n}(j)| \nonumber \\
& \quad \quad \leq 4+2\sup_{\substack{|i-j|\leq \eta V_{n} \\ u (j) \in \llbracket u(0), u(i) \rrbracket  }} |G_{n}^{\delta}(i)-G_{n}^{\delta}(j)| + 2\sup_{\substack{|i-j|\leq \eta V_{n} \\ u (j) \in \llbracket u(0), u(i) \rrbracket  }} |\hat{G}_{n}^{\delta}(i)-\hat{G}_{n}^{\delta}(j)| + 2I_{n}^{\delta} + 4\delta b_{n} + \underset{|i-j|\leq \eta V_{n}}{\sup} J_{n}(i,j) \nonumber  \\
& \quad \quad \leq 4 + 2 b_{n} \omega_{\eta}(b_{n}^{-1}G_{n}^{\delta}) + 2 b_{n} \omega_{\eta}(b_{n}^{-1} \hat{G}_{n}^{\delta}) + 2I_{n}^{\delta} + 4\delta b_{n} + \underset{|i-j|\leq \eta V_{n}}{\sup} J_{n}(i,j). 
\end{align}

\indent \noindent For $0 < \eta <1$ and $k \in \mathbb{N}$, a sequence $\Delta_{k} = \{0 = t_{0} < t_{1} < \cdots < t_{k} = 1 \}$ of subdivisions of $[0, 1]$ is called $\eta$-sparse if it satisfies $\min_{1 \leq i \leq k} (t_{i} - t_{i-1}) \geq \eta$. The so-called modified modulus of continuity in $\mathbb{D}([0,1], \mathbb{R})$ is given by
\begin{eqnarray*}
\bar{\omega}_{\eta}(g) \coloneqq \inf_{\Delta_{k}} \, \max_{1 \leq i \leq k} \, \, \sup_{r,r^{\prime} \in [t_{i-1}, t_{i})} |g(r) - g(r^{\prime})|,  \quad \text{for} \, \, g \in \mathbb{D}([0,1], \mathbb{R}),
\end{eqnarray*}
\noindent where the infimum extends over all $\eta$-sparse sets $\Delta_{k}$. Note that (see e.g. \cite[equation (12.9)]{Billi1999})
\begin{align*}
\omega_{\eta}(b_{n}^{-1} G_{n}^{\delta}) \leq 2\bar{\omega}_{\eta}( b_{n}^{-1} G_{n}^{\delta})  + \delta  \quad \text{and} \quad \omega_{\eta}(b_{n}^{-1} \hat{G}_{n}^{\delta}) \leq 2\bar{\omega}_{\eta}(b_{n}^{-1} \hat{G}_{n}^{\delta})  + \delta.
\end{align*}

\noindent On the other hand, it is not difficult to see that, for $n$ large enough, by $\ref{B2}-\ref{B}$ that $I_{n}^{\delta} \leq I^{\delta} = \max \{i \geq 1: \theta_{i}>\delta \}$. Hence, by (\ref{eq3NSe3}), 
\begin{align} \label{eq4NSe3}
 \underset{|i-j|\leq \eta V_{n}}{\sup}b_{n}^{-1}|H_{n}(i)-H_{n}(j)| \leq 4 b_{n}^{-1} + 4 \bar{\omega}_{\eta}(b_{n}^{-1}G_{n}^{\delta}) + 4 \bar{\omega}_{\eta}(b_{n}^{-1} \hat{G}_{n}^{\delta}) + 2I^{\delta}b_{n}^{-1} + 6\delta + b_{n}^{-1} \underset{|i-j|\leq \eta V_{n}}{\sup} J_{n}(i,j). 
\end{align}
\noindent Given $0 < \eta < 1$ arbitrary, one can choose $\delta$ small enough so that the the first five terms in (\ref{eq4NSe3}) are arbitrary small (in probability) when $n$ is large. Indeed, Theorem \ref{Theo5} allows us to bound the terms $\bar{\omega}_{\eta}(b_{n}^{-1}G_{n}^{\delta})$ and $\bar{\omega}_{\eta}(b_{n}^{-1} \hat{G}_{n}^{\delta})$. It only remains to control the sixth term in (\ref{eq4NSe3}), that is, the term concerning the number of vertices of degree one. This is done as in the last part of the \cite[Proof of Lemma 8]{Broutin2014} to which we refer for a detailed argument. Informally, the idea is first sample a tree $\mathbf{t}_{n}$ with degree sequence $\mathbf{s}_{n}$ and then remove from it all the vertices of degree one, then the tightness of the associated height processes follows by (\ref{eq4NSe3}); one then needs to plug back these vertices of degree one. 
\end{proof}

Interestingly, Proposition~\ref{The7} shows that the limit of certain TGDSs can be the Brownian CRT. Specifically, it suffices to consider degree sequences where the number of vertices with degree $1$ is asymptotically negligible compared to the total number of vertices in the TGDS.

\begin{corollary}
Suppose that $\mathbf{s}_{n}$ satisfies \ref{B1}-\ref{B}. If moreover,
\begin{align} \label{Rev2SPAEq1}
\lim_{n \rightarrow \infty} \frac{b_{n}^{2}}{V_{n}} = 1 \quad \text{and} \quad \lim_{n \rightarrow \infty} \frac{N_{1}^{n}}{V_{n}} =0,
\end{align}
\noindent then
\begin{align*}
(\mathbf{t}_{n},  b_{n}^{-1} r_{n}^{{\rm gr}}, \rho_{n}, \mu_{n} ) \xrightarrow[ ]{d} (\mathcal{T}_{\rm Br}, r_{\rm Br}, \rho_{\rm Br}, \mu_{\rm Br}), \hspace*{3mm} \text{as} \hspace*{2mm}  n \rightarrow \infty,
\end{align*}
\noindent for the Gromov-Hausdorff-Prohorov topology, where $\mathcal{T}_{\rm Br}$ is the Brownian CRT.
\end{corollary}

\begin{proof}
By Proposition \ref{The7}, it is enough to identify that the limiting Inhomogeneous CRT is indeed the Brownian CRT. This is equivalent to showing that $\sum_{i \geq 1} \theta_{i}^{2} =0$ (i.e.\ $\theta_{0}=1$). To see this, note that
\begin{align*}
\sum_{i \geq 0} (i-1)^{2}N_{i}^{n} - V_{n} & = \sum_{i \geq 0} (i-1)^{2}N_{i}^{n}  - \sum_{i \geq 0} N_{i}^{n}  = -N_{1}^{n} + \sum_{i \geq 3} i(i-2) N_{i}^{n}  \nonumber \\
& = -N_{1}^{n} + \sum_{i \geq 1} d_{n}(i) (d_{n}(i)-2) \mathbf{1}_{\{d_{n}(i) \geq 3 \}}  \geq -N_{1}^{n} + \sum_{i \geq 1} d_{n}(i) (d_{n}(i)-2) \mathbf{1}_{\{d_{n}(i) > \delta b_{n} \}},
\end{align*} 
\noindent for all $\delta >0$ and $n$ such that $\delta b_{n} > 2$. Then, \eqref{Rev2SPAEq1} and our assumptions \ref{B1}-\ref{B} imply that 
\begin{align*}
\sum_{i \geq 1} \theta_{i}^{2} \mathbf{1}_{\{ \theta_{i} >\delta \}} = \liminf_{n \rightarrow \infty} \sum_{i \geq 1} \frac{d_{n}(i) (d_{n}(i)-2)}{V_{n}} \mathbf{1}_{\{d_{n}(i) \geq \delta b_{n} \}} \leq \liminf_{n \rightarrow \infty} \frac{N_{1}^{n}}{V_{n}} = 0,
\end{align*}
\noindent for all $\delta >0$. This implies that $\sum_{i \geq 1} \theta_{i}^{2} = 0$, which concludes our proof.
\end{proof}

\section{Lamination-valued processes} \label{sec:laminations}

This section is devoted to the proof of Theorem \ref{thm:cvlamproc}, which establishes the convergence of lamination-valued processes associated with plane trees. We will start by defining the set of laminations in Section \ref{laminations}. Then, we will rigorously define laminations for discrete and continuum trees in Sections \ref{discretelamination} and \ref{continuumlamination}, respectively, before presenting the proof of Theorem \ref{thm:cvlamproc} in Section \ref{proofLaminationT}.

Throughout this section, whenever we consider a random rooted plane tree $\tau$, we assume that the number of its vertices, $\zeta(\tau)$, is deterministic.  Similarly, for a sequence of such trees  $(\tau_{n}, n\geq 1)$, we also assume that $\zeta(\tau_{n})$ is deterministic for all $n \geq 1$.

\subsection{The set of laminations} \label{laminations}

Recall that a lamination is a closed subset of the closed unit disk $\bar{\D}$ which can be written as the union of the unit circle $\bS^1$ and a collection of chords which do not intersect in the open unit disk $\D$. By definition, a lamination is always compact. In this section, we denote by $d_{\rm H}$ the Hausdorff distance on the set $\mathbb{K}(\bar{\mathbb{D}})$ of compact subsets of $\bar{\mathbb{D}}$. Since $\bL(\bar{\mathbb{D}}) \subset \mathbb{K}(\bar{\mathbb{D}})$, $\bL(\bar{\mathbb{D}})$ is also naturally equipped with $d_H$. We denote by $d_{\rm Sk}^{\mathbb{L}}$ the $J_{1}$ Skorohod distance on $\mathbf{D}(\mathbb{R}_{+}, \bL(\bar{\mathbb{D}}))$; see e.g. \cite[Section 5 in Chapter 3]{Ethier1986} or \cite[Chapter 3]{Billi1999} for a precise definition.

\subsection{The discrete setting} \label{discretelamination}

\subsubsection{Laminations associated to plane trees}

We start by considering rooted plane trees. In this setting, as for fragmentation processes, there are two natural ways to define a lamination-valued process: either the one obtained from removing edges one by one at integer times, or the one that we get when putting i.i.d.\ variables on edges and removing those whose variable is smaller than a given value. 

\begin{definition}[Discrete lamination-valued process] \label{Def2}
Let $\tau$ be a rooted plane tree with contour function $C_{\tau}$ and  let $(e_{1}, \dots, e_{\zeta(\tau)-1})$ be a random uniform ordering of its edges. For $k = 1, \dots, \zeta(\tau)-1$, let $g_{k}$ and $d_{k}$ be the first and last times at which the contour function $C_{\tau}$ visits the endpoint of the edge $e_{k}$ further from the root $\emptyset$. Associate to $e_{k}$ the chord $c_{k} \coloneqq [e^{-2\pi i g_{k}/2\zeta(\tau)}, e^{-2\pi i d_{k}/2\zeta(\tau)}] \subset \bar{\mathbb{D}}$. We define the lamination-valued process $(\mathbb{L}_{t}(\tau), t \geq 0)$ associated to $\tau$ by letting
\begin{align*}
\mathbb{L}_{t}(\tau) \coloneqq \mathbb{S}^{1} \cup \bigcup_{k=1}^{\lfloor t \rfloor \wedge (\zeta(\tau)-1)} c_{k}, \hspace*{3mm} \text{for} \hspace*{2mm} t \in [0,\infty].
\end{align*}
\end{definition}

In particular, the process $(\mathbb{L}_{t}(\tau), t \geq 0)$ interpolates between $\mathbb{S}^{1}$ and $\mathbb{L}(\tau) \coloneqq \mathbb{L}_{\infty}(\tau)$; see Figure \ref{fig:arbconlam}. We also consider a dynamic continuous-time version of the lamination-valued process $(\mathbb{L}_{t}(\tau), t \geq 0)$. 

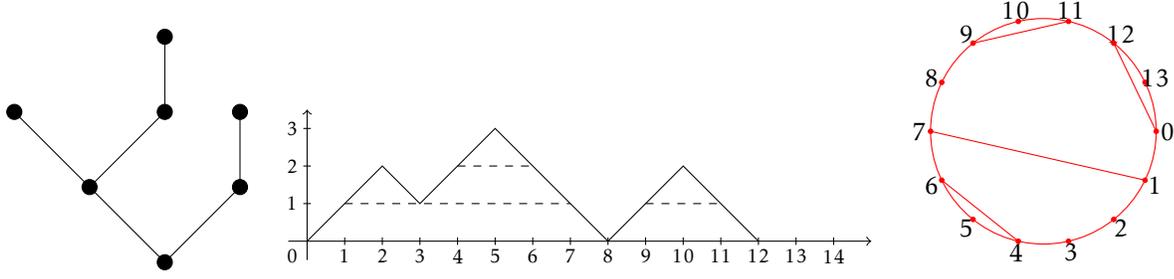
\begin{figure}[!htb]
\begin{tabular}{c c c}
\begin{tikzpicture}
\draw (1,2) -- (1,1) -- (0,0) -- (-1,1)--(0,2)--(0,3) (-1,1) -- (-2,2);
\draw[fill=black] (0,0) circle (.1);
\draw[fill=black] (1,1) circle (.1);
\draw[fill=black] (-1,1) circle (.1);
\draw[fill=black] (0,2) circle (.1);
\draw[fill=black] (0,3) circle (.1);
\draw[fill=black] (-2,2) circle (.1);
\draw[fill=black] (1,2) circle (.1);
\end{tikzpicture}
&
\begin{tikzpicture}[scale=.5, every node/.style={scale=0.7}]
\draw (0,0) -- (1,1) -- (2,2) -- (3,1) -- (4,2) -- (5,3) -- (6,2) -- (7,1) -- (8,0) -- (9,1) -- (10,2) -- (11,1) -- (12,0);
\draw[->] (0,-.5) -- (0,3.5);
\draw[->] (-.5,0) -- (15,0);
\draw (1,.1) -- (1,-.1);
\draw (1,-.4) node{1};
\draw (2,.1) -- (2,-.1);
\draw (2,-.4) node{2};
\draw (3,.1) -- (3,-.1);
\draw (3,-.4) node{3};
\draw (4,.1) -- (4,-.1);
\draw (4,-.4) node{4};
\draw (5,.1) -- (5,-.1);
\draw (5,-.4) node{5};
\draw (6,.1) -- (6,-.1);
\draw (6,-.4) node{6};
\draw (7,.1) -- (7,-.1);
\draw (7,-.4) node{7};
\draw (8,.1) -- (8,-.1);
\draw (8,-.4) node{8};
\draw (9,.1) -- (9,-.1);
\draw (9,-.4) node{9};
\draw (10,.1) -- (10,-.1);
\draw (10,-.4) node{10};
\draw (11,.1) -- (11,-.1);
\draw (11,-.4) node{11};
\draw (12,.1) -- (12,-.1);
\draw (12,-.4) node{12};
\draw (13,.1) -- (13,-.1);
\draw (13,-.4) node{13};
\draw (14,.1) -- (14,-.1);
\draw (14,-.4) node{14};
\draw (.1,1) -- (-.1,1);
\draw (-.4,1) node{1};
\draw (.1,2) -- (-.1,2);
\draw (-.4,2) node{2};
\draw (.1,3) -- (-.1,3);
\draw (-.4,3) node{3};
\draw (-.4,-.4) node{0};
\draw [dashed](1,1) -- (7,1);
\draw [dashed](4,2) -- (6,2);
\draw [dashed](9,1) -- (11,1);
\end{tikzpicture}
&
\begin{tikzpicture}[scale=1.5, every node/.style={scale=.9}]
\foreach \i in {0,...,13}
{
\draw[auto=right] ({1.1*cos(-(\i)*360/14)},{1.1*sin(-(\i)*360/14)}) node{\i};
\draw[red,fill=red] ({cos(-(\i-1)*360/14)},{sin(-(\i-1)*360/14)}) circle (.02);
}
\draw[red] (0,0) circle (1);
\draw[red] ({cos(-360*0/14)},{sin(-360*0/14)}) -- ({cos(-360*12/14)},{sin(-360*12/14)});
\draw[red] ({cos(-360*7/14)},{sin(-360*7/14)}) -- ({cos(-360*1/14)},{sin(-360*1/14)});
\draw[red] ({cos(-360*4/14)},{sin(-360*4/14)}) -- ({cos(-360*6/14)},{sin(-360*6/14)});
\draw[red] ({cos(-360*9/14)},{sin(-360*9/14)}) -- ({cos(-360*11/14)},{sin(-360*11/14)});
\end{tikzpicture}
\end{tabular}
\caption{A tree $\tau$, its contour function $C_{\tau}$, and the associated lamination $\mathbb{L}(\tau)$.}
\label{fig:arbconlam}
\end{figure}

\begin{definition}[Dynamic discrete lamination-valued process] \label{Def3}
Let $\tau$ be a rooted plane tree with contour function $C_{\tau}$ and denote by $(e_{1}, \dots, e_{\zeta(\tau)-1})$ its edges (their ordering is irrelevant). Given the tree $\tau$, its edges are equipped with i.i.d.\ exponential random variables of parameter $1$, say $(\gamma_{1}, \dots, \gamma_{\zeta(\tau)-1})$. For $k = 1, \dots, \zeta(\tau)-1$, let $g_{k}$ and $d_{k}$ be the first and last times at which the contour function $C_{\tau}$ visits the endpoint of the edge $e_{k}$ further from the root $\emptyset$. For $t \geq 0$, we associate to $e_{k}$ the chord $c_{k}(t) \coloneqq [e^{-2\pi i g_{k}/2\zeta(\tau)}, e^{-2\pi i d_{k}/2\zeta(\tau)}] \subset \bar{\mathbb{D}}$ whenever $\gamma_{k} \leq t$, and otherwise we set $c_{k}(t) = \mathbb{S}^{1}$. We define the dynamic lamination-valued process $(\mathbb{L}_{t}^{\rm d}(\tau), t \geq 0)$ by letting
\begin{align*}
\mathbb{L}^{\rm d}_{t}(\tau) \coloneqq \mathbb{S}^{1} \cup \bigcup_{k=1}^{\zeta(\tau)-1} c_{k}(t), \hspace*{3mm} \text{for} \hspace*{2mm} t \in [0,\infty].
\end{align*}  
\end{definition}

The process $(\mathbb{L}^{\rm d}_{t}(\tau), t \geq 0)$ also interpolates between $\mathbb{S}^{1}$ and $\mathbb{L}(\tau) \coloneqq \mathbb{L}^{\rm d}_{\infty}(\tau)$. Furthermore, it is possible to couple $(\mathbb{L}_{t}(\tau), t \geq 0)$ and $(\mathbb{L}_{t}^{\rm d}(\tau), t \geq 0)$ in a natural way so that, under mild assumptions, they are asymptotically close. Indeed, with the same notation as in Definition \ref{Def3}, since the i.i.d.\ exponential variables $(\gamma_k, 1 \leq k \leq \zeta(\tau)-1)$ are a.s.\ distinct, their ordering induces a uniform ordering on the edges of $\tau$. In what follows, we always implicitly consider that we work under this coupling.

\begin{proposition}
\label{prop:closeprocesses}
Let $(\tau_{n}, n \geq 1)$ be a sequence of random rooted plane trees and $(a_{n}, n \geq 1)$ be a sequence of positive real numbers such that $a_{n} \rightarrow \infty$ and $\zeta(\tau_{n})/a_{n} \rightarrow \infty$, as $n \rightarrow \infty$. Then, under the coupling defined above, for any $\varepsilon >0$,
\begin{align*}
\lim_{n \rightarrow \infty} \mathbb{P} \left( d_{\rm Sk}^{\mathbb{L}} \left( \left(\bL_{\frac{ta_n}{\zeta(\tau_{n})}}^{\rm d}(\tau_{n} ), t \geq 0 \right), ( \bL_{ta_n}( \tau_{n}), t \geq 0) \right) > \varepsilon  \right)= 0.
\end{align*}
\end{proposition}

\begin{proof}
For $1 \leq k \leq \zeta(\tau_{n})-1$, let $\kappa_{n,k}$ be the time at which the $k$-th chord is added in $\left(\bL_{ta_{n}/\zeta(\tau_{n})}^{\rm d}(\tau_{n} ), t \geq 0 \right)$. Then, 
\begin{align*}
\bL_{ta_n}( \tau_{n}) =  \bL_{\frac{ \kappa_{n, \lfloor ta_{n} \rfloor}a_n}{\zeta(\tau_{n})}}^{\rm d}(\tau_{n} ), \hspace*{3mm} \text{for all} \hspace*{2mm} t \geq 0.
\end{align*}

\noindent Thus, our claim follows from \cite[Theorem 3.9]{Billi1999}, \cite[Theorem 3.1]{Whitt1980} and \cite[Theorem 1.14 in Chapter VI]{jacod2003} provided that, for each $t \geq 0$,
\begin{equation}
\label{eq:tau2}
\sup_{s \in [0,t]} \left| \kappa_{n, \lfloor s a_n \rfloor} - s \right| \xrightarrow[ ]{\mathbb{P}} 0, \hspace*{3mm} \text{as} \hspace*{2mm} n \rightarrow \infty.
\end{equation}

Let $\gamma_1, \dots, \gamma_{\zeta(\tau_{n})-1}$ be the i.i.d.\ exponential random variables of parameter $1$ of Definition \ref{Def3}, and define the process
\begin{eqnarray*}
N_{n}(s) = \frac{1}{a_{n}} \sum_{i=1}^{\zeta(\tau_{n})-1} \mathbf{1}_{\{E_{i} \leq s a_{n}/\zeta(\tau_{n}) \}}, \hspace*{2mm} \hspace*{2mm} s \geq 0. 
\end{eqnarray*}

\noindent An application of the Chebyshev inequality shows that, for $s \geq 0$, $N_{n}(s) \rightarrow s$, in probability, as $n \rightarrow \infty$. On the other hand, observe that $\kappa_{n, \lfloor s a_n \rfloor} = \inf\{u \geq 0: a_{n} N_{n}(u) \geq \lfloor s a_n \rfloor \}$. Then, by inversion, we have that $\kappa_{n, \lfloor s a_n \rfloor}  \rightarrow s$, in probability, as $n \rightarrow \infty$. Moreover, since $\kappa_{n, \lfloor s a_n \rfloor}$ is non-decreasing as a function of $s$, we obtain (\ref{eq:tau2}); see e.g. \cite[Lemma 2.2 in Chapter 5]{Gut2009}.
\end{proof}

\subsubsection{Reduced tree and reduced lamination built from discrete trees}

Fix $q \geq 1$ (an integer) and let $u_{1}, \dots, u_{q}$ be $q$ i.i.d.\ uniform random vertices of a rooted plane tree $\tau$. The reduced tree $\tau^{(q)}$ of $\tau$ is obtained by keeping only the root $\emptyset$ of $\tau$, these $q$ vertices and the branching points (if any), i.e.\ the vertices $w \in \tau$ such that $\llbracket \emptyset, u_{i} \rrbracket \cap \llbracket \emptyset, u_{j} \rrbracket = \llbracket \emptyset, w \rrbracket$ for some $1 \leq i < j \leq q$. Then one puts an edge between two vertices of $\tau^{(q)}$ if one is the ancestor of the other in $\tau$, and there is no other vertex of $\tau^{(q)}$ inbetween. The length of an edge $e$ in $\tau^{(q)}$ is defined as the number of edges between the vertices of $\tau$ corresponding to the endpoints of $e$. The tree $ \tau^{(q)}$ is rooted at $\emptyset$ and has a plane structure induced by that of $\tau$; see Figure \ref{FigureReduce}. Note that its number of vertices is a priori random. 

\begin{figure}[!htb]
\center
\begin{tabular}{c c c}
\begin{tikzpicture}
\draw (-1.25,3) -- (-1,2) -- (-.5,1) -- (0,0) -- (.5,1) -- (1,2) -- (1,3) (-.75,3) -- (-1,2) (-.5,2) -- (-.5,1) -- (-.2,2) (.5,2) -- (.5,1) -- (.2,2) -- (0,3) (0,0) -- (0,1) (.2,2) -- (.4,3);
\draw[fill=black] (-1.25,3) circle (.1);
\draw[fill=black] (-1,2) circle (.1);
\draw[fill=black] (-.5,1) circle (.1);
\draw[fill=black] (0,0) circle (.1);
\draw[fill=black] (.5,1) circle (.1);
\draw[fill=red] (-.5,2) circle (.1);
\draw[fill=black] (1,2) circle (.1);
\draw[fill=red] (1,3) circle (.1);
\draw[fill=red] (-.75,3) circle (.1);
\draw[fill=black] (.5,2) circle (.1);
\draw[fill=black] (-.2,2) circle (.1);
\draw[fill=black] (.4,3) circle (.1);
\draw[fill=black] (0,3) circle (.1);
\draw[fill=black] (0,1) circle (.1);
\draw[fill=red] (.2,2) circle (.1);
\end{tikzpicture}
&
\begin{tikzpicture}
\draw[white] (0,0) -- (3,0);
\end{tikzpicture}
&
\begin{tikzpicture}
\draw (-.5,1) -- (-.25,2) (-.75,2) -- (-.5,1) -- (0,0) -- (.5,1) -- (.2,2)  (.5,1) -- (1,2);
\draw (-.5,.5) node{$1$};
\draw (.5,.5) node{$1$};
\draw (.15,1.5) node{$1$};
\draw (-.85,1.5) node{$2$};
\draw (1,1.5) node{$2$};
\draw (-.2,1.5) node{$1$};
\draw[fill=black] (0,0) circle (.1);
\draw[fill=black] (-.5,1) circle (.1);
\draw[fill=black] (.5,1) circle (.1);
\draw[fill=black] (-.25,2) circle (.1);
\draw[fill=black] (-.75,2) circle (.1);
\draw[fill=black] (-.25,2) circle (.1);
\draw[fill=black] (.2,2) circle (.1);
\draw[fill=black] (1,2) circle (.1);
\end{tikzpicture}
\end{tabular}
\caption{A tree $\tau$ with $4$ marked vertices, and the reduced tree $\tau^{(4)}$ with its edge lengths.}
\label{FigureReduce}
\end{figure}
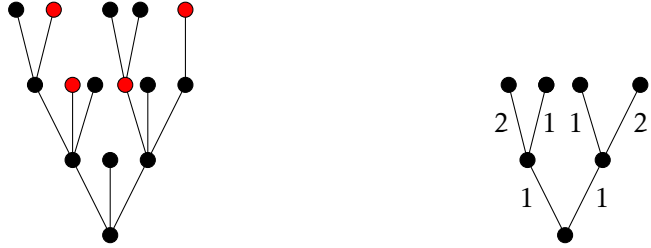

The notion of reduced tree naturally translates in the lamination setting into the notion of reduced lamination. Suppose that $\tau^{(q)}$ has exactly $q$ leaves. Let $u_{0,q} \coloneqq \emptyset$ and  $u_{1,q}, \dots, u_{q,q}$ be the $q$ leaves of $\tau^{(q)}$ listed in lexicographical order.  Let $a_{1}, \dots, a_{q}$ be $q$ i.i.d.\ uniform random points on the unit circle $\mathbb{S}^{1}$. We define $a_{0,q}=1$, and let $a_{1,q}, \dots, a_{q,q}$ be these $q$ points sorted in clockwise order. Set $A_{q} = \{a_{0,q}, a_{1,q}, \dots, a_{q,q} \}$ and $E_{q} = \{ u_{0,q}, u_{1,q}, \dots, u_{q,q} \}$. Observe that removing any edge of $\tau^{(q)}$ splits $E_{q}$ into two subsets, which are made of consecutive elements of $E_{q}$ in lexicographical order (up to cyclic shift), corresponding to two subsets of $A_{q}$ of consecutive points. We now associate to a reduced tree $\tau^{(q)}$ a lamination $\bL(\tau^{(q)})$ as follows:

\begin{definition}[Discrete reduced laminations] \label{Def4}
By convention, if $\tau^{(q)}$ does not have exactly $q$ leaves, we set $\bL(\tau^{(q)}) \coloneqq \bS^1$. Otherwise, for each $0 \leq i \leq j \leq q$, denote by $A_{i,j}$ the event that there exists an edge in $\tau^{(q)}$ splitting the set $E_{q}$ into $\{u_{i+1,q}, \dots, u_{j,q} \}$ and $E_{q} \backslash \{ u_{i+1,q}, \ldots, u_{j,q} \}$. Define $y_{i,q} := e^{-2i\pi \left(a_{i,q}+a_{i+1,q}\right)/2}$ and let $$\bL(\tau^{(q)}) := \bS^1 \cup \{ [y_{i,q}, y_{j,q}] \, | \, \, A_{i,j} \text{ holds}\}.$$ 
\end{definition} 

It is clear that $\bL(\tau^{(q)})$ is a lamination satisfying the following property: for any $0 \leq i < j \leq q$, there exists a chord in $\bL(\tau^{(q)})$ between the open arcs $(a_{i,q}, a_{i+1,q})$ and $(a_{j,q}, a_{j+1,q})$ (with the convention that $a_{q+1,q}=a_{0,q}$), if and only if $A_{i,j}$ holds.

We then associate to $\tau^{(q)}$ a random lamination-valued process. Recall that $\ell(e)$ is defined as the number of edges between the vertices of $\tau$ corresponding to the endpoints of $e$. Equip the edges of $\tau$ with i.i.d.\ exponential random variables of parameter $1$, and for each edge $e$ of $\tau^{(q)}$, denote by $\gamma_e$ the minimum of the $\ell(e)$ exponential random variables associated to the edges of $\tau$ between the endpoints of $e$. In particular, $(\gamma_{e}, e \in \tau^{(q)})$ is a sequence of independent exponential random variables of respective parameters $(\ell(e), e \in \tau^{(q)})$. 

\begin{definition}[Discrete reduced lamination-valued process] \label{Def5}
Consider the lamination $\mathbb{L}(\tau^{(q)})$. We define the reduced lamination-valued process $(\mathbb{L}_{t}^{(q)}(\tau), t \geq 0)$ from $\mathbb{L}(\tau^{(q)})$. Specifically, $\mathbb{L}_{t}^{(q)}(\tau)$ is the unit circle $\mathbb{S}^{1}$ if $\tau^{(q)}$ does not have exactly $q$ leaves. However, if $\tau^{(q)}$ has precisely $q$ leaves, then $\mathbb{L}_{t}^{(q)}(\tau)$ is the union of $\mathbb{S}^{1}$ and the set of chords of $\mathbb{L}(\tau^{(q)})$ corresponding to an edge $e$ if and only if $\gamma_{e} \leq t$.
\end{definition}

In particular, if $\tau^{(q)}$ has precisely $q$ leaves, then this process interpolates between $\mathbb{S}^{1}$ ($t=0$) and the lamination $\mathbb{L}(\tau^{(q)})$ ($t \rightarrow \infty$).

\subsection{The continuum setting} \label{continuumlamination}

In this section, we define lamination-valued processes associated to  so-called plane continuum random trees. Let us first recall the notion of $\R$-tree. A metric space $(\mathcal{T},r)$ is an $\mathbb{R}$-tree, if for every $x,y \in \mathcal{T}$: (i) there exists a unique isometry $f_{x,y}: [0,r(x,y)] \rightarrow \mathcal{T}$ such that $f_{x,y}(0)=x$ and $f_{x,y}(r(x,y))=y$; (ii) for any continuous injective function $g: [0,1] \rightarrow \mathcal{T}$ such that $g(0)=x$ and $g(1)=y$, we have $g([0,1])=f_{x,y}([0,r(x,y)])$. The range of the mapping $f_{x,y}$ is the geodesic between $x,y \in \mathcal{T}$ and is denoted by $\llbracket x,y \rrbracket$. A point $x \in \mathcal{T}$ is called a leaf if $\mathcal{T} \backslash \{ x \}$ is connected, and a branching point if $\mathcal{T} \backslash \{ x \}$ has at least three disjoint connected components. We denote by ${\rm Lf}(\mathcal{T})$ the set of leaves of $\mathcal{T}$, by ${\rm Br}(\cT)$ its set of branching points, and by ${\rm Skel}(\mathcal{T}) \coloneqq \mathcal{T} \setminus {\rm Lf}(\mathcal{T})$ its skeleton. The distance $r$ in $\mathcal{T}$ induces a \textit{length measure} $\lambda_{r}$ on ${\rm Skel}(\mathcal{T})$ given by $\lambda_{r}(\llbracket x,y \rrbracket) = r(x,y)$ for all $x,y \in \mathcal{T}$. A rooted $\mathbb{R}$-tree $(\mathcal{T},r,\rho)$ is a $\mathbb{R}$-tree $(\mathcal{T},r)$ with a distinguished point $\rho \in \mathcal{T}$ called the root of $\mathcal{T}$. For $x,y \in \cT$, we define $x \wedge y$ as the most recent common ancestor of $x$ and $y$, that is, the unique $z \in \mathcal{T}$ such that $\llbracket \rho, z \rrbracket = \llbracket \rho, x \rrbracket \cap \llbracket \rho, y \rrbracket$. Finally, for each $x \in \cT$, let
\begin{align*}
\cT(x) \coloneqq \{y \in \mathcal{T}: x \in \llbracket \rho, y \rrbracket \}
\end{align*}
\noindent be the subtree of $\cT$ above $x$ rooted at $x$.

\begin{definition} \label{Def6}
A (rooted) continuum tree is a quadruple $(\mathcal{T},r,\rho, \mu)$, where $(\mathcal{T},r,\rho)$ is a rooted $\mathbb{R}$-tree and $\mu$ is a non-atomic Borel probability measure on $\mathcal{T}$ such that $\mu({\rm Lf}(\mathcal{T})) =1$ and for every non-leaf vertex $x \in \mathcal{T}$,  $\mu(\mathcal{T}(x))  >0$. We call $\mu$ the \textit{mass measure} of $\cT$.
\end{definition}

In \cite{AldousIII}, Aldous makes slightly different definitions of these quantities which, in particular, restricts his discussion to binary trees, but the theory can be easily extended. Note that the definition of a continuum tree implies that the $\mathbb{R}$-tree $\mathcal{T}$ satisfies certain extra properties; for example, ${\rm Lf}(\mathcal{T})$ must be uncountable, have no isolated point and $\lambda_r$ must be $\sigma$-finite. In what follows, $\mathcal{T}$ will always denote a continuum tree $(\mathcal{T},r,\rho, \mu)$.
The following result is well known when the space is locally compact, see e.g. \cite[Lemma $3.1$]{DW07}.

\begin{lemma} \label{lem:bp}
The set of branching points ${\rm Br}(\cT)$ of a continuum tree $\mathcal{T}$ is at most countable. 
\end{lemma}

\begin{proof}
By definition, for any branching point $y$ of $\cT$, all connected components of $\cT \backslash \{y \}$ which do not contain the root have non-zero $\mu$-mass. For all $k \geq 1$, let $B_k$ be the set of branching points of $\cT$ such that at least two connected components of $\cT \backslash \{ y \}$ not containing the root have $\mu$-mass $>1/k$. Then, the number of points in $B_k$ has to be less than $k$. Our claim follows by taking the union over all $k \geq 1$.
\end{proof}

Lemma \ref{lem:bp} will become relevant later when we define a lamination-valued process on a continuum tree $\mathcal{T}$ by adding chords according to a Poisson point process on $\cT \times \R_+$ with intensity measure $\lambda_r \times dt$, where $dt$ is the Lebesgue measure on $\R_+$ (see Section \ref{sec:lamnoncompact}). Since $\lambda_r$ is the length measure on $\mathcal{T}$, Lemma \ref{lem:bp} ensures that no chord is associated with a branching point of $\mathcal{T}$. This is convenient because whenever a chord appears, removing the corresponding points in $\mathcal{T}$ splits it into exactly two connected components, and the connection to the fragmentation process arises from this splitting.

The continuum trees that we consider here are equipped with what we call a compatible total (or linear) order; see \cite{Duq06}.

\begin{definition} \label{Definitionorder}
We say that a relation $\leq$ on a probability space $\mathcal{X}$ is a total order if:
\begin{itemize}
\item For all $x \in \mathcal{X}$, $x \leq x$;
\item For all $x,y \in \mathcal{X}$, if $x \leq y$ and $y \leq x$, then $x=y$;
\item For all $x,y,z \in \mathcal{X}$, if $x \leq y$ and $y \leq z$, then $x \leq z$;
\item The set $\{ (x,y) \in \mathcal{X} \times \mathcal{X} \, : \, x \leq y \}$ is measurable.
\end{itemize}
\noindent A total order on a continuum tree $\cT$ is compatible with $(\cT,\rho, \mu)$ if the following hold:
\begin{itemize}
\item For all $x_1, x_2 \in \cT$, if $x_1 \in \llbracket \rho, x_2 \rrbracket$, then $x_1 \leq x_2$.
\item For $x_1, x_2, x_3 \in \cT$ such that $x_1 \leq x_2 \leq x_3$, if $y$ is the branching point of $x_1$ on the subtree spanned by $\rho, x_2$ and $x_3$ (i.e., $\llbracket \rho, y \rrbracket = \llbracket \rho, x_1 \rrbracket \cap (\llbracket \rho, x_2 \rrbracket \cup \llbracket \rho, x_3 \rrbracket)$), then  $y \in \llbracket \rho, x_2 \rrbracket$.
\item For any distinct $x_{1}, x_{2} \in \cT$ such that $x_{1} < x_{2}$, we have $\mu(\{ x \in \mathcal{T}: x_{1} < x < x_{2} \}) >0$.
\end{itemize}
\end{definition}

In what follows, unless specified, we always consider continuum trees endowed with a compatible linear order $\leq$.  Observe that if $\mu$ satisfies the last property in Definition \ref{Definitionorder}, then its topological support ${\rm Supp}(\mu)$ is equal to $\mathcal{T}$. In particular, by adapting e.g.\ \cite[Proposition 2.6]{Duq06}, a compatible linear order corresponds to a certain choice of orderings of the $I_x$'s, $x \in {\rm Br}(\cT) \cup \{ \rho \}$, where $I_x$ stands for the set of connected components of $\cT \backslash \{ x \}$ that do not contain the root $\rho$. By a small abuse of notation, we call \textit{plane continuum tree} a continuum tree endowed with a compatible linear order, that is, $(\cT, \leq)$, and call $\leq$ the lexicographical order on $\cT$.

\subsubsection{Laminations associated to plane continuum trees} \label{sec:lamnoncompact}

We show here how to associate to a plane continuum tree $(\cT, \leq)$ a lamination-valued process $(\bL_t(\cT), t \geq 0)$. In particular, this construction is valid even when the tree $\cT$ is not compact. For $x \in \cT$, set
\begin{align} \label{LeftRight}
G(x) \coloneqq  \mu \left( \{ y \in \cT: y \leq x \} \right), \quad  D(x) \coloneqq  \mu \left( \{ y \in \cT: y \leq x \} \cup \cT(x) \right),
\end{align}
\noindent and define the chord $c_x := [e^{-2i\pi G(x)}, e^{-2i\pi D(x)}]$. Note that the set $\{ y \in \cT: y \leq x \}$ is a Borel set (see for example, the representation in \cite[Proof of Proposition 2.10]{Duq06} which is valid even for non-compact plane continuum trees). Note also that if $x, x^{\prime} \in \mathcal{T}$ such that $x < x^{\prime}$, then the last property in Definition \ref{Definitionorder} implies that  $G(x) < G(x^{\prime})$. 

\begin{proposition}
For any plane continuum tree $(\cT, \leq)$, the set 
\begin{align*}
\bL(\cT) = \bS^1 \cup \overline{\bigcup_{x \in \cT} c_x}.
\end{align*}
\noindent is a lamination.
\end{proposition}

\begin{proof}
The set $\bL(\cT)$ is clearly closed. The only thing that we need to prove is that the chords do not cross. Take $x,y \in \cT$ such that $G(x) < G(y) < D(x)$. We want to prove that $D(y) \leq D(x)$. Since $\leq$ is a total order and $G(x) < G(y)$, we have that $x \leq y$. First, we claim that $x$ is an ancestor of $y$. Indeed, assume that $x \wedge y \neq x$ and take $z \in \cT(x)$. By definition, $x \leq z$. If $x \leq y \leq z$ then by definition $x := x \wedge z \in \llbracket \rho, y \rrbracket$, which contradicts our assumption. Hence, for all $z \in \cT(x)$, $z < y$. But then $D(x) \leq G(y)$ which contradicts our initial assumption. Therefore, $x$ is an ancestor of $y$. Now take a point $w \in \cT$ such that $x < w \leq y$. By Definition \ref{Definitionorder}, the branching point of $x$ on the subtree spanned by $\rho, w$ and $y$ is on $\llbracket \rho, w \rrbracket$. In particular, $w \in \cT(x)$. This shows that $\mu( \{ x^{\prime} \in \cT: x < x^{\prime} \leq y \}) \leq \mu(\cT(x))-\mu(\cT(y))$ (since $y \leq w$ for all $w \in \cT(y)$). This can be rephrased in $D(y) \leq D(x)$. Thus, $c_x$ and $c_y$ do not cross inside the disk.
\end{proof}

Let also $\Pi$ be a Poisson point process on $\cT \times \R_+$, with intensity measure $\lambda_r \times dt$, where $dt$ is the Lebesgue measure on $\R_+$. For any $t \geq 0$, set $\Pi_t := \{ x \in \cT : \exists s \leq t, (x,s) \in \Pi \}$. We define a lamination-valued process $(\bL_t(\cT), t \geq 0)$ by letting, for all $t \geq 0$:
\begin{align} \label{ContinuumLamination}
\bL_t(\cT) := \bS^1 \cup \overline{\bigcup_{x \in \Pi_t} c_x}.
\end{align}

\subsubsection{Compact continuum trees and excursion-type functions} \label{CompactSec}

We consider here a particular case of plane continuum trees, the compact ones. A common way to construct compact $\mathbb{R}$-trees is from continuous excursion-type functions, i.e.\ continuous functions $f:[0,1] \rightarrow \mathbb{R}_{+}$ such that $f(0) = f(1) =0$ and $f(x) \geq 0$ for all $0 \leq x \leq 1$. Let $f$ be such a function, consider the pseudo-distance on $[0,1]$,
\begin{eqnarray*}
r_{f}(x,y) \coloneqq f(x) + f(y) - 2 \inf_{z \in [x \wedge y, x \vee y]} f(z), \hspace*{3mm} \text{for} \hspace*{3mm} x,y \in [0,1],
\end{eqnarray*}

\noindent and define an equivalence relation on $[0,1]$ by setting $x \sim_{f} y$ if and only if $r_{f}(x,y) = 0$. The image of the projection $p_{f}:[0,1] \rightarrow [0,1] \setminus\sim_{f}$ endowed with the pushforward of $r_{f}$ (again denoted $r_{f}$), i.e. $\mathcal{T}_{f} = (\mathcal{T}_{f}, r_{f}, \rho_{f}) \coloneqq (p_{f}([0,1]), r_{f}, p_{f}(0))$, is a rooted plane $\mathbb{R}$-tree, with the linear order induced by the usual order on $[0,1]$ (we say that, for $x,y \in \cT_f$, $x \leq y$ if $\inf \{ p_f^{-1}(x) \} \leq \inf \{ p_f^{-1}(y) \}$); see \cite[Lemma 3.1]{EW2006}. In particular, $(\mathcal{T}_{f}, r_{f})$ is a compact and connected metric space. Conversely, it has been noted in \cite[Remark following Theorem 2.2]{Legall22005} (see also \cite[Corollary 1.2]{Duq06}) that for every compact $\mathbb{R}$-tree $(\mathcal{T}, r)$ there exists a continuous excursion-type function $f:[0,1] \rightarrow \mathbb{R}$ such that $(\mathcal{T}, r)$ and $(\mathcal{T}_{f}, r_{f})$, are isometric. We can endow $\mathcal{T}_{f}$ with the probability measure $\mu_{f}$ given by the pushforward of the Lebesgue measure on $[0,1]$ under the projection $p_{f}$. Suppose furthermore that the set of one-sided local minima of $f$ has Lebesgue measure $0$ (recall that $x \in [0,1]$ is a one-sided local minimum of $f$ if there exists $\varepsilon>0$ such that $f(x) = \inf\{f(y): x \leq y \leq x+\varepsilon \}$ or $f(x) = \inf\{f(y): x-\varepsilon \leq y \leq x \}$). Then, $\mu_{f}$ is a non-atomic measure and $\mu_{f}(\text{Lf}(\mathcal{T}_{f}))=1$; see \cite[Proof of Theorem 13]{AldousIII}. Moreover, $\mathcal{T}_{f} = (\mathcal{T}_{f}, r_{f}, \rho_{f}, \mu_{f})$ is a (rooted) continuum tree. 

In this setting, we can construct a lamination $\bL(f)$ and a lamination-valued process $(\bL_{t}(f), t \geq 0)$ associated to the continuous excursion-type function $f$ (and thus to $\mathcal{T}_{f}$). Let us recall the definition of $(\bL_{t}(f), t \geq 0)$ and refer to \cite{The19} for further details. First, define the epigraph of $f$ as the set of points below its graph, that is,
\begin{align*}
\cEG(f) \coloneqq \left\{ (x,y) \in \R^2: x \in (0,1), 0 \leq y < f(x) \right\}.
\end{align*}

\noindent To each $(x,y) \in \cEG(f)$, associate the chord $c(x,y) \coloneqq [ e^{-2\pi i g(x,y)}, e^{-2 \pi i d(x,y)} ] \in \bar{\mathbb{D}}$, where $g(x,y) \coloneqq \sup \{ z \leq x: f(z)<y \}$ and $d(x,y) \coloneqq \inf \{ z \geq x: f(z)<y \}$. Consider now a Poisson point process $\cN^{f}$ on $\R^2 \times \R_{+}$ with intensity measure
\begin{align*}
\frac{1}{d(x,y)-g(x,y)} \mathbf{1}_{\{ (x,y) \in \cEG(f) \}} dx dy ds.
\end{align*}
\noindent Here, $ds$ denotes the Lebesgue measure on $\R_+$, while $dx$ and $dy$ both denote the Lebesgue measure on $\R$. For $t \geq 0$, consider also the Poisson point process $\cN^{f}_{t}(\cdot) \coloneqq \cN^{f}(\cdot \times [0,t])$ on $\cEG(f)$ and construct the lamination-valued process $(\bL_{t}(f),t \geq 0)$ associated to $f$ as follows. For all $t \geq 0$,
\begin{align*}
\bL_{t}(f) = \bS^1 \cup \overline{\bigcup_{(x,y) \in \cN_{t}^{f}} c(x,y)}.
\end{align*}

\noindent Clearly, this process is non-decreasing for the inclusion. Moreover, define 
\begin{align*}
\bL_{\infty}(f) \coloneqq  \bL(f) = \overline{\bigcup_{t \geq 0} \bL_{t}(f)}.
\end{align*}
It is straightforward that chords of the lamination $\bL_\infty(f)$ are in bijection with points of $\cT_f$. In particular it is useful to define the process directly from $\cT_f$. For all $x \in \cT_f$, recall the definition of $G(x)$ in \eqref{LeftRight} and observe that $G(x) := \inf \{p_f^{-1}(x) \}$.

\begin{proposition} \label{identitycouplingLam}
We have that $( \bL_{t}(\cT_{f}), t \geq 0 ) \overset{d}{=} (\bL_{t}(f), t \geq 0 )$.
\end{proposition}

\begin{proof}
The idea consists in coupling $( \bL_{t}(\cT_{f}), t \geq 0)$ and $(\bL_{t}(f), t \geq 0)$. For any $(x,y) \in \cN^{f}$, let $w(x,y) \in \cT_{f}$ be the equivalence class of $g(x,y)$ with respect to $\sim_{f}$. Then, the chord $c(x,y)$ is exactly the chord $c_{w(x,y)}$. Thus, we only need to check that the image of $\cN^f$ under the projection $p_{f}$ is a Poisson point process on ${\rm Skel}(\cT_f)$ with the correct intensity. To this end, observe that, for any $x^{\prime}, y^{\prime} \in \cT_{f}$, we have that
\begin{align*}
\int_{E(x^{\prime}, y^{\prime})} \frac{1}{d(x,y)-g(x,y)} \mathbf{1}_{\{ (x,y) \in \cEG(f) \}} dx dy = \lambda_{r}(\llbracket x^{\prime}, y^{\prime} \rrbracket),
\end{align*}
\noindent where $E(x^{\prime}, y^{\prime}) \coloneqq \{ (x,y) \in \R^2: w(x,y) \in \llbracket x^{\prime}, y^{\prime} \rrbracket \}$ (this representation can be found, for example, in \cite[Example 4.34]{Evans60}). The result follows.
\end{proof}

\begin{remark}
In fact, under the natural coupling defined in the proof of Proposition \ref{identitycouplingLam}, the lamination-valued processes $( \bL_{t}(\cT_{f}), t \geq 0 )$ and $(\bL_{t}(f), t \geq 0 )$ are almost surely equal.
\end{remark}

In particular, $(\bL_{t}(\cT_f), t \geq 0)$ is non-decreasing and it interpolates between $\bS^1$ ($t=0$) and $\bL(\cT_{f})$ ($t \rightarrow \infty$). Indeed, by \cite[Proposition $2.2$ (ii)]{The19}, we have that $\bL_t(\cT_{f}) \rightarrow \bL(\cT_{f})=\bL_\infty(f)$, as $t \rightarrow \infty$, on  $(\mathbb{L}(\bar{\mathbb{D}}), d_{\rm H})$, whenever $\cT$ is compact.

\subsubsection{Reduced tree and reduced lamination built from continuum trees}

For $q \geq 1$ (an integer), let $x_{1}, \dots, x_{q}$ be $q$ i.i.d.\ random leaves of a  plane continuum tree $\mathcal{T}$ sampled from its mass measure $\mu$. Observe that they are a.s.\ all distinct, and set $E_q \coloneqq \{ x_{i}: 1 \leq i \leq q \} \cup \{ \rho \}$. The reduced tree $\cT^{(q)}$ of $\mathcal{T}$ is the rooted plane tree with edge lengths whose vertices are the leaves $x_1, \ldots, x_q$, the root $\rho$ of $\mathcal{T}$ and all branching points. The length of an edge is simply the length measure of the unique geodesic path in $\cT$ between the corresponding endpoints. 

We can also define the notion of reduced lamination and reduced lamination-valued process in the continuum setting. For $x \in \mathcal{T}$, recall the definition of $G(x)$ in \eqref{LeftRight}. Let $x_{0,q} = \rho$ be the root of $\cT^{(q)}$ and $x_{1,q}, \dots x_{q,q}$ be its $q$ leaves listed in lexicographical order. Set $a_{j,q} \coloneqq e^{-2i\pi G(x_{j,q})}$, for $0 \leq j \leq q$ (in particular, $a_{0,q}=1$). The following result must be clear since $\mu$ is non-atomic.

\begin{lemma}
\label{lem:measure}
If $x \in {\rm Lf}(\mathcal{T})$ is distributed according to $\mu$, then $G(x)$ is uniformly distributed on $[0,1]$.
\end{lemma}

As a consequence, for all $q \geq 1$, if the $q$ leaves of $\mathcal{T}$ are sampled in an i.i.d.\ way according to $\mu$, then $a_{1,q}, \ldots, a_{q,q}$ are the order statistics of $q$ i.i.d.\ uniform variables on the unit circle.

Analogously to the discrete case, we now associate to $\cT^{(q)}$ a lamination $\bL(\cT^{(q)})$ which satisfies the following property: for $0 \leq i<j \leq q$, there exists a chord in $\bL(\cT^{(q)})$ between open arcs $(a_{i,q}, a_{i+1,q})$ and $(a_{j,q}, a_{j+1,q})$ (with the convention that $a_{q+1,q}=a_{0,q}$) if and only if there exists an edge in $\cT^{(q)}$ splitting $E_{q}$ into $\{ x_{i+1, q}, \dots, x_{j, q} \}$ and $E_{q} \backslash \{ x_{i+1, q}, \dots, x_{j,q} \}$ (denote this last event by $A_{i,j}$).

\begin{definition}[Continuum reduced lamination] \label{Def8}
For each $0 \leq i \leq q$, let $y_{i,q} := e^{-2i\pi \left(a_{i,q}+a_{i+1,q}\right)/2}$. We let $$\bL(\cT^{(q)}) := \bS^1 \cup \{ [y_{i,q}, y_{i+1,q}] \, | \, \, A_{i,j} \text{ holds}\}.$$
\end{definition}

It is clear that $\bL(\cT^{(q)})$ is a lamination satisfying the property that we want. Let us now state and prove a result which will be useful in what follows.

\begin{lemma} \label{lem:repartitionunifcircle}
For all $a,a' \in \bS^1$, let $d(a,a^{\prime})$ denote the length of the shortest arc from $a$ to $a^{\prime}$ in $\mathbb{S}^{1}$. For all $\varepsilon>0$,
\begin{align*}
\lim_{q \rightarrow \infty} \mathbb{P}\left( \sup_{1 \leq j \leq q} d \left(a_{j,q}, e^{-2\pi i j/q}\right) < \varepsilon \right) = 1.
\end{align*}
\end{lemma}

\begin{proof}
Fix $\varepsilon>0$, choose an integer $K \geq 1$ such that $2/K<\varepsilon$ and take $q \geq K^{8}$. Let $A_q := \{a_{j,q}, 0 \leq j \leq q \}$. Split the unit circle into the $K$ arcs of the form, $(e^{-2\pi i (k-1)/K}, e^{-2\pi i k/K})$ for $1 \leq k \leq K$. For $1 \leq j \leq q$, almost surely no point of the form $a_{j,q}$ is one of the endpoints of these arcs. For any $q \geq 2$ and $1 \leq k \leq K$, denote by $A_q(k)$ the set $A_{q} \cap (e^{-2\pi i (k-1)/K}, e^{-2\pi i k/K})$, and $M_{q}(k)$ the number of points in $A_{q}(k)$. In particular, $M_{q}(k)$ is distributed as a binomial random variable with parameters $(q,1/K)$. Hoeffding's inequality implies that, for $q$ large enough:
\begin{eqnarray*}
\mathbb P ( |M_{q}(k) - q/K| \geq q^{3/4}) \leq 2 e^{ - 2 \sqrt{q}}.
\end{eqnarray*}

\noindent Hence, the probability that $|M_{q}(k) - q/K| \leq q^{3/4}$, for all $1 \leq k \leq K$, is at least $1-2Ke^{ - 2 \sqrt{q}}$ and our claim follows by choosing $K$ large enough (depending on $\varepsilon$). Indeed, suppose that $|M_{q}(k) - q/K| \leq q^{3/4}$, for all $1 \leq k \leq K$. Then, for any $1\leq k \leq K$ and any integer $(k-1)q/K < j \leq kq/K$, we necessarily have that $a_{j,q}$ is in $A_q(k-1), A_q(k)$ or $A_q(k+1)$ (with the convention that $A_{q}(0) = A_{q}(K)$ and $A_{q}(K+1) = A_{q}(1)$) and since $e^{-2\pi i j/q} \in A_q(k)$, the result in Lemma \ref{lem:repartitionunifcircle} holds.
\end{proof}

We now associate to $\cT^{(q)}$ a lamination-valued process. Recall that $\Pi$ denotes a Poisson point process on $\cT \times \R_+$, with intensity measure $\lambda_r \times dt$, where $dt$ is the Lebesgue measure on $\R_+$. For each edge $e \in \cT^{(q)}$, let $\gamma_{e}$ be the first time at which a point of $\Pi$ falls on the geodesic of $\mathcal{T}$ that corresponds to $e$. In particular, $\gamma_{e}$ is an exponential random variable of parameter $\ell(e)$ the length of $e$, and $(\gamma_{e}, e \in \mathcal{T}^{(q)})$ is a collection of independent random variables. 

\begin{definition}[Continuum reduced lamination-valued process] \label{Def9}
We define the process $(\bL_{t}^{(q)}(\cT), t \geq 0)$ from $\bL(\cT^{(q)})$ by letting $\bL_{t}^{(q)}(\cT)$ be the union of the unit circle $\mathbb{S}^{1}$ and the set of chords of $\bL(\cT^{(q)})$ corresponding to an edge $e$ if and only if $\gamma_{e} \leq t$.
\end{definition}

It turns out that these reduced processes actually approximate the usual lamination-valued process $(\bL_{t}(\mathcal{T}), t \geq 0)$.

\begin{proposition}
\label{prop:cvlamredcont}
Let $\mathcal{T}$ be a random rooted plane tree. Then,
\begin{align*}
(\bL_{t}^{(q)}(\mathcal{T}), t \geq 0) \xrightarrow[ ]{d} (\bL_{t}(\mathcal{T}), t \geq 0), \hspace*{3mm} \text{as} \hspace*{2mm}  q \rightarrow \infty, \hspace*{2mm} \text{in} \hspace*{2mm} \mathbf{D}(\mathbb{R}_{+}, \bL(\bar{\mathbb{D}})).
\end{align*}
\end{proposition}

\begin{proof}
We only need to prove that, for all $\varepsilon >0$ and for every $M >0$, 
\begin{eqnarray*}
\lim_{q \rightarrow \infty} \mathbb{P}\left( \sup_{t \in [0,M]} d_{\rm H} \left(\bL_{t}^{(q)}(\mathcal{T}), \bL_{t}(\mathcal{T}) \right) < \varepsilon \right) = 1. 
\end{eqnarray*}

Since the support of $\lambda_{r}$ is ${\rm Skel}(\cT)$, we only consider chords $c_{x}$ that are coded by a point $x \in {\rm Skel}(\cT)$. Moreover, by Lemma \ref{lem:bp}, a.s. no such $x$ is a branching point. It follows from Lemmas \ref{lem:measure} and \ref{lem:repartitionunifcircle} that for fixed $\varepsilon'>0$, we can and will choose $Q >0$ large enough such that, with probability at least $1-\varepsilon'$, for any $q \geq Q$, 
\begin{align}
\label{eq:distances}
\sup_{1 \leq j \leq q} d \left(a_{j,q}, e^{-2\pi i j/q}\right) < \varepsilon'. 
\end{align}

\noindent Fix $M >0$ and consider a point $x \in \Pi_{M}$. There are two cases: either $x$ falls in the geodesic of $\mathcal{T}$ that corresponds to an edge of $\cT^{(q)}$, or not. If $x$ does not fall in such a geodesic, then removing $x$ does not split the set of leaves of $\cT^{(q)}$ and thus necessarily $d_{\rm H}(c_x, \bS^1) < 2\pi (2 \varepsilon' + 1/Q)$ by \eqref{eq:distances}. Now, suppose that $x$ falls in such a geodesic. By definition, since $x \in \Pi_M$, there exists a chord $c' \in \bL_{M}^{(q)}(\cT)$ corresponding to a point in the same edge of $\cT^{(q)}$ as $x$. Thus, there exist two arcs $A_1, A_2$ between clockwise consecutive $a_{j,q}$'s such that $c_x$ and $c'$ connect $A_1$ and $A_2$. Hence, by \eqref{eq:distances} and e.g. \cite[Lemma $5.2$ (i)]{FLT21}, $d_{\rm H}(c',c_x) < 2\pi (2 \varepsilon' + 1/Q)$. Finally, our claim follows by choosing $\varepsilon'$ so that $2\pi (2 \varepsilon' + 1/Q) < \varepsilon$.
\end{proof}

\subsection{Convergence of the process of laminations} \label{proofLaminationT}

In this section, we prove Theorem \ref{thm:cvlamproc}, stating the equivalence between the planar version of the Gromov-weak convergence of trees and the convergence of their associated lamination-valued processes. We first need to introduce some notation and establish some additional geometric properties of reduced trees.

A rooted metric measure space is a quadruple $\mathcal{X} =(\mathcal{X}, r, \rho, \mu)$,  where $(\mathcal{X}, r)$ is a metric space such that $(\text{Supp}(\mu), r)$ is complete and separable, the so-called sampling measure $\mu$ is a finite measure on $(\mathcal{X}, r)$ and $\rho \in \mathcal{X}$ is a distinguished point which is referred to as the root; the support $\text{Supp}(\mu)$ of $\mu$ is defined as the smallest closed set $\mathcal{X}_{0} \subseteq \mathcal{X}$ such that $\mu(\mathcal{X}_{0}) = \mu(\mathcal{X})$. Two rooted metric measure spaces $(\mathcal{X}, r, \rho, \mu)$ and $(\mathcal{X}^{\prime}, r^{\prime}, \rho^{\prime}, \mu^{\prime})$ are said to be equivalent if there exists an isometry $\phi: {\rm Supp}(\mathcal{\mu}) \cup \{\rho\} \rightarrow {\rm Supp}(\mu^{\prime})\cup \{\rho^{\prime}\}$ such that $\phi(\rho) = \rho^{\prime}$ and $\phi_{\ast} \mu = \mu^{\prime}$, where $\phi_{\ast} \mu$ is the pushforward of $\mu$ under $\phi$. We denote by $\mathbb{K}_{0}$ the space of rooted metric measure spaces. We consider that $\mathbb{K}_{0}$ is equipped with the Gromov-weak topology; see Gromov's book \cite{Grom2007} or \cite{Greven2009, LWV}. In particular, the Gromov-weak topology is metrized by the so-called pointed Gromov-Prohorov metric $d_{\rm pGP}$. Moreover, $(\mathbb{K}_{0}, d_{\rm pGP})$ is a complete and separable metric space; see \cite[Proposition 2.6]{LWV}. Let us give a simple characterization for convergence in the Gromov-weak topology, see e.g.\ \cite{Greven2009, LWV}. For each $n \in \mathbb{N} \cup \{ \infty\}$, consider a rooted metric measure space $\mathcal{X}_{n} = (\mathcal{X}_{n}, r_{n}, \rho_{n}, \mu_{n})$, set $\xi_{n}(0)= \rho_{n}$ and $(\xi_{n}(i), i \geq 1)$ i.i.d.\ random variables sampled according to $\mu_{n}$. The convergence $\mathcal{X}_{n} \rightarrow \mathcal{X}_{\infty}$, as $n \rightarrow \infty$, for the Gromov-weak topology is equivalent to the convergence in distribution of the matrices
\begin{eqnarray} \label{GromovWC}
(r_{n}(\xi_{n}(i), \xi_{n}(j)): 0 \leq i,j \leq q) \xrightarrow[ ]{d} (r_{\infty}(\xi_{\infty}(i), \xi_{\infty}(j)): 0 \leq i,j \leq q), \hspace*{3mm} \text{as} \hspace*{2mm}  n \rightarrow \infty,
\end{eqnarray}

\noindent for every fixed integer $q \geq 1$. By Gromov's reconstruction theorem 
\cite[Subsection 3$\frac{1}{2}$.7]{Grom2007}, the distribution of $(r_{n}(\xi_{n}(i), \xi_{n}(j)):  i,j \geq 0)$ characterizes (the equivalence class of) $\mathcal{X}_{n}$. 

We consider the following notion of convergence that is in a sense a planar version of \eqref{GromovWC}.
\begin{definition}
\label{def:planarGW}
For $n \in \mathbb{N} \cup \{ \infty\}$, let $\mathcal{X}_{n} = (\mathcal{X}_{n}, r_{n}, \rho_{n}, \mu_{n})$ be a rooted metric measure space such that $\mathcal{X}_{n}$ is a totally ordered set (or linearly ordered set). Set, for all $n$, $\xi_{n}(0)= \rho_{n}$ and $(\xi_{n}(i), i \geq 1)$ i.i.d.\ random variables sampled according to $\mu_{n}$. Then, we say that $\mathcal{X}_n \rightarrow \mathcal{X}_{\infty}$, as $n \rightarrow \infty$, in the planar Gromov-weak sense if, for all fixed integer $q \geq 1$, 
\begin{eqnarray} \label{GromovWCplanar}
(r_{n}(\xi'_{n}(i), \xi'_{n}(j)): 0 \leq i,j \leq q) \xrightarrow[ ]{d} (r_{\infty}(\xi'_{\infty}(i), \xi'_{\infty}(j)): 0 \leq i,j \leq q), \hspace*{3mm} \text{as} \hspace*{2mm}  n \rightarrow \infty,
\end{eqnarray}
where $(\xi'_n(i))_{0 \leq i \leq q}, (\xi'_{\infty}(i))_{0 \leq i \leq q}$ are the respective order statistics of $(\xi_{n}(i))_{0 \leq i \leq q}, (\xi_{\infty}(i))_{0 \leq i \leq q}$.
\end{definition}

\noindent The measurability of the total order ensures that the order statistics of $q$ points are well-defined and measurable. Observe that, in particular, convergence in the planar Gromov-weak sense implies convergence for the Gromov-weak topology. A plane continuum tree $\mathcal{T} = (\mathcal{T}, r, \rho, \mu)$ is a particular case of rooted metric measure space. For $q \geq 1$, let $x_{1}, \dots, x_{q}$ be $q$ i.i.d.\ leaves sampled according to $\mu$. Let $x_{0,q} = \rho$ and $x_{1,q}, \dots, x_{q,q}$ be the $q$ leaves $x_{1}, \dots, x_{q}$ of $\mathcal{T}$ listed in lexicographical order. Set $E_{q} = \{x_{0,1}, \dots, x_{q,q}\}$. For $\omega \subset \{0, \dots, q\}$, let $\Xi_{q}(\omega)$ be the set of points in $\mathcal{T}$ (if any) whose removal separates the set $E_{q}$ into $E_{q}^{\omega,1} = \{ x_{k,q}: k \in \omega \}$ and $E_{q}^{\omega,2} = \{ x_{k,q}: k \not \in \omega \}$, such that each of these subsets, $E_{q}^{\omega,1}$ and $E_{q}^{\omega,2}$, is contained within a single connected component of $\mathcal{T}$.

The following two lemmas show that the sets of points whose removal splits the set of leaves of the reduced trees into two given subsets are either empty, or a geodesic corresponding to an edge of the reduced tree.

\begin{lemma} 
\label{lem:distances}
For $q \geq 1$ and $\omega \subset \{0, \dots, q\}$, we have that $\Xi_{q}(\omega)$ is either empty or a geodesic of $\cT$ that corresponds precisely to an edge $e_{\omega}$ in $\mathcal{T}^{(q)}$ of length $\ell(e_{\omega}) = \lambda_{r}(\Xi_{q}(\omega)) >0$. Reciprocally, for every edge $e$ in $\mathcal{T}^{(q)}$ of length $\ell(e)$ there exists $\omega_{e} \subset \{0, \dots, q\}$ such that $\Xi_{q}(\omega_{e})$ is a geodesic of $\cT$ that corresponds precisely to $e$ such that $\lambda_{r}(\Xi_{q}(\omega_{e})) = \ell(e)$. 

Furthermore, for any points $x_1, x_2 \in E_{q}^{\omega,1}$ and $y_1, y_2 \in E_{q}^{\omega,2}$, define
\begin{align*}
f\left( x_1, x_2; y_1,y_2 \right) := \frac{r(x_1,y_1)+r(x_1,y_2)+r(x_2,y_1)+r(x_2,y_2)}{4} - \frac{r(x_1,x_2)+r(y_1,y_2)}{2}.
\end{align*}
Then $\Xi_{q}(\omega)$ is not empty if and only if $g(\omega) \coloneqq \min\{f(x_1,x_2;y_1,y_2): \, x_1,x_2 \in E_{q}^{\omega,1}, y_1, y_2 \in E_{q}^{\omega,2}\}>0$, in which case $g(\omega)=\ell(e_\omega)$.
\end{lemma}

\begin{proof}
For an edge $e$ in $\mathcal{T}^{(q)}$, we denote by $e_{g}$ the geodesic in $\mathcal{T}$ that corresponds to $e$. Let us first prove the first part. Consider $\omega \subset \{ 0, \ldots, q\}$ such that $\Xi_q(\omega)$ is not empty and recall that a point $x \in \Xi_q(\omega)$ splits $E_{q}$ into $E_{q}^{\omega,1}$ and $E_{q}^{\omega,2}$. Observe that $x$ cannnot be a branching point of $\cT^{(q)}$. Then, $x \in e_{g}$ for some edge $e$ of $\cT^{(q)}$. Take $y \in E_{q}^{\omega,1}$ and $y^{\prime} \in E_{q}^{\omega,2}$. By definition, $x \in \llbracket y, y^{\prime} \rrbracket$ (otherwise, $y$ and $y^{\prime}$ are in the same connected component of $\cT \backslash \{ x \}$). Since the geodesic $\llbracket y, y^{\prime} \rrbracket$ is injective and connects two vertices of $\cT^{(q)}$ in $\mathcal{T}$, we have $e_{g} \subset \llbracket y, y^{\prime} \rrbracket$. In particular, for any $x^{\prime} \in e_{g}$ and for any $y \in E_{q}^{\omega,1}$ and $y^{\prime} \in E_{q}^{\omega,2}$, we have that $x^{\prime} \in \llbracket y, y^{\prime} \rrbracket$. Thus, $e_{g} \subset \Xi_q(\omega)$.

Now, assume that there exists another edge $e^{\prime} \neq e$ such that $\hat{y} \in e_{g}^{\prime}$ and $\hat{y} \in \Xi_{q}(\omega)$. Choose $b \in \llbracket x, \hat{y} \rrbracket$ a branching point of $\mathcal{T}^{(q)}$, and let $z \in E_{q}$ such that it is not in the same connected component of $\cT \backslash \{ b \}$ as $x$ nor as $\hat{y}$ (if $b = \rho$, we take instead $z=b$). Assume without loss of generality that $z \in E_{q}^{\omega,1}$. Then, for any $z^{\prime} \in E_{q}^{\omega,2}$, $x$ and $\hat{y}$ both belong to $\llbracket z, z^{\prime} \rrbracket$, which contradicts the fact that $\llbracket z, z^{\prime} \rrbracket$ is a geodesic. Hence, $\Xi_{q}(\omega)=e_{g}$.

Conversely, by definition, it is not difficult to see that every edge of $\cT^{(q)}$ corresponds to one $\Xi_q(\omega)$, for some $\omega \subset \{ 0, \dots, q\}$. 

We now prove the second part. First, assume that $\Xi_{q}(\omega)$ is a geodesic that corresponds to an edge $e_{\omega}$ of $\mathcal{T}^{(q)}$ and let $x,y$ be its endpoints, so that $\ell(e_\omega)=r(x,y)$. For $z \in \Xi_{q}(\omega)$, $\cT \backslash \{z\}$ has two connected components, one of them containing $x$ and the other containing $y$. Without loss of generality, suppose that $x$ is in the connected component of all points of $E_{q}^{\omega,1}$, and $y$ is in the one of all points of $E_{q}^{\omega,2}$. Then, for any $x^{\prime} \in E_{q}^{\omega,1}$ and $y^{\prime} \in E_{q}^{\omega,2}$, we have $r(x^{\prime},y^{\prime})=r(x^{\prime},x)+r(x,y)+r(y,y^{\prime})$. Thus, for all $x_1,x_2 \in E_{q}^{\omega,1}$, $y_1, y_2 \in E_{q}^{\omega,2}$, we get that
\begin{align*}
f(x_1,x_2;y_1,y_2) &= \frac{2r(x_1,x)+2r(x_2,x)+4r(x,y)+2r(y_1,y)+2r(y_2,y)}{4} - \frac{r(x_1,x_2)+r(y_1,y_2)}{2}  \\ 
&= r(x,y)+\frac{r(x_1,x)+r(x_2,x)-r(x_1,x_2)}{2}+\frac{r(y_1,y)+r(y_2,y)-r(y_1,y_2)}{2}.
\end{align*}

\noindent By the triangle inequality, $f(x_1,x_2;y_1,y_2) \geq r(x,y)$. Furthermore, if $\omega$ and $\{ 0,\dots,q \} \backslash \omega$ both have at least two elements, then $x$ and $y$ are branching points of $\mathcal{T}^{(q)}$, and we can find $x_1,x_2 \in E_{q}^{\omega,1}$, $y_1,y_2 \in E_{q}^{\omega,2}$ such that $f(x_1,x_2;y_1,y_2)=r(x,y)$. Otherwise, if $\omega$ is a singleton, say $\{ i \}$, then $x=x_{i,q}$ and we can find two elements $y_1,y_2 \in E_{q}^{\omega,2}$ such that $f(x,x;y_1,y_2)=r(x,y)$. The case where $\{ 0,\dots, q \} \backslash \omega$ is a singleton is handled the same way. Hence, we have proved that if $\Xi(\omega)$ is not empty then $g(\omega)=\ell(e_\omega)>0$. 

Now we assume that $g(\omega)>0$. If $\omega$ is a singleton, say $E_{q}^{\omega,1} \coloneqq \{ x \}$, then for all $y_1, y_2 \in E_{q}^{\omega,2}$ such that $y_{1} \neq y_{2}$ we have
\begin{align*}
f(x,x;y_1,y_2) = \frac{r(x,y_1)+r(x,y_2)-r(y_1,y_2)}{2} \geq g(\omega).
\end{align*}
Hence, by letting $b$ be the branching point of $y_1,y_2$ and $x$ in $\cT^{(q)}$ (that is, $b= \llbracket x, y_1 \rrbracket \cap \llbracket x, y_2 \rrbracket \cap \llbracket y_1, y_2 \rrbracket$), we have $b \neq x$ and $r(x,b)\geq g(\omega)$. Remark that if $x = \rho$ then $x \neq b$ since a point of $\Xi_q(\omega)$ should belong to both $\llbracket \rho, y_1\rrbracket$ and $\llbracket \rho, y_2\rrbracket$. Then if $x$ is a leaf of $\cT^{(q)}$, for any edge $e \in \cT^{(q)}$ and $z \in e_{g}$ such that $r(x,z) \in (0,g(\omega))$, we see that $z \in \Xi_q(\omega)$. If $x=\rho$, any $z \in \llbracket \rho, y \rrbracket$ for some $y \in E_q^{\omega,2}$ such that $r(x,z) \in (0,g(\omega))$, we have that $z \in \Xi_q(\omega)$.

Finally, assume that $\omega$ and $\{ 0,\ldots, q \} \backslash \omega$ are not singletons. Fix $x_1, x_2 \in E_{q}^{\omega,1}$ such that $x_1 \neq x_2$. First, we prove that for all $y \in E_{q}^{\omega,2}$, the  branching points of $x_1,x_2$ and $y$ in $\mathcal{T}^{(q)}$ are the same. Indeed, let $y_1 \neq y_2$ be two elements of $E_{q}^{\omega,2}$ and $z_1,z_2$ their respective  branching points with $x_1$ and $x_2$. Then, if $z_1 \neq z_2$, using the fact that $z_1,z_2 \in \llbracket x_1, x_2 \rrbracket$, we get
\begin{align*}
f(x_1,x_2;y_1,y_2) = -r(z_1,z_2)<0.
\end{align*}

\noindent Let $a(x_1,x_2)$ be therefore the  branching point of $x_1,x_2$ and $y$ in $\mathcal{T}^{(q)}$, for all $y \in E_{q}^{\omega,2}$. Symmetrically, for any $y_1,y_2 \in E_{q}^{\omega,2}$, let $b(y_1,y_2)$ be the  branching point of $y_1,y_2$ and $x$ in $\mathcal{T}^{(q)}$, for all $x \in E_{q}^{\omega,1}$. If there exists $x_1 \neq x_2 \in E_{q}^{\omega,1}, y_1 \neq y_2 \in E_{q}^{\omega,2}$ such that $a(x_1,x_2)=b(y_1,y_2)$, then $f(x_1,x_2;y_1,y_2)=0$ which contradicts our assumption. Choose $x_1,x_2 \in E_{q}^{\omega,1}, y_1,y_2 \in E_{q}^{\omega,2}$ such that $r(a(x_1,x_2),b(y_1,y_2))>0$ is minimum, and take $z \in \llbracket a(x_1,x_2), b(y_1,y_2) \rrbracket \backslash \{a(x_1,x_2), b(y_1,y_2) \}$ . If $z \notin \Xi_{q}(\omega)$ then without loss of generality there exists $x \in E_{q}^{\omega,1}$ in the component of $\cT \backslash \{ z \}$ containing $y_1$ and $y_2$. But in this case $a(x_1,x) \in \llbracket a(x_1,x_2), b(y_1,y_2) \rrbracket \backslash \{a(x_1,x_2), b(y_1,y_2) \}$ which contradicts the minimality assumption. Thus, $z \in \Xi_{q}(\omega)$ which concludes the proof.
\end{proof}

The previous lemma admits a discrete counterpart. For $n \geq 1$, recall that a rooted plane tree $\tau_{n}$ can also be viewed as a rooted metric measure space $(\tau_{n}, r_{n}^{\text{gr}}, \emptyset_{n}, \mu_{n})$. For $q \geq 1$, let $u_{1}^{n}, \dots, u_{q}^{n}$ be $q$ i.i.d.\ random vertices sampled according to $\mu_{n}$. Let $u_{0,q}^{n} = \emptyset_{n}$ and $u_{1,q}^{n}, \dots, u_{q,q}^{n}$ be the $q$ vertices $u_{1}^{n}, \dots, u_{q}^{n}$ of $\tau_{n}$ listed in lexicographical order. Set $E_{q,n} = \{ u_{0,q}^{n}, \dots, u_{q,q}^{n}\}$. For $\omega \subset \{0, \dots, q\}$, let $\Xi_{q,n}(\omega)$ be the set of edges (if any) of $\tau_{n}$ splitting the set $E_{q,n}$ into $E_{q,n}^{\omega,1} = \{ u_{k,q}^{n}: k \in \omega \}$ and $E_{q,n}^{\omega,2} = \{ u_{k,q}^{n}: k \not \in \omega \}$. 

\begin{lemma}
\label{lem:discretedistances}
For $q \geq 1$ and $\omega \subset \{0, \dots, q\}$, we have that $\Xi_{q,n}(\omega)$ is either empty or a collection of neighbouring edges in $\tau_{n}$ that corresponds precisely to an edge $e_{\omega,n}$ in $\tau_{n}^{(q)}$ of length $\ell(e_{\omega, n})$ given by the number of edges in the set $\Xi_{q,n}(\omega)$. Reciprocally, for every edge $e_{n}$ in $\tau_{n}^{(q)}$ of length $\ell(e_{n})$ there exists $\omega_{e_{n}} \subset \{0, \dots, q\}$ such that $\Xi_{q,n}(\omega_{e_{n}})$ is a collection of neighbouring edges in $\tau_{n}$ that corresponds precisely to $e_{n}$ such that number of edges in the set $\Xi_{q,n}(\omega_{e_{n}})$ is equal to $\ell(e_{n})$. 

Furthermore, for any points $x_1, x_2 \in E_{q,n}^{\omega,1}$ and $y_1, y_2 \in E_{q,n}^{\omega,2}$, define
\begin{align*}
f_{n}\left( x_1, x_2; y_1,y_2 \right) \coloneqq \frac{r_{n}^{\rm gr}(x_1,y_1)+r_n^{\rm gr}(x_1,y_2)+r_n^{\rm gr}(x_2,y_1)+r_n^{\rm gr}(x_2,y_2)}{4} - \frac{r_n^{\rm gr}(x_1,x_2)+r_n^{\rm gr}(y_1,y_2)}{2}.
\end{align*}
Then $\Xi_{q,n}(\omega)$ is not empty if and only if $g_{n}(\omega) \coloneqq \min\{f_{n}(x_1,x_2;y_1,y_2): \, x_1,x_2 \in E_{q,n}^{\omega,1}, y_1, y_2 \in E_{q,n}^{\omega,2}\}>0$, in which case $g_{n}(\omega)=\ell(e_{\omega,n})$.
\end{lemma}

\begin{proof}
It follows as in the proof of Lemma \ref{lem:distances}. 
\end{proof}

If $\Xi_{q,n}(\omega)$ is not empty, let $e_{\omega,n}$ be the corresponding edge of $\tau_{n}^{(q)}$ that splits $E_{q,n}$ into $E_{q,n}^{\omega,1}$ and $E_{q,n}^{\omega,2}$. Denote by $\ell(e_{\omega,n})$ the number of edges in $\Xi_{q,n}(\omega)$, that is, the length of the edge $e_{\omega,n}$, with the convention $\ell(e_{\omega,n}) = 0$ whenever $\Xi_{q,n}(\omega)$ is empty. If $\Xi_{q}(\omega)$ is not empty, let $e_{\omega}$ be the corresponding edge of $\mathcal{T}^{(q)}$ that splits $E_{q}$ into $E_{q}^{\omega,1}$ and $E_{q}^{\omega,2}$. Denote by $\ell(e_{\omega})$ the length of the geodesic $\Xi_{q}(\omega)$; with the convention $\ell(e_{\omega}) = 0$ whenever $\Xi_{q}(\omega)$ is empty. The following lemma, whose proof makes use of Lemmas \ref{lem:distances} and \ref{lem:discretedistances}, states the equivalence of the convergence of a sequence of discrete trees and the convergence of the lengths of edges of the reduced trees.

\begin{lemma}
\label{lem:ghcv}
Let $(\tau_{n}, n \geq 1)$ be a sequence of random rooted plane trees, $\cT$ be a random plane continuum tree and $(a_{n}, n \geq 1)$ a sequence of non-negative real numbers satisfying $a_{n} \rightarrow \infty$ and $\zeta(\tau_{n})/a_{n} \rightarrow \infty$, as $n \rightarrow \infty$. Then, the following assertions are equivalent:
\begin{itemize}
\item[(i)] $\displaystyle  \left(\tau_{n}, \frac{a_{n}}{\zeta(\tau_{n})} r_{n}^{\rm gr}, \emptyset_{n}, \mu_{n} \right) \rightarrow (\mathcal{T}, r, \rho, \mu)$, as $n \rightarrow \infty$, in the planar Gromov-weak sense.
\item[(ii)] For every integer $q \geq 1$ fixed, $\displaystyle \left (\frac{a_{n}}{\zeta(\tau_{n})} \ell(e_{\omega,n}): \omega \subset \{0, \dots, q\} \right) \xrightarrow[ ]{d} (\ell(e_{\omega}): \omega \subset \{0, \dots, q\})$, as $n \rightarrow \infty$.
\end{itemize} 
\end{lemma}

\begin{proof}[Proof of Lemma \ref{lem:ghcv}]
For all $0 \leq i,j \leq q$, let $\Omega_{n}(i,j)$ be the set of subsets of $\omega \subset \{0, \dots, q\}$ such that $u^{n}_{i,q} \in E_{q,n}^{\omega,1}$ and $u^{n}_{j,q} \in E_{q,n}^{\omega,2}$. Similarly, let $\Omega(i,j)$ be the set of subsets of $\omega \subset \{0, \dots, q\}$ such that $x_{i,q} \in E_{q}^{\omega,1}$ and $x_{j,q} \in E_{q}^{\omega,2}$. Our claim then follows from (\ref{GromovWCplanar}), Lemma \ref{lem:distances}, Lemma \ref{lem:discretedistances} and the identities 
\begin{align*}
r_{n}^{\text{gr}}(u^{n}_{i,q}, u^{n}_{j,q}) = \sum_{\omega \in \Omega_{n}(i,j)} \ell(e_{\omega, n}) \qquad \text{ and } \hspace*{5mm} r(x_{i,q}, x_{j,q}) = \sum_{\omega \in \Omega(i,j)} \ell(e_{\omega}).
\end{align*}
\end{proof}

The rest of the section is devoted to the proof of Theorem \ref{thm:cvlamproc}. Let us start by proving that \ref{D2} implies \ref{D1}. As a preparation step, we need the following proposition. 

\begin{proposition}
\label{prop:resultsofcvg}
In the setting of Theorem \ref{thm:cvlamproc}, suppose that \ref{D2} is satisfied. Then, we have that:
\begin{itemize}
\item[(i)] 
there exists a coupling between the lamination-valued processes such that
\begin{align*}
\limsup_{n \rightarrow \infty} \, d_{\rm Sk}^{\mathbb{L}} \left( \left(\bL_{\frac{ta_n}{\zeta(\tau_{n})}}^{(q)} (\tau_{n}), t \geq 0 \right),\left(\bL_{t}^{(q)}(\cT), t \geq 0\right) \right) \overset{\mathbb P}{\underset{q \rightarrow \infty}{\rightarrow}} 0;
\end{align*}

\item[(ii)] for any $\varepsilon > 0$, 
\begin{align*}
\lim_{q \rightarrow \infty} \liminf_{n \rightarrow \infty} \, \mathbb{P}\left( d_{\rm Sk}^{\mathbb{L}} \left( \left(\bL_{\frac{ta_n}{\zeta(\tau_{n})}}^{(q)} (\tau_{n}), t \geq 0 \right), (\bL_{t a_n}(\tau_{n}), t \geq 0) \right)  < \varepsilon \right) =1.
\end{align*}
\end{itemize}
\end{proposition}

Roughly speaking, (i) states that, as $q \rightarrow \infty$, the time-rescaled discrete reduced lamination-valued process obtained from sampling $q$ i.i.d.\ uniform vertices of $\tau_n$ is asymptotically close to its continuum counterpart $(\bL_{t}^{(q)}(\cT), t \geq 0)$. On the other hand, by (ii), the complete process $(\bL_{t a_n}(\tau_{n}), t \geq 0)$ associated to $\tau_n$ can be approximated by the time-rescaled discrete reduced lamination-valued process, whenever $q$ is large enough.

\begin{proof}[Proof of Theorem \ref{thm:cvlamproc}, \ref{D2} $\Rightarrow$ \ref{D1}]
It follows from Propositions \ref{prop:cvlamredcont} and \ref{prop:resultsofcvg}.
\end{proof}

We now prove Proposition \ref{prop:resultsofcvg}.

\begin{proof}[Proof of Proposition \ref{prop:resultsofcvg} (i)]
We can and will assume, by Skorohod's representation theorem, that \ref{D2} holds almost surely. By \ref{D2}, we have that for all $\varepsilon>0, \mathbb{P}(|u_n|>\varepsilon \zeta(\tau_n)) \rightarrow 0$ as $n \rightarrow \infty$, where $u_n$ denotes a uniform vertex of $\tau_n$. This implies that, for all $q \in \mathbb N$ fixed, $\tau_n^{(q)}$ has $q$ leaves with high probability as $n \rightarrow \infty$. 

In the setting of Lemma \ref{lem:ghcv}, for any $\omega \subseteq \{ 0, \ldots, q \}$, we define $\gamma_{\omega,n}$ an exponential variable of parameter $\ell(e_{\omega, n})$ associated to the edge $e_{\omega, n}$ of $\tau_{n}^{(q)}$ whenever $\Xi_{q,n}(\omega)$ is not empty, otherwise we let $\gamma_{\omega,n} = \infty$ almost surely. Those exponential random variables correspond to the ones used in the definition of $(\bL_{t}^{(q)}(\tau_{n}), t \geq 0)$. Similarly, denote by $\gamma_{\omega}$ the exponential variable of parameter $\ell(e_{\omega})$ associated to the edge $e_{\omega}$ of $\mathcal{T}^{(q)}$ whenever $\Xi_{q}(\omega)$ is not empty, otherwise $\gamma_{\omega} = \infty$ almost surely. The latter exponential random variables correspond to the ones used in the definition of $(\bL_{t}^{(q)}(\mathcal{T}), t \geq 0)$. By Lemma \ref{lem:ghcv}, it follows, jointly with \ref{D2}, that the following convergence holds in $[0,\infty]^{\Omega_q}$, where $\Omega_q$ denotes the set of subsets of $\{ 0, \ldots, q\}$:
\begin{equation} \label{eq:poisson2}
\left( \frac{\zeta(\tau_{n})}{a_{n}} \gamma_{\omega,n}: \omega \subset \{ 0, \ldots, q \} \right) \xrightarrow[ ]{d} \left( \gamma_\omega: \omega \subset \{ 0, \ldots, q \}\right), \hspace*{3mm} \text{as} \hspace*{2mm}  n \rightarrow \infty.
\end{equation}

\noindent This implies that the ``jump times'' of the time-rescaled discrete reduced lamination-valued process converge to those of the continuum reduced lamination-valued process. Here, ``jump times'' refers to the times a new chord is added in the corresponding reduced lamination-valued processes. In particular, if $\gamma_{\omega,n} = \infty$ (resp.\ $\gamma_{\omega} = \infty$), then no chord associated to $\omega$ is added in the reduced lamination-valued processes (no jump), i.e., there is no edge in $\tau_{n}^{(q)}$ (resp.\ $\mathcal{T}^{(q)}$) associated to $\omega$. In fact, the ``actual jump times '' that will count toward the limit are those $\gamma_{\omega,n}$'s and $\gamma_{\omega}$'s for which $\ell(e_{\omega}) >0$. 

We can now assume, by Skorohod's representation theorem, that \ref{D2} and (\ref{eq:poisson2}) hold almost surely. Denote by $\Theta_{\infty}$ the class of strictly increasing, continuous mappings $\theta: \mathbb{R}_{+} \rightarrow \mathbb{R}_{+}$ with $\theta(0) = 0$ and $\theta(t) \uparrow \infty$, as $t \uparrow \infty$. Then to prove our claim, it is enough to show that there exists a sequence of functions $(\theta_{n}, n \geq 1) \in \Theta_{\infty}$ such that, for all $M \geq 0$,
\begin{eqnarray*}
\limsup_{n \rightarrow \infty} \sup_{t \geq 0}| \theta_{n}(t) - t| = 0 \hspace*{3mm} \text{and} \hspace*{3mm} \limsup_{n \rightarrow \infty}  \sup_{t \in [0,M]} d_{\rm H} \left( \bL_{\frac{\theta_{n}(t)a_n}{\zeta(\tau_{n})}}^{(q)}(\tau_{n} ), \bL_{t}^{(q)}(\cT) \right) \overset{\mathbb P}{\underset{q \rightarrow \infty}{\rightarrow}} 0; 
\end{eqnarray*}

\noindent see \cite[Proposition 3.5.3]{Ethier1986} (or \cite[Theorem 1.14 in Chapter VI]{jacod2003}). 

Let $(s_{j}: j =1,\dots,J) := \{ \gamma_\omega \, | \,  \ell(e_\omega) > 0 \}$ be the sequence of ``actual jump times'' of the continuum reduced lamination-valued process arranged in increasing order. Similarly, let $(s_{j}^{n}: j =1, \dots, J) := \{ \gamma_{\omega,n} \, |  \, \ell(e_{\omega,n}) > 0 \}$ be the corresponding sequence of ``actual jump times'' of the time-rescaled reduced lamination-valued process arranged in increasing order. Set $s_{0} = 0$ and define $\theta_{n}$ by letting $\theta_{n}(s_{0})=0$, $\theta_{n}(s_{j}) = s_{j}^{n}$, for $j = 1, \dots, J$, such that $\theta_{n}$ is linear on $[s_{j}, s_{j+1}]$ for $j = 0, 1, \dots, J-1$ and $[s_{J}, \infty)$ (with slope $1$ after $s_J$). Clearly, (\ref{eq:poisson2}) implies that $\theta_{n}$ converges uniformly to the identity mapping on $[0,\infty)$, as $n \rightarrow \infty$. Thus, it only remains to check that the chords that we add at ``actual jump times'' are asymptotically close to each other. Consider $\omega \subset \{ 0, \ldots, q \}$ such that $\ell(e_{\omega}) >0$. Let $e_{\omega, n}$ be the corresponding edge of $\tau_n^{(q)}$ (which exists for $n$ large enough). Denoting by $c_{\omega,n}$ and $c_{\omega}$ the chords coding respectively $e_{\omega,n}$ and $e_{\omega}$, we have that
\begin{align}
\sup_{t \in [0,M]} d_{\rm H} \left( \bL_{\frac{\theta_{n}(t)a_n}{\zeta(\tau_{n})}}^{(q)}(\tau_{n} ), \bL_{t}^{(q)}(\cT) \right) \leq \sup_{\omega \subset \{ 0, \ldots, q\}} d_{\rm H}(c_{\omega,n}, c_{\omega}).
\end{align} 
Furthermore, Lemma \ref{lem:repartitionunifcircle} implies that $\limsup_{n \rightarrow \infty} \sup_{\omega \subset \{ 0, \ldots, q\}} d_{\rm H}(c_{\omega,n}, c_{\omega}) \rightarrow 0$, in probability, as $q \rightarrow \infty$. This concludes our proof. 
\end{proof}

In order to prove Proposition \ref{prop:resultsofcvg} (ii), we make use of the following lemma:
\begin{lemma}
\label{lem:smallsubtrees}
Let $(\tau_n, n \geq 1)$ be a sequence of random rooted plane trees such that $\zeta(\tau_n) \rightarrow \infty$, as $n \rightarrow \infty$. Suppose that, for any $\eta>0$, 
\begin{equation} \label{eq:preuve1}
\lim_{n \rightarrow \infty}\mathbb P(|u_n| > \eta \zeta(\tau_n) ) =0,
\end{equation}
\noindent where $u_n$ is a uniformly distributed vertex in $\tau_n$. For a vertex $u \in \tau_{n}$, let $\tau_{n}(u)$ be the subtree of $\tau_{n}$ rooted at $u$ (i.e., $\tau_{n}(u)$ consists of $u$ and all its descendants) and let $\zeta(\tau_{n}(u))$ be its size (i.e., number of vertices). Then, for all $\varepsilon>0$,
\begin{align*}
\lim_{n \rightarrow \infty}\mathbb P( \zeta(\tau_n(u_n)) > \varepsilon \zeta(\tau_n) ) =0. 
\end{align*}
\end{lemma}

\begin{proof}
Fix $\varepsilon>0$, and set $\Theta_n(\varepsilon) \coloneqq \left\{ v \in \tau_n: \zeta(\tau_n(v))>\varepsilon \zeta(\tau_n) \right\}$. Call $v \in \Theta_n(\varepsilon)$ maximal if no child of $v$ belongs to $\Theta_n(\varepsilon)$. Then, it is clear that all subtrees rooted at a maximal vertex of $\Theta_n(\varepsilon)$ are disjoint. Hence the number of maximal vertices is $\leq \varepsilon^{-1}$. Now observe that, as $n \rightarrow \infty$,
\begin{equation}
\label{eq:preuve2}
\sup_{v \in \Theta_n(\varepsilon)} |v| =o\left( \zeta(\tau_n) \right).
\end{equation}
\noindent Indeed, clearly $\Theta_n(\varepsilon)$ is nonempty and, for any $v \in \Theta_n(\varepsilon)$, $\mathbb P\left( u_n \in \tau_n(v) \right) \geq \varepsilon > 0$. Furthermore, if $u_n \in \tau_n(v)$, then $|u_n| \geq |v|$. Thus, $\mathbb P(|u_n| \geq \sup_{v \in \Theta_n(\varepsilon)} |v|) \geq  \varepsilon$, which implies \eqref{eq:preuve2} by the assumption \eqref{eq:preuve1}. Finally, letting $M_n(\varepsilon)$ be the number of maximal vertices in $\Theta_n(\varepsilon)$, we have
\begin{align*}
\# \Theta_n(\varepsilon) \leq M_n(\varepsilon) \sup_{v \in \Theta_n(\varepsilon)} |v| = o\left( \zeta(\tau_n) \right).
\end{align*}
\noindent In particular,  $\mathbb P( u_n \in \Theta_n(\varepsilon) ) \rightarrow 0$, as $n \rightarrow \infty$, which implies our claim. 
\end{proof}

\begin{proof}[Proof of Proposition \ref{prop:resultsofcvg} (ii)] 
By Proposition \ref{prop:closeprocesses}, it is enough to prove that, for any $\varepsilon >0$:
\begin{align*}
\lim_{q \rightarrow \infty} \liminf_{n \rightarrow \infty} \mathbb{P}\left( d_{\rm Sk}^{\mathbb{L}} \left( \left(\bL_{\frac{ta_n}{\zeta(\tau_{n})}}^{(q)} (\tau_{n}), t \geq 0 \right), \left( \bL^{\rm d}_{\frac{ta_n}{\zeta(\tau_{n})}}(\tau_{n}), t \geq 0 \right)\right)  < \varepsilon \right) =1.
\end{align*}

Fix $q \geq 1$ and recall that, by \ref{D2}, we can assume that $\tau^{(q)}_{n}$ has exactly $q$ leaves. Consider $\omega \subset \{ 0, \ldots, q \}$ such that $\Xi_{q,n}(\omega)$ is not empty. Then, the time at which the chord coding the edge $e_{\omega,n}$ in $\tau^{(q)}_{n}$ appears in the lamination-valued process $(\bL^{(q)}_t(\tau_n), t \geq 0)$ is distributed as an exponential variable of parameter $\ell(e_{\omega,n})$ (which is the minimum of the $\ell(e_{\omega,n})$ exponential random variables of parameter $1$ associated to the set of edges $\Xi_{q,n}(\omega)$; see also Definition \ref{Def5}). Hence, we only need to prove that for all $\varepsilon>0$ there exists $Q$ large enough so that with probability $>1-\varepsilon$, for all $q \geq Q$, all $n$ large enough,
\begin{itemize}
\item[(a)] for any $\omega \subset \{ 0, \ldots, q \}$ such that $\Xi_{q,n}(\omega)$ is not empty, and for any edge $e_{n} \in \Xi_{q,n}(\omega)$, $d_{\rm H}(c_{n}, \tilde{c}_{\omega,n}) < \varepsilon$, where $c_{n}$ is the chord of $\bL(\tau_n)$ coding $e_{n}$ and $\tilde{c}_{\omega,n}$ is the chord of $\bL^{(q)}(\tau_n)$ coding $e_{\omega,n}$; 
\item[(b)] for any edge $e_{n} \in \tau_n$ whose removal does not split the set $E_{q,n}$, $d_{\rm H}(c_{n}, \bS^1) < \varepsilon$;
\end{itemize}

Fix $\varepsilon>0$. Then, by Lemma \ref{lem:repartitionunifcircle}, we can take $Q$ large enough such that, for all $q \geq Q$,
\begin{equation}
\label{eq:lemma4inproof}
\mathbb{P}\left( \sup_{1 \leq j \leq q} d \left(a_{j,q}, e^{-2\pi i j/q}\right) < \varepsilon \right) > 1 -\varepsilon, 
\end{equation} 
\noindent where we recall that $(a_{j,q})_{1 \leq j \leq q}$ are $q$ i.i.d.\ points on $\bS^1$ sorted in clockwise order (starting from $1$). In particular, observe that (a) follows directly from \eqref{eq:lemma4inproof} along with Definition \ref{Def4}. So, it only remains to prove (b). For a vertex $u \in \tau_{n}$, recall that $\tau_{n}(u)$ is the subtree of $\tau_{n}$ rooted at $u$, i.e., $\tau_{n}(u)$ consists of $u$ and all its descendants. Observe that the edges considered in (b) are of two kinds: either they are in a subtree of the form $\tau_{n}(u_{k,q}^n)$ for some $1 \leq k \leq q$, or they are in subtrees branching out of the set of edges that are in the geodesic paths of $\tau_n^{(q)}$.

To deal with the edges of the first kind, observe that with probability $1-o(1)$, as $n \rightarrow \infty$, all subtrees $(\tau_{n}(u_{k,q}^n), 1 \leq k \leq q)$ have size $o(\zeta(\tau_{n}))$. This follows from \ref{D2}, since $|u_n|=o(\zeta(\tau_n))$, in probability, where $u_n$ denotes a uniform vertex of $\tau_n$. Then, by Lemma \ref{lem:smallsubtrees}, with probability $1-o(1)$, all chords corresponding to edges of the first kind have length $o(1)$, as $n \rightarrow \infty$. 

To deal with edges of the second kind, we use the definition of the lamination-valued process from the contour function $C_{\tau_{n}}$ of $\tau_n$; see Definition \ref{Def2} or Definition \ref{Def3}. First, we recall a way of sampling a uniform vertex of $\tau_n$. Consider a uniform random variable $U$ on $[0,1]$. Then, let $e_{U}$ be the edge of $\tau_n$ visited at time $2 \zeta(\tau_{n}) U$ by $C_{\tau_{n}}$, and let $v_{U}$ be the endpoint of $e_{U}$ further from the root (if $2\zeta(\tau_{n}) U \geq 2\zeta(\tau_{n})-2$, set $v_{U}=\emptyset_{n}$). The vertex $v_{U}$ is clearly uniform among the vertices of $\tau_n$. We use this procedure to sample the $q$ uniform vertices $u_{1}^{n}, \dots, u_{q}^{n}$ of $\tau_n$ from $q$ i.i.d.\ uniform random variables $U_{1}, \ldots, U_{q}$. We get that chords that code edges of the second kind necessarily have their two endpoints between two consecutive points on $\bS^1$ (in clockwise order) of the set $\{ e^{-2\pi i U_{k}}: 0 \leq k \leq q \}$ (with the convention that $U_{0}=0$). Finally, we conclude by Lemma \ref{lem:repartitionunifcircle}.
\end{proof}

We can now prove the other implication in Theorem \ref{thm:cvlamproc}, that is, the convergence of the lamination-valued process implies the planar Gromov-weak convergence of the rooted plane trees.

\begin{proof}[Proof of Theorem \ref{thm:cvlamproc}, \ref{D1} $\Rightarrow$ \ref{D2}]
By Skorohod's representation theorem, suppose that \ref{D1} holds almost surely. In particular, \ref{D1} and Proposition \ref{prop:closeprocesses} imply that 
\begin{eqnarray} \label{D12}
\left(\bL_{\frac{t a_{n}}{\zeta(\tau_{n})}}^{\rm d}(\tau_{n}), t \geq 0 \right) \xrightarrow[ ]{a.s.} (\bL_{t}(\mathcal{T}), t \geq 0), \hspace*{3mm} \text{as} \hspace*{2mm}  n \rightarrow \infty, \hspace*{2mm} \text{in} \hspace*{2mm} \mathbf{D}(\mathbb{R}_{+}, \bL(\bar{\mathbb{D}})). 
\end{eqnarray}  
Fix $q \geq 1$ and recall that we set $x_{0,q}=\rho$ and $x_{1,q}, \ldots, x_{q,q}$ the $q$ i.i.d.\ leaves of $\cT$ distributed according $\mu$ and listed in lexicographical order. Recall that for $x \in \mathcal{T}$, we denote by $G(x)$ the $\mu$-mass of the set of leaves of $\mathcal{T}$ that lie on the left of $x$; see \eqref{LeftRight}. For $n \geq 1$, we couple the reduced trees $\tau_n^{(q)}$ as follows. For $1 \leq k \leq q$, denote by $u^n_{k,q}$ the unique vertex of $\tau_n$ such that the edge between $u^{n}_{k,q}$ and its parent is visited at time $2\zeta(\tau_n) \cdot G(x_{k,q})$ by the contour function of $\tau_n$. If $2\zeta(\tau_n) G(x_{k,q}) \geq 2\zeta(\tau_n)-2$, set $u^{n}_{k,q}=\emptyset_{n}$. It follows from Lemma \ref{lem:measure} that the vertices $u^{n}_{1,q}, \dots u^{n}_{q,q}$ are i.i.d.\ uniform vertices of $\tau_n$ in lexicographical order. Recall also that we write $u^{n}_{0,q} = \emptyset_{n}$. 

Let us prove that the sequence of properly rescaled rooted plane trees converges in the planar Gromov-weak sense toward $\cT$. To be precise, we check that (\ref{GromovWC}) is satisfied in this setting - or equivalently Lemma \ref{lem:ghcv} (ii).

By Lemma \ref{lem:measure}, sampling $x_{1,q}, \dots, x_{q,q}$ is equivalent to sample the order statistics of $q$ i.i.d.\ uniform points on $\bS^1$, say $a_{1,q}, \ldots, a_{q,q}$, by letting $a_{k,q} = e^{-2\pi i G(x_{k,q})}$, for $1 \leq k \leq q$. We also set $a_{0,q} = e^{-2\pi i G(x_{0,q})} =1$. Recall from Lemma \ref{lem:distances}, that for any $\omega \subset \{0,\ldots, q \}$ such that $\Xi_{q}(\omega)$ is not empty, the set $\Xi_{q}(\omega)$ is a geodesic of $\mathcal{T}$ that corresponds to an edge $e_\omega$ of $\cT^{(q)}$ that splits $E_q$ into $E_{q}^{\omega, 1}$ and $E_{q}^{\omega, 2}$. By the definition of the continuous lamination-valued process in \eqref{ContinuumLamination}, any point of the Poisson point process $\Pi$ on $\cT$ falling on $\Xi_{q}(\omega)$ is coded by a chord splitting $A_{q} \coloneqq \{a_{0,q}, \ldots, a_{q,q} \}$ into $\{ a_{k,q}: k \in \omega \}$ and $\{ a_{k,q}: k \notin \omega \}$. Denote by $c_{\omega}$ the first such chord in the lamination-valued process $(\bL_t(\cT), t \geq 0)$ and by $T_{\omega}$ the time at which it appears. For all $n \geq 1$, let also $T_{\omega,n}$ be the time at which the first such chord, say $c_{\omega,n}$, appears in the lamination-valued process $( \bL_{t}^{\rm d}(\tau_n), t \geq 0)$. If there is no such edge $e_\omega$ (i.e., $\Xi_{q}(\omega)$ is empty) or no chord $c_{n,\omega}$, set $T_{\omega}=\infty$ and $T_{\omega,n}=\infty$, respectively. Therefore, Theorem \ref{thm:cvlamproc}, \ref{D1} $\Rightarrow$ \ref{D2} follows by showing that the following convergence holds in $[0,\infty]^{\Omega_q}$:
\begin{equation} \label{Eqtimes}
\left( \frac{\zeta(\tau_{n})}{a_{n}} T_{\omega,n}: \omega \subset \{ 0, \ldots, q \} \right) \xrightarrow[ ]{d} \left( T_\omega: \omega \subset \{ 0, \ldots, q \}\right), \hspace*{3mm} \text{as} \hspace*{2mm}  n \rightarrow \infty.
\end{equation}

\noindent Indeed, in the setting of Lemma \ref{lem:ghcv}, for $\omega \subset \{0,\ldots, q \}$ such that $\Xi_{q}(\omega)$ and $\Xi_{q,n}(\omega)$ are not empty, observe that $T_{\omega,n}$ and $T_{\omega}$ are distributed as exponential random variables of respective parameters $\ell(e_{\omega, n})$ and $\ell(e_{\omega})$. Then, it is a simple exercise to check that (\ref{Eqtimes}) implies the statement of Lemma \ref{lem:ghcv} (ii) and therefore our result.

Let us then prove (\ref{Eqtimes}). Observe that (\ref{Eqtimes}) is clear for the $\omega$'s such that $T_{\omega}=\infty$. Indeed, if $T_{\omega, n}$ was bounded by some $K > 0$ along a subsequence then by (\ref{D12}) the sequence of associated chords would converge (up to taking again a subsequence) toward a chord which would appear in the continuous lamination-valued process before time $K$. Furthermore, this sequence of chords cannot degenerate into a point, since this point would be a leaf (or the root) of $\cT^{(q)}$ and $T_{\omega}<\infty$ for any singleton $\omega$. Therefore, we only have to focus on the case $T_{\omega}<\infty$.

By (\ref{D12}), necessarily $\liminf_{n \rightarrow \infty} (\zeta(\tau_{n})/a_{n}) T_{\omega,n} \geq T_{\omega}$ almost surely. On the other hand, let us prove that for every $\omega \subset \{0,\ldots, q \}$ such that $\Xi_{q}(\omega)$ is not empty, the chord $c_\omega$ is necessarily well approximated by a sequence of chords in the discrete lamination-valued processes. To this end, let $(a_1,a_2)$ and $(b_1,b_2)$ be the two arcs between consecutive points of $A_q$ connected by $c_\omega$ (with $a_1, a_2, b_1, b_2 \in A_{q}$ in this clockwise order), and denote by $p_{\omega}$ the middle of the chord $c_\omega$. Suppose that there exists a subsequence $(n_{m}, m \geq 1)$ of non-negative integers such that, for all $m \geq 1$, there exists a chord $c_{m}$ in $\bL_{a_{n_{m}} (T_{\omega}+1/m)/\zeta(\tau_{n_{m}})}^{\rm d}(\tau_{n_{m}})$ satisfying $d(p_{\omega},c_{m})\leq m^{-1}$ and that does not connect the arcs $(a_1,a_2)$ and $(b_1,b_2)$; here $d(p_{\omega},c_{m})$ denotes the distance from the point $p_{\omega}$ to the set $c_{m}$. Up to taking a subsequence, we can assume that $c_m$ has an endpoint in the arc $(a_2,b_1)$. Hence, since by (\ref{D12}), $\bL_{a_{n_m} (T_{\omega}+1/m)/\zeta(\tau_{n_m})}^{\rm d}(\tau_{n_m}) \rightarrow \bL_{T_{\omega}}(\cT)$, as $m \rightarrow \infty$, almost surely, there would exist a chord in $\bL_{T_{\omega}}(\cT)$ containing $p_{\omega}$ and with an endpoint in $(a_2,b_1)$, and thus crossing $c_\omega$. However, the above necessarily does not happen and, along all sub-sequences $(n_{m}, m \geq 1)$, for $m$ large enough, a chord $c_{m}$ of $\bL_{a_{n_{m}} (T_{\omega}+1/m)/\zeta(\tau_{n_{m}})}^{\rm d}(\tau_{n_{m}})$ such that $d(p_{\omega},c_{m})<m^{-1}$ connects the arcs $(a_1,a_2)$ and $(b_1,b_2)$. Therefore, we get that $\limsup_{n \rightarrow \infty} (\zeta(\tau_{n})/a_{n}) T_{\omega,n} \leq T_{\omega}$ almost surely. This concludes the proof of (\ref{Eqtimes}). 
\end{proof}

\section{Fragmentation and process of masses}
\label{sec:fragmentation}

The aim of this final section is to provide sufficient conditions on a sequence of rooted plane trees $(\tau_{n}, n \geq 1)$ to ensure that their appropriately rescaled fragmentation processes converge to a fragmentation process constructed from an excursion-type function, and consequently, to prove Corollary \ref{cor:masses}. 
\subsection{Convergence of the fragmentation process of trees}

For convenience, we consider a slightly different version of the fragmentation process described at the beginning of Section \ref{SecFragLaMain}, where edges are removed at i.i.d.\ uniform random times, rather than at integer times.

Let $\tau$ be a rooted plane tree and equip the edges $\textbf{edge}(\tau)$ of $\tau$ with i.i.d.\ uniform random variables (or weights) $\mathbf{w} = (w_{e}: e \in \textbf{edge}(\tau))$ on $[0,1]$ independent of $\tau$. In particular, for a vertex $v \in \tau$ with $k_{v} \geq 1$ children, we write $(w_{v,i}, 1 \leq i \leq k_{v})$ for the weights of the edges connecting $v$ with its children. For $s \in [0,1]$, we then keep the edges of $\tau$ with weight smaller than $s$ and discard the others. This gives rise to a forest $\mathbf{f}_{\tau}(s)$ with set of edges given by $\textbf{edge}(\mathbf{f}_{\tau}(s)) = \{ e \in \textbf{edge}(\tau): w_{e} \leq s\}$. Furthermore, each vertex $v \in \mathbf{f}_{\tau}(s)$ has $k_{s}(v) = \sum_{i=1}^{k_{v}} \mathbf{1}_{\{ w_{v,i} \leq s \}}$ children if $k_{v} \geq 1$; otherwise, $k_{s}(v) = 0$ whenever $k_{v} =0$. The forest $\mathbf{f}_{\tau}(s)$ associated to $\tau $ and $\mathbf{w}$ is called {\sl the fragmentation forest} at time $s$. Let $\mathbf{F}_{\tau}(u)$ be the sequence of sizes (number of vertices) of the connected components of the forest $\mathbf{f}_{\tau}(1-u)$, ranked in decreasing order. We view the sequence of sizes of the components of $\mathbf{f}_{\tau}(1-u)$ as an infinite sequence, by completing it with an infinite number of zero terms. In particular, $\mathbf{F}_{\tau}(0) = (\zeta(\tau), 0, 0, \dots)$ and $\mathbf{F}_{\tau}(1) = (1, 1, \dots, 1, 0, 0, \dots)$ (with $\zeta(\tau)$ $1$'s). We refer to $\mathbf{F}_{\tau} = (\mathbf{F}_{\tau}(u), u \in [0,1])$ as the fragmentation process associated with $\tau$. 

Following \cite[Section 3]{Bertoin2001}, we next explain how to construct fragmentation processes from excursion-type functions. A function $g = (g(s), s \in [0,1]) \in \mathbf{D}([0,1], \mathbb{R})$ is an excursion-type function if $g(0) = g(1)=g(1-) =0$, it is non-negative and it makes only positive jumps (i.e.\ $g(s-) \leq g(s)$ for all $s \in (0, 1]$). For such a function $g$ and every $t \geq 0$, define $g^{(t)} = (g^{(t)}(s), s \in [0,1])$ and $I^{(t)}_{g} = (I^{(t)}_{g}(s), s \in [0,1])$ as
\begin{eqnarray*} 
g^{(t)}(s) = g(s) - ts \hspace*{2mm} \text{and} \hspace*{2mm} I^{(t)}_{g}(s) = \inf_{u \in [0,s]} g^{(t)}(u), \hspace*{4mm} \text{for} \hspace*{2mm} s \in [0,1].
\end{eqnarray*} 

\noindent For $t \geq 0$, we write $\mathbf{F}_{g}(t) = (F_{g, 1}(t), F_{g, 2}(t), \dots)$ for the ranked sequence (in decreasing order) of the lengths of the interval components of the complement of the support of the Stieltjes measure ${\rm d} (- I^{(t)}_{g})$; note that $s \mapsto -I^{(t)}_{g}(s) = \sup_{u \in [0,s]} - g^{(t)}(u)$ is an increasing process. The process $\mathbf{F}_{g} = (\mathbf{F}_{g}(t), t \geq 0)$ is the fragmentation process associated to the excursion-type function $g$. Let $\text{Supp}({\rm d} (- I^{(t)}_{g}))$ denote the support of ${\rm d} (- I^{(t)}_{g})$ and note that $(0,1) \setminus \text{Supp}({\rm d} (- I^{(t)}_{g}))$ is the union of all open intervals on which the function $-I^{(t)}_{g}$ is constant. We call constancy interval of $-I^{(t)}_{g}$ any interval component of $(0,1) \setminus \text{Supp}({\rm d} (- I^{(t)}_{g}))$.

\begin{theorem} \label{Theo3}
Let $(\tau_{n}, n \geq 1)$ be a sequence of random rooted plane trees. Suppose that there are a sequence $(a_{n}, n \geq 1)$ of positive real numbers and a random excursion-type function $X = (X(u), u \in [0,1])$ satisfying
\begin{enumerate}[label=(\textbf{D.\arabic*})]
\item $a_{n} \rightarrow \infty$ and $\frac{\zeta(\tau_{n})}{a_{n}} \rightarrow \infty$, as $n \rightarrow \infty$; \label{C1}

\item for $n \geq 1$, let $(W_{\tau_{n}}^{\rm prim}(\zeta(\tau_{n}) u), u \in [0,1])$ be the (time-scaled) Prim path of $\tau_{n}$ with respect to $\mathbf{w}$. Then, $(a_{n}^{-1}W_{\tau_{n}}^{\rm prim}(\zeta(\tau_{n}) u), u \in [0,1]) \xrightarrow[ ]{d} (X(u), u \in [0,1])$, as $n \rightarrow \infty$, in $\mathbf{D}([0,1], \mathbb{R})$; \label{C2}

\item for every fixed $t \geq 0$, $X^{(t)}(s) \wedge X^{(t)}(s-) > I_{X}^{(t)}(s)$, for $s \in (s^{\prime},s^{\prime \prime})$, whenever $(s^{\prime},s^{\prime \prime}) \in [0,1]$ is an interval of constancy of $- I_{X}^{(t)}$, almost surely.
\label{C3}
\setcounter{Cond}{\value{enumi}}
\end{enumerate}

\noindent Then, for every fixed $t >0$,
\begin{eqnarray*}
\frac{1}{\zeta(\tau_{n})}\mathbf{F}_{\tau_{n}}\left(t\frac{a_{n}}{\zeta(\tau_{n})} \right) \xrightarrow[ ]{d} \mathbf{F}_{X}(t), \hspace*{3mm} \text{as} \hspace*{2mm}  n \rightarrow \infty, \hspace*{2mm} \text{in} \hspace*{2mm}  \boldsymbol{\Delta},
\end{eqnarray*}
\noindent equipped with the topology of pointwise convergence. If moreover, 
\begin{enumerate}[label=(\textbf{D.\arabic*})]
\setcounter{enumi}{\value{Cond}}
\item for every fixed $t \geq 0$, $\mathbf{F}_{X}(t) \in \boldsymbol{\Delta}_{1}$ almost surely, where $\boldsymbol{\Delta}_{1} \subset \boldsymbol{\Delta}$ is the space of the elements of $\boldsymbol{\Delta}$ with sum $1$, \label{C4}

\item $\mathbb{E}[ r_{\tau_{n}}^{\rm gr}(u_{n}, u_{n}^{\prime})] = O(\zeta(\tau_{n})/a_{n})$, where $r_{\tau_{n}}^{\rm gr}(u_{n}, u_{n}^{\prime})$ denotes the graph distance between two independent uniformly random vertices $u_{n}$ and $u_{n}^{\prime}$ of $\tau_{n}$.  \label{C5}
\end{enumerate}

\noindent then
\begin{eqnarray*}
\left(\frac{1}{\zeta(\tau_{n})}\mathbf{F}_{\tau_{n}}\left(t\frac{a_{n}}{\zeta(\tau_{n})} \right), t \geq 0 \right) \xrightarrow[ ]{d} (\mathbf{F}_{X}(t), t \geq 0), \hspace*{3mm} \text{as} \hspace*{2mm}  n \rightarrow \infty, \hspace*{2mm} \text{in} \hspace*{2mm} \mathbf{D}(\mathbb{R}_{+}, \boldsymbol{\Delta}).
\end{eqnarray*}
\end{theorem}

 To prove Theorem \ref{Theo3}, we follow the approach developed in \cite{Brou2016}, which was also recently used in \cite{Berzunza2020}. Therefore, to avoid unnecessary repetitions, we only provide enough details to convince the reader that the arguments can be adapted from these works.

Let $\emptyset = u(0) \prec_{\text{prim}}  u(1) \prec_{\text{prim}} \dots  \prec_{\text{prim}} u(\zeta(\tau_{n})-1)$ be the Prim order of the vertices of $\tau_{n}$ with respect to $\mathbf{w}$. Since $\mathbf{f}_{\tau_{n}}(s)$ and $\tau_{n}$ possess the same set of vertices, we can and will consider that the vertices of $\mathbf{f}_{\tau_{n}}(s)$ are ordered according to the Prim order of the vertices in $\tau_{n}$. For $n \geq 1$ and $s \in [0,1]$, we associate to the Prim order of the vertices of $\mathbf{f}_{\tau_{n}}(s)$ an exploration path $W_{\tau_{n}}^{(s)} = (W_{\tau_{n}}^{(s)}(i), 0 \leq i \leq \zeta(\tau_{n}))$ by letting $W_{\tau_{n}}^{(s)}(0) = 0$, and for $0 \leq i \leq \zeta(\tau_{n})-1$, $W_{\tau_{n}}^{(s)}(i+1) = W_{\tau_{n}}^{(s)}(i) + k_{s}(u(i))-1$, where $k_{s}(u(i))$ denotes the number of children of $u(i)$ in $\mathbf{f}_{\tau_{n}}(s)$. We shall think of such a path as the step function on $[0,\zeta(\tau_{n})]$ given by $u \mapsto W_{\tau_{n}}^{(s)}(\lfloor u \rfloor)$. For fixed $t \geq 0$, consider the sequence $(s_{n}(t), n \geq 1)$ of positive times given by
\begin{eqnarray*}
s_{n}(t) = \max\left( 1-\frac{a_{n}}{ \zeta(\tau_{n})}t, 0 \right),
\end{eqnarray*} 

\noindent and define the process $X_{n}^{(t)} = (X_{n}^{(t)}(u),u \in [0,1])$ by letting
\begin{eqnarray*}
X_{n}^{(t)}(u) =  \frac{1}{a_{n}} W_{\tau_{n}}^{(s_{n}(t))}( \zeta(\tau_{n}) u), \hspace*{4mm} \text{for} \hspace*{2mm} u \in [0,1]. 
\end{eqnarray*}

\noindent For simplicity, we use the notation $X_{n} = (X_{n}^{(t)}, t \geq 0)$. For any $u \in [0,1]$, the mapping $t \mapsto X_{n}^{(t)}(u)$ is non-increasing in $t$, which implies that $X_{n}$ has c\`adl\`ag paths. In particular, we can view the process $t \mapsto X_{n}^{(t)}$ as a random variable taking values in the space $\mathbf{D}(\mathbb{R}_{+}, \mathbf{D}([0,1], \mathbb{R}))$. In other words, for fixed $t \geq 0$, $X_{n}^{(t)}$ is a random variable in $\mathbf{D}([0,1], \mathbb{R})$.

\begin{theorem}  \label{Theo4}
In the setting of Theorem \ref{Theo3}, we have that
\begin{eqnarray*} 
(X_{n}^{(t)}, t \geq 0) \xrightarrow[ ]{d} (X^{(t)}, t \geq 0), \hspace*{3mm} \text{as} \hspace*{2mm} n \rightarrow \infty, \hspace*{2mm} \text{in} \hspace*{2mm} \mathbf{D}(\mathbb{R}_{+}, \mathbf{D}([0,1], \mathbb{R})).
\end{eqnarray*}
\end{theorem}

\begin{proof}[Proof of Theorem \ref{Theo4}]
The proof of Theorem \ref{Theo4} follows from \ref{C1}-\ref{C2} by adapting the argument used in the proof of \cite[Theorem 3]{Berzunza2020} (see also \cite[Section 5]{Brou2016}). It involves two steps: convergence of the finite-dimensional distributions and tightness of the sequence of processes $(X_{n}, n \geq 1)$. Indeed, one only needs to be aware that, for fixed $t \geq 0$ and for $0 \leq i \leq \zeta(\tau_{n})-1$, the number of children $k_{s_{n}(t)}(u(i))$ of the vertex $u(i) \in \mathbf{f}_{\tau_{n}}(s_{n}(t))$ is distributed as a binomial random variable with parameters $(k_{u(i)}, s_{n}(t))$. Details are left to the interested reader.
\end{proof}

Next we prove Theorem \ref{Theo3}.

\begin{proof}[Proof of Theorem \ref{Theo3}]
For $t \geq 0$, define the process $I_{n}^{(t)} = (I_{n}^{(t)}(u), u \in [0,1])$ by letting
\begin{eqnarray*}
I_{n}^{(t)}(u) = \inf_{s \in [0,u]} X_{n}^{(t)}(s),  \hspace*{4mm} \text{for} \hspace*{2mm} u \in [0,1].
\end{eqnarray*}
\noindent Similarly, we define the process $I^{(t)} = (I^{(t)}(u), u \in [0,1])$ by letting $I^{(t)}(u) = \inf_{s \in [0,u]} X^{(t)}(s)$, for  $u \in [0,1]$. 

\noindent For $t \geq 0$, we write $\mathbf{F}(-I_{n}^{(t)}) = (F_{1}(-I_{n}^{(t)}), F_{2}(-I_{n}^{(t)}), \dots)$ for the ranked sequence (in decreasing order) of the lengths of the intervals components of the complement of the support of the Stieltjes measure ${\rm d} (-I_{n}^{(t)})$. By \cite[Lemma 1]{Berzunza2020}, we know that 
\begin{eqnarray} \label{eq1NewRevII}
\frac{1}{\zeta(\tau_{n})}\mathbf{F}_{\tau_{n}}\left(t\frac{a_{n}}{\zeta(\tau_{n})} \right) = \mathbf{F}( - I_{n}^{(t)}), \hspace*{4mm} \text{for} \hspace*{2mm} t \geq 0. 
\end{eqnarray}

\noindent Observe that  $X_{n}^{(t)}(0) = X^{(t)}(0) = 0$, for all $t \geq 0$. Our first claim follows from Theorem \ref{Theo4}, \ref{C3} and  \cite[Lemma 4]{Bertoin2001}. 

Next, we prove the second claim. By the Skorokhod representation theorem, we can and we will work in a probability space where the convergence in Theorem \ref{Theo4}  together with \ref{C3} and \ref{C4}  holds almost surely. By Theorem \ref{Theo4}, there exists a dense subset $D$ of $\mathbb{R}_{+}$ such that for any fixed $k \in \mathbb{N}$ and collection $0 \leq t_{1} < t_{2} < \cdots < t_{k} < \infty$ with $t_{1}, \dots, t_{k} \in D$, we have that a.s.,
\begin{eqnarray*}
(I_{n}^{(t_{1})}, \dots, I_{n}^{(t_{k})}) \rightarrow (I^{(t_{1})}, \dots, I^{(t_{k})}), \hspace*{3mm} \text{as} \hspace*{3mm} n \rightarrow \infty,
\end{eqnarray*}

\noindent in $\mathbf{D}([0,1], \mathbb{R})^{k}$ (i.e., the $k$-fold space of $\mathbf{D}([0,1], \mathbb{R})$). Then \cite[Lemma 4]{Bertoin2001} implies that a.s.,
\begin{eqnarray*}
(\mathbf{F}(- I_{n}^{(t_{1})}), \dots, \mathbf{F}(- I_{n}^{(t_{k})})) \rightarrow (\mathbf{F}(-I^{(t_{1})}), \dots, \mathbf{F}(-I^{(t_{k})})), \hspace*{3mm} \text{as} \hspace*{3mm} n \rightarrow \infty,
\end{eqnarray*}
\noindent in $\boldsymbol{\Delta}^{k}$ (i.e., the $k$-fold space of $\boldsymbol{\Delta}$ equipped with the $\ell^{1}$-norm). Note that the conditions in \cite[Lemma 4]{Bertoin2001} are satisfied by our assumptions (in fact, one has to apply \cite[Lemma 4]{Bertoin2001} to $-I_{n}$ and $-I$).  This shows the convergence of the finite-dimensional distributions of the sequence of processes $((\mathbf{F}(- I_{n}^{(t)}), t \geq 0))_{n \geq 1}$ to those of the process $(\mathbf{F}(- I^{(t)}), t \geq 0)$. On the other hand, \ref{C5} and \cite[Corollary 5.5]{GbCeSv2025} imply that $((\mathbf{F}(- I_{n}^{(t)}), t \geq 0))_{n \geq 1}$ is tight in $\mathbf{D}(\mathbb{R}_{+}, \boldsymbol{\Delta})$, which concludes our proof (recall \eqref{eq1NewRevII}).
\end{proof}

The following result establishes the convergence, after proper rescaling, of the fragmentation process associated with a tree $\mathbf{t}_{n}$ having a given degree sequence.  

\begin{theorem} \label{Theo1}
Suppose that $\mathbf{s}_{n}$ satisfies \ref{B1}-\ref{B3} and \ref{A4}. Then, 
\begin{eqnarray*}
\left(\frac{1}{V_{n}}\mathbf{F}_{\mathbf{t}_{n}}\left(t\frac{b_{n}}{V_{n}} \right), t \geq 0 \right) \xrightarrow[ ]{d} (\mathbf{F}_{X^{\rm exc}}(t), t \geq 0), \hspace*{3mm} \text{as} \hspace*{2mm}  n \rightarrow \infty, \hspace*{2mm} \text{in} \hspace*{2mm} \mathbf{D}(\mathbb{R}_{+}, \boldsymbol{\Delta}).
\end{eqnarray*}
\noindent where $X^{\rm exc}$ is the Vervaat transform of an EI bridge with parameters $(\theta_{i}, i \geq 0)$.
\end{theorem}

Note that Theorem \ref{Theo1} holds under more general assumptions than Theorem \ref{Theo2}.  In essence, it requires the convergence of the $\L$ukasiewicz path (recall Theorem \ref{Theo3}), which we can prove under the weaker assumptions (see Theorem \ref{Theo5}). The limiting process $(\mathbf{F}_{X^{\rm exc}}(t), t \geq 0)$ corresponds precisely to the fragmentation process  $(\mathbf{F}(t), t \geq 0)$ in the statement of Corollary \ref{cor:masses}, as we will see in its proof in the next section. We recall that $(\mathbf{F}_{X^{\rm exc}}(t), t \geq 0)$ is the fragmentation process as constructed by Bertoin \cite{Bertoin2000, Bertoin2001}, and that it is also equivalent to the fragmentation process $(\mathbf{F}_{\cT_\theta}(t), t \geq 0)$ of the Inhomogeneous CRT $\cT_\theta$, with parameters $\theta = (\theta_{0}, \theta_{1}, \dots)$, see \cite{AldousPitmanI2000}. 

\begin{proof}[Proof of Theorem \ref{Theo1}]
The assumptions \ref{B1}-\ref{B3} and \ref{A4} in Theorem \ref{Theo1} imply \ref{C1} and \ref{C2} in Theorem \ref{Theo3} with $X$ given by the Vervaat transform of an EI bridge with parameters $(\theta_{i}, i \geq 0)$. Indeed, Lemma \ref{lemma1} and Theorem \ref{Theo5} imply \ref{C1} and \ref{C2}, respectively. 
\cite[Lemma 7]{Bertoin2001} shows \ref{C3} and \ref{C4}. On the other hand, it follows from \cite[Proposition 4.5]{Cyril2019} that there exists two universal constants $c_{1}, c_{2}>0$ such that 
\begin{align}
\mathbb{P}\left( |u_{n}| \geq x V_{n}/b_{n}  \right) \leq c_{1}e^{-c_{2}x}
\end{align}
\noindent uniformly for $x >0$ and all $n$ large enough, where $u_{n}$ is a uniformly random vertex of $\mathbf{t}_{n}$. This implies \ref{C5}. Therefore, Theorem \ref{Theo1} follows from Theorem \ref{Theo3}.
\end{proof}

\subsection{Convergence of the process of masses} \label{sec:proofcor} 

We conclude this section with the proof of Corollary \ref{cor:masses}. The proof makes use of Theorem \ref{Theo1}.

\begin{proof}[Proof of Corollary \ref{cor:masses}]
For all $n \geq 1$ and $1 \leq k \leq E_{n}$, let $\kappa_{n,k}$ be the time at which the $k$-th edge of $\mathbf{t}_{n}$ is removed in the fragmentation process $(V_{n}^{-1}\mathbf{F}_{\mathbf{t}_{n}}(tb_{n}/V_{n}), t \geq 0)$. Then, 
\begin{align*}
\mathbb{M}{\rm ass} [ \bL_{t b_n}(\mathbf{t}_n)] =  \frac{1}{V_n} \mathbf{F}_{\mathbf{t}_{n}}\left( \frac{b_{n}}{V_n} \kappa_{n,\lfloor t b_n \rfloor} \right), \hspace*{2mm} \text{for all} \hspace*{2mm} t \geq 0.
\end{align*}

\noindent Thus, the first claim of Corollary \ref{cor:masses} follows from Theorem \ref{Theo1}, \cite[Theorem 3.9]{Billi1999} and \cite[Theorem 3.1]{Whitt1980} provided that, for each $t \geq 0$,
\begin{equation}
\label{eq:tau}
 \sup_{s \in [0,t]} \left| \kappa_{n, \lfloor s b_n \rfloor} - s \right| \xrightarrow[ ]{\mathbb{P}} 0, \hspace*{3mm} \text{as} \hspace*{2mm} n \rightarrow \infty.
\end{equation}

\noindent This follows as in the proof of Proposition \ref{prop:closeprocesses}.

Finally, we prove the second claim of Corollary \ref{cor:masses}. Following Aldous-Pitman \cite{AldousPitmanI2000}, we recall the construction of the fragmentation process $(\mathbf{F}_{\cT_\theta}(t), t \geq 0)$ associated to the Inhomogeneous CRT $\mathcal{T}_{\theta} = (\mathcal{T}_{\theta}, r_{\theta}, \rho_{\theta}, \mu_{\theta})$ by cutting-down its skeleton through a Poisson point process $\Pi$ of cuts with intensity $\lambda_{\theta} \times dt$ on $\mathcal{T}_{\theta} \times \mathbb{R}_{+}$, where $\lambda_{\theta}$ denotes the length measure of $\mathcal{T}_{\theta}$. For all $t \geq 0$, define an equivalence relation $\sim_{t}$ on $\mathcal{T}_{\theta}$ by saying that $x \sim_{t} y$, for $x, y \in \mathcal{T}_{\theta}$, if and only if no atom of the Poisson process $\Pi$ that has appeared before time $t$ belongs to the geodesic $\llbracket x, y \rrbracket$. These cuts split $\mathcal{T}_{\theta}$ into a continuum forest, which is a countably infinite set of smaller connected components. Let $\mathcal{T}_{\theta,1}^{(t)}, \mathcal{T}_{\theta,2}^{(t)}, \dots$ be the distinct equivalence classes for $\sim_{t}$ (connected components of $\mathcal{T}_{\theta}$), ranked according to the decreasing order of their $\mu_{\theta}$-masses. So, $(\mathbf{F}_{\cT_\theta}(t), t \geq 0)$ is the process given by $\mathbf{F}_{\cT_\theta}(t) = (\mu_{\theta}(\mathcal{T}_{\theta,1}^{(t)}),  \mu_{\theta}(\mathcal{T}_{\theta,2}^{(t)}), \dots)$, for $t \geq 0$ ; in particular, $\mathbf{F}_{\cT_\theta}(0) = (1, 0, 0, \dots)$. 

Consider now that the fragmentation process of $\mathcal{T}_{\theta}$ and its lamination valued-process $(\bL_t(\cT_{\theta}), t \geq 0)$ are constructed from the same Poisson point process $\Pi$. Observe that for any $t \geq 0$, the connected components associated to the fragmentation process of $\mathcal{T}_{\theta}$ at time $t$ are in natural bijection with the faces of $\bL_{t}(\cT_{\theta})$. Let $F^{(t)}$ be a face of $\bL_{t}(\cT_{\theta})$ and $C_{F^{(t)}}$ be the connected component of $\cT_{\theta} \backslash \Pi_t$ coding $F^{(t)}$. Moreover, let $c_{0}$ the unique chord in the boundary of $F^{(t)}$ separating $F^{(t)}$ from $1$ and the other chords $c_{1}, c_{2}, \dots$ bounding $F^{(t)}$, ranked according to the decreasing order of their lengths. Now fix $i \geq 0$. Suppose that $c_{i}$ codes a point $x_i \in \Pi_t$. Then $c_{i}$ splits the unit  circle into two arcs of respective lengths $2\pi \ell_i \coloneqq 2\pi(D(x_i)-G(x_i))$ and $2\pi (1-\ell_i)$ that exactly corresponds to $2\pi$ times the $\mu_\theta$-masses of the two components of $\cT_{\theta} \backslash \{ x_i \}$; see \eqref{LeftRight}. Suppose that $c_{i}$ does not code a point of $\Pi_t$. Then, there exists a sequence of chords $(c_{i}^{k}, k \geq 1)$ in $\bL_{t}(\cT_\theta)$ coding respectively a sequence of points $(x_{i}^{k}, k \geq 1)$ of $\Pi_{t}$ such that $c_{i}^{k} \rightarrow c_{i}$, as $k \rightarrow \infty$, for the Hausdorff distance. In particular, $G(x_i^k)$ and $D(x_i^k)$ converge to some values $G_i, D_i$ such that $c_{i}$ splits the unit circle into two arcs of lengths $2\pi \ell_i \coloneqq  2\pi(D_i-G_i)$ and $2\pi(1-\ell_i)$. On the other hand, $\cT_{\theta} \backslash \bigcup_{k \geq 1} \{x_i^k\}$ also possesses a connected component of $\mu_\theta$-mass $(1-\ell_i)$. Finally, observe that the mass of $F^{(t)}$ and the $\mu_\theta$-mass of $C_{F^{(t)}}$ both can be written as $\ell_0 - \sum_{i \geq 1} \ell_i$. This concludes our proof.
\end{proof}

\paragraph{Acknowledgements.}
We would like to thank Cyril Marzouk and Igor Kortchemski for fruitful discussions about the connection between convergence of trees and convergence of laminations.
We would like to express our gratitude to the referees for their careful and insightful reading of the paper. Their comments led to many improvements, which we believe have made the paper better and more polished.
The second and third authors are supported by the Knut and Alice Wallenberg Foundation, the Ragnar Söderbergs Foundation and the Swedish Research Council.


\providecommand{\bysame}{\leavevmode\hbox to3em{\hrulefill}\thinspace}
\providecommand{\MR}{\relax\ifhmode\unskip\space\fi MR }
\providecommand{\MRhref}[2]{%
  \href{http://www.ams.org/mathscinet-getitem?mr=#1}{#2}
}
\providecommand{\href}[2]{#2}

\end{document}